\documentclass[11pt]{amsart}

\DeclareMathOperator{\Lie}{Lie}

\DeclareMathOperator{\rk}{rk}
\DeclareMathOperator{\ad}{ad}
\DeclareMathOperator{\Der}{Der}
\usepackage{amsmath,amssymb,amsthm,graphics,amscd}
\usepackage{longtable}
\usepackage{cite,xypic}
\usepackage{color}
\usepackage{graphicx}
\usepackage{epsf}
\usepackage{fancyhdr,hyperref}
\usepackage{lscape}
\hypersetup{backref,
colorlinks=true,backref}

  \renewenvironment{thebibliography}[1]{
    \begin{oldthebibliography}{#1}
      \setlength{\parskip}{0ex}
      \setlength{\itemsep}{0ex}
  }
  {
    \end{oldthebibliography}
  }

\setlongtables
\begin{document}

\newcounter{rownum}
\setcounter{rownum}{0}
\newcommand{\ab}{\addtocounter{rownum}{1}\arabic{rownum}}

\newcommand{\x}{$\times$}
\newcommand{\bb}{\mathbf}

\newcommand{\Ind}{\mathrm{Ind}}
\newcommand{\Char}{\mathrm{char}}
\newcommand{\hra}{\hookrightarrow}
\newtheorem{lemma}{Lemma}[section]
\newtheorem{theorem}[lemma]{Theorem}
\newtheorem*{TA}{Theorem A}
\newtheorem*{TB}{Theorem B}
\newtheorem*{TC}{Theorem C}
\newtheorem*{CorC}{Corollary C}
\newtheorem*{TD}{Theorem D}
\newtheorem*{TE}{Theorem E}
\newtheorem*{PF}{Proposition E}
\newtheorem*{C3}{Corollary 3}
\newtheorem*{T4}{Theorem 4}
\newtheorem*{C5}{Corollary 5}
\newtheorem*{C6}{Corollary 6}
\newtheorem*{C7}{Corollary 7}
\newtheorem*{C8}{Corollary 8}
\newtheorem*{claim}{Claim}
\newtheorem{cor}[lemma]{Corollary}
\newtheorem{conjecture}[lemma]{Conjecture}
\newtheorem{prop}[lemma]{Proposition}
\newtheorem{question}[lemma]{Question}
\theoremstyle{definition}
\newtheorem{example}[lemma]{Example}
\newtheorem{examples}[lemma]{Examples}
\theoremstyle{remark}
\newtheorem{remark}[lemma]{Remark}
\newtheorem{remarks}[lemma]{Remarks}
\newtheorem{obs}[lemma]{Observation}
\theoremstyle{definition}
\newtheorem{defn}[lemma]{Definition}

 \def\hal{\unskip\nobreak\hfil\penalty50\hskip10pt\hbox{}\nobreak
 \hfill\vrule height 5pt width 6pt depth 1pt\par\vskip 2mm}

\renewcommand{\labelenumi}{(\roman{enumi})}
\newcommand{\Hom}{\mathrm{Hom}}
\newcommand{\Int}{\mathrm{int}}
\newcommand{\Ext}{\mathrm{Ext}}
\newcommand{\opH}{\mathrm{H}}
\newcommand{\D}{\mathcal{D}}
\newcommand{\soc}{\mathrm{Soc}}
\newcommand{\SO}{\mathrm{SO}}
\newcommand{\Sp}{\mathrm{Sp}}
\newcommand{\SL}{\mathrm{SL}}
\newcommand{\GL}{\mathrm{GL}}
\newcommand{\OO}{\mathcal{O}}
\newcommand{\diag}{\mathrm{diag}}
\newcommand{\End}{\mathrm{End}}
\newcommand{\tr}{\mathrm{tr}}
\newcommand{\Stab}{\mathrm{Stab}}
\newcommand{\red}{\mathrm{red}}
\newcommand{\Aut}{\mathrm{Aut}}
\renewcommand{\H}{\mathcal{H}}
\renewcommand{\u}{\mathfrak{u}}
\newcommand{\Ad}{\mathrm{Ad}}
\newcommand{\N}{\mathcal{N}}
\newcommand{\G}{\mathcal{G}}
\newcommand{\Z}{\mathbb{Z}}
\newcommand{\la}{\langle}\newcommand{\ra}{\rangle}
\newcommand{\gl}{\mathfrak{gl}}
\newcommand{\g}{\mathfrak{g}}
\newcommand{\F}{\mathbb{F}}
\newcommand{\m}{\mathfrak{m}}
\renewcommand{\b}{\mathfrak{b}}
\newcommand{\p}{\mathfrak{p}}
\newcommand{\q}{\mathfrak{q}}
\renewcommand{\l}{\mathfrak{l}}
\newcommand{\del}{\partial}
\newcommand{\h}{\mathfrak{h}}
\renewcommand{\t}{\mathfrak{t}}
\renewcommand{\k}{\mathfrak{k}}
\newcommand{\Gm}{\mathbb{G}_m}
\renewcommand{\c}{\mathfrak{c}}
\renewcommand{\r}{\mathfrak{r}}
\newcommand{\n}{\mathfrak{n}}
\newcommand{\s}{\mathfrak{s}}
\newcommand{\Q}{\mathbb{Q}}
\newcommand{\z}{\mathfrak{z}}
\newcommand{\pso}{\mathfrak{pso}}
\newcommand{\so}{\mathfrak{so}}
\renewcommand{\sl}{\mathfrak{sl}}
\newcommand{\psl}{\mathfrak{psl}}
\renewcommand{\sp}{\mathfrak{sp}}
\newcommand{\Ga}{\mathbb{G}_a}

\newenvironment{changemargin}[1]{%
  \begin{list}{}{%
    \setlength{\topsep}{0pt}%
    \setlength{\topmargin}{#1}%
    \setlength{\listparindent}{\parindent}%
    \setlength{\itemindent}{\parindent}%
    \setlength{\parsep}{\parskip}%
  }%
  \item[]}{\end{list}}

\parindent=0pt
\addtolength{\parskip}{0.5\baselineskip}

\title{A modular analogue of Morozov's theorem on maximal subalgebras of simple Lie algebras}

\author{Alexander Premet}
\thanks{Supported by the Leverhulme Trust (Grant~RPG-2013-293)}
\address{School of Mathematics, The University of Manchester, Oxford Road, M13 9PL, UK} \email{alexander.premet@manchester.ac.uk}
\pagestyle{plain}
\begin{abstract}
Let $G$ be a simple algebraic group over an algebraically closed field of characteristic $p>0$  and suppose that $p$ is a very good prime for $G$. In this paper we prove that any maximal Lie subalgebra $M$ of $\g=\Lie(G)$ with ${\rm rad}(M)\ne 0$ has the form $M=\Lie(P)$ for some maximal parabolic subgroup $P$ of $G$. This means that Morozov's theorem on maximal subalgebras is valid under mild assumptions on $G$. We show that such assumptions are necessary by providing a counterexample to Morozov's theorem for groups of type ${\rm E}_8$ over fields of characteristic $5$. Our proof relies on the main results and methods of the classification theory of finite dimensional simple Lie algebras over fields of prime characteristic.
\end{abstract}
\maketitle
\section{Introduction}
Let $\Bbbk$ be an algebraically closed field of characteristic $p>0$ and let $G$ be a simple algebraic group over $\Bbbk$.
Let $\Pi$ be a basis of simple roots of the root system $\Phi$ of $G$.
Recall that $p$ is said to be {\it very good} for $G$
if $\Phi$ is not of type ${\rm A}_{kp-1}$ for some
$k\in\Z_{>0}$ and all coefficients $n_\alpha$ of the highest root $\widetilde{\alpha}=\sum_{\alpha\in\Pi}n_\alpha\alpha
\in\Phi^+(\Pi)$ are less than $p$.
It is well known that if $p$ is very good for $\g$ then the Lie algebra $\g=\Lie(G)$ is simple and the centraliser of any semisimple element of $\g$ is a Levi subalgebra of $\g$.
The goal of this paper is to show that an analogue of Morozov's theorem on non-semisimple maximal subalgebras holds for the Lie algebra $\g$. Over complex numbers, Morozov's theorem is the starting point of Dynkin's classification of the maximal subalgebras of simple Lie algebras \cite{Dyn} and one hopes that the theorem stated below will play a similar role in the modular setting.
\begin{theorem}\label{thm:main}
Let $G$ be a simple algebraic $\Bbbk$-group, where $p={\rm char}(\Bbbk)$ is a very good prime for $G$, and let $M$ be a maximal Lie subalgebra of $\g=\Lie(G)$ with ${\rm rad} (M)\ne 0$. Then $M=\Lie(P)$ for some maximal parabolic subgroup $P$ of $G$.
\end{theorem}

In good characteristic, it follows from the description of maximal closed subsystems of irreducible root systems that if $P$ is a maximal parabolic subgroup of $G$ then ${\Lie}(P)$ is  a maximal subalgebra of $G$; see Lemma~\ref{regsub} for detail. Thus Theorem~\ref{thm:main} characterises the class of non-semisimple maximal subalgebras of $\g$.

If $G$ is one of the classical groups $\SL(V)$, $\SO(V)$ or $\Sp(V)$ and $p$
is a very good prime for $G$ then Theorem~\ref{thm:main} can be proved very quickly by using Lemma~\ref{nilne0} (which holds when $p$ is very good for $G$) and the reducibility of the action of $M$ on the $G$-module $V$; see
\cite[\S~7]{HS1} for a detailed argument.
Therefore, in what follows we assume that $G$ is an exceptional algebraic group.
No good substitute for $V$ is available in this case and our proof of Theorem~\ref{thm:main}  relies on completely different methods coming from the classification theory of finite dimensional simple Lie algebras over fields of characteristic $p>3$.
We shall see later that if $p$ is not very good for $G$ then Theorem~\ref{thm:main} breaks down very badly for many simple algebraic groups $G$ (classical and exceptional).
For $\g=\sl(V)$ with $p\,| \dim V$ this is due to the fact that any maximal subalgebra of $\g$ acting irreducibly on $V$ is neither semisimple nor parabolic because it has to contain the scalar endomorphisms of $V$.

In characteristic zero, Theorem~\ref{thm:main} is a  classical result of Lie Theory and it has a long history. It was first proved by Morozov in his doctoral dissertation published by Kazan State University in 1943. This text to which Dynkin refers in his seminal article
\cite{Dyn} is hard to find nowadays. Nevertheless,  Panyushev and Vinberg managed to do this and reported that
Morozov's original proof was quite long and relied on case-by-case considerations; see \cite{PV10}. Morozov was probably aware of that because in the short note  \cite{Mor} he has simplified his original arguments.
His shorter (ingenious) proof is easily accessible and is essentially reproduced in \cite[Ch.~VIII, \S~10]{Bour1}. It relies on some results and observations which are no longer valid in prime characteristic.

It should be stressed at this point that a group analogue of Theorem~\ref{thm:main} holds without any restriction on $p$. More precisely,
it was proved by Weisfeiler \cite{W1} and independently by Borel--Tits \cite[Corollaire~3.3]{BT} that if $G$ is a connected reductive algebraic group defined over an algebraically closed field of arbitrary characteristic and $H$ is a maximal subgroup of $G$ then either $H$ is reductive or $H$ is a maximal parabolic subgroup of $G$
(there are a couple of minor glitches in Weisfeiler's argument and it only works over perfect fields).
Although this result implies Morozov's theorem in the characteristic zero case, it cannot be applied for proving Theorem~\ref{thm:main} in full generality. Indeed, a priori it is not clear that
$M=\Lie(H)$ for some Zariski closed subgroup $H$ of $G$ as all  modular substitutes for exponentiation become feeble when $p$ is small.

In the last section of the paper we exhibit two bizarre examples where the equality $M=\Lie(H)$ breaks down in bad characteristic (it would be interesting to determine all such instances). More precisely,  in the case where $G$ is of type ${\rm E}_8$ and $p=5$ we construct a $74$-dimensional maximal Lie subalgebra $\mathfrak{w}$ of $\Lie(G)$  which resides in the middle of a short exact sequence
$$0\to A\to \mathfrak{w}\to W(2;\underline{1})\to 0$$
where $A={\rm nil}(\mathfrak{w})$ is an abelian ideal of $\mathfrak{w}$ isomorphic to $\big(O(2;\underline{1})/\Bbbk 1\big)^*$
as $W(2;\underline{1})$-modules. Here $O(2;\underline{1})=\Bbbk[X,Y]/(X^5,Y^5)$, a truncated polynomial ring in two variables, and $W(2;\underline{1})\,=\,\Der\big(O(2;\underline{1})\big)$, a
Witt--Jacobson Lie algebra. Our second example occurs in the case where $G$ is a group of type ${\rm G}_2$ and $p=2$. It 
is more straightforward and has to do with the fact that in characteristic $2$ the Lie algebra $\Lie(G)$ forgets its identity and becomes isomorphic to $\psl_4$. Several examples of similar nature can be found in other exceptional Lie algebras over fields of bad characteristic. This was recently discovered by  Thomas Purslow who also observed that some {\it very} exotic simple Lie algebras appear as subquotients of exceptional Lie algebras over fields of characteristic $2$ and $3$. One such example in type ${\rm F}_4$ is explicitly described in \cite{Tom}.

Theorem~\ref{thm:main} is the starting point of the joint project with David Stewart which aims to extend Dynkin's classification of maximal subalgebras to the case of reductive Lie algebras over fields of characteristic $p>3$. At the end of the paper we put forward a conjecture on non-semisimple maximal subalgebras of Lie algebras
of type ${\rm E}_8$ over fields of characteristic $5$.

{\bf Acknowledgement.} I would like to thank Jim Humphreys for useful comments on an earlier version of this paper and Tom Purslow for using GAP to verify some crucial properties of the subalgebra
$\mathfrak{w}$. I am very grateful to David Stewart who has read the whole proof, suggested several improvements, pointed out a gap in an earlier version of (\ref{3.20}) and sent me a long list of typos. I am also thankful to Floriana Amicone who spotted a missing case (in type ${\rm F}_4$) in my list of $p$-balanced $G$-orbits of $\g$. I would like to express my appreciation to the anonymous referee for very careful reading and helpful suggestions.
\section{Notation and preliminary results}
\subsection{} Throughout this paper $G$ is an exceptional algebraic $\Bbbk$-group and $p={\rm char}(\Bbbk)$ is a good prime for $G$.
It is well known that in this case the Lie algebra $\g=\Lie(G)$ is simple and its Killing form $\kappa$ is non-degenerate. Being the Lie algebra of an affine algebraic group, $\g$ carries a canonical $[p]$-th power map $\g\ni x\longmapsto x^{[p]}\in\g$ equivariant under the adjoint action of $G$. An element $h\in \g$ is called {\it toral} if $h^{[p]}=h$. The adjoint endomorphism of any toral element of $\g$ is semisimple and its eigenvalues lie in $\F_p$. Therefore, any such element is contained in a maximal toral subalgebra of $\g$. 
Given $x\in\g$ we write $\g_x$ for the centraliser of $x$ in $\g$. If $V$ is a vector space over $\Bbbk$ of dimension $rp$, where $r$ is a positive integer, then we denote by $\mathfrak{psl}(V)$, $\mathfrak{psl}_{rp}$ and $\mathfrak{pgl}_{rp}$
the (restricted) Lie algebras $\sl(V)/\Bbbk\, {\rm Id}_V$, $\sl_{rp}(\Bbbk)/\Bbbk 1$, and $\gl_{rp}(\Bbbk)/\Bbbk 1$, respectively. It is well known (and easily seen) that $\mathfrak{psl}_{rp}=[\mathfrak{pgl}_{rp},\mathfrak{pgl}_{rp}]$ has codimension $1$ in $\mathfrak{pgl}_{rp}$ and $\mathfrak{pgl}_{rp}=\Lie({\rm PGL}_{rp})$. 
If $(p,r)\ne (2,1)$ then 
$\mathfrak{psl}_{rp}\cong \mathfrak{psl}(V)$ is a simple Lie algebra.

Let $\N(\g)$ denote the nilpotent cone of $\g=\Lie(G)$, the set of all $x\in\g$ with $x^{[p]^e}=0$ for $e\gg 0$.
The group $G$ acts on $\N(\g)$ with finitely many orbits which are labelled by their weighted Dynkin diagrams exactly as in the characteristic $0$ case; see \cite{P03}, for example.
The dimensions of nilpotent orbits can be found in \cite[pp.~401--407]{Car}.
Let $\N_p(\g)\,=\,\{x\in \g\,|\,\,x^{[p]}=0\}$, the {\it restricted nullcone} of $\g$.
By \cite{CLNP}, the variety $\N_p(\g)$ coincides with the Zariski closure of a single nilpotent
$G$-orbit; in particular, it is always irreducible.
\subsection{} Let $L$ be a Lie subalgebra of $\g$. The {\it nilradical} of $L$, denoted ${\rm nil}(L)$, is the maximal ideal of $L$ consisting of nilpotent elements of $\g$.
\begin{lemma}\label{nilne0}
If $M$ is a maximal Lie subalgebra of $\g$ with ${\rm rad}(M)\ne 0$ then ${\rm nil}(M)\ne 0$.
\end{lemma}
\begin{proof}
Let $A$ be a maximal abelian ideal of $M$. Since $M$ is
maximal in $\g$, it is a restricted subalgebra of $\g$. Since ${\rm rad}(M)\ne 0$ is a restricted ideal of $M$, the ideal $A\ne 0$ is closed under taking $[p]$-powers in $\g$. If $A\subseteq \N(\g)$
then $0\ne A\subseteq{\rm nil}(M)$ and we are done.
So suppose $A\not\subseteq\N(\g)$. The $A$ must contain a nonzero semisimple element of $\g$, say $t$. Note  that $t$ lies in the span of  all $t^{[p]^i}$ with $i\ge 1$.
As
$$[t^{[p]^i}, M]=(\ad t)^{p^i}(M)\subseteq (\ad t)^2(M)\subseteq [t,A]=0$$ for all $i\ge 1$, this yields $t\in \z(M)$.
Therefore, $M=\g_t$ by the maximality of $M$.
Since $\g$ is simple and $p$ is a good prime for $G$, the centraliser
$\g_t$ is a {\it proper} Levi subalgebra of $\g$. However, such a subalgebra cannot be maximal as it normalises the nilradical of a proper  parabolic subalgebra of $\g$.
This contradiction shows that ${\rm nil}(M)\ne 0$ as stated.
\end{proof}
\subsection{}
A connected reductive $\Bbbk$-group $G$ is called {\it standard} if the derived subgroup of $G$ is simply connected, $p$ is a good prime for $G$, and the Lie algebra $\g=\Lie(G)$ admits a non-degenerate $G$-invariant symmetric bilinear form. Note that if $G$ is standard then so is any Levi subgroup of $G$.
The following lemma is a straightforward generalisation of the main result of \cite{LMT}.
\begin{lemma}\label{borel} Let $G$ be a standard reductive
$\Bbbk$-group and let $L$ be a Lie subalgebra of $\g$ such that $[L,L]$ consists of nilpotent elements of $\g$. Then $L$
is contained in a Borel subalgebra of  $\g$.
\end{lemma}
\begin{proof} Replacing $L$ by its $p$-closure in $\g$ if need be we may assume that $L$ is a restricted subalgebra of $\g$.
We use induction on the dimension of $G$.
Suppose the statement holds for all standard reductive $\Bbbk$-groups of dimension $\le n$ (this is obviously true when $n=1$).
Now suppose $\dim(G)=n+1$.
By Engel's theorem, the Lie algebra $[L,L]$ is nilpotent.
If $[L,L]\ne 0$ then ${\rm nil}(L)\ne 0$.
If $[L,L]=0$ then
$L=L_s\oplus L_n$ where $L_s$ is a toral subalgebra of $\g$ and $L_n={\rm nil}(L)$. If $L=L_s$ then $L\subseteq \Lie(T)$ for some maximal torus
$T$ of $G$; see \cite[Theorem~13.3]{Hum}. So we are done in this case.

Thus we may assume that $\n:={\rm nil}(L)\ne 0$.
Then $\z(\n)\ne 0$. As $[L,L]\subseteq {\rm nil}(L)$ by our assumption on $L$, the adjoint action of $L$ induces a representation of the abelian Lie algebra $L/[L,L]$ on $\z(\n)$.
So there exists a nonzero $e\in \z(\n)$ such that $[L,e]\subseteq \Bbbk e$.
As $e$ is a nilpotent element of $\g$ it admits a cocharacter $\lambda\colon\,\Bbbk^\times\to G$ optimal in the sense of the Kempf--Rousseau theory.
Let $P$ be the parabolic subgroup associated with $\lambda$ and $\p=\Lie(P)$.
Then $\p=\bigoplus_{i\ge 0}\,\g(\lambda,i)$.
By \cite[Theorem~2.3]{P03}, one can choose $\lambda$ in such a way that $e\in\g(\lambda,2)$
and $\g_e\subseteq \p$. Furthermore, $[\g(\lambda,i),e]=\g(\lambda,i+2)$ for all $i\ge 0$.
Taking $i=0$ we find $h_0\in\g(\lambda,0)$ with $[h_0,e]=e$. This implies that $$L\subseteq\n_\g(\Bbbk e)=\Bbbk h_0\oplus\g_e\subseteq \p.$$
By our induction assumption, the image of $L$ in the Levi  subalgebra $\g(\lambda,0)\cong\p/{\rm nil}(\p)$ of $\g$ is contained in a Borel subalgebra of $\g(\lambda,0)$, say $\b$. Since the inverse image of $\b$ under the canonical homomorphism $\p\twoheadrightarrow \g(\lambda,0)$ is a Borel subalgebra of $\g$, this accomplishes the induction step of our proof.
\end{proof}
\subsection{}
We denote by $\OO_{\rm min}$ the minimal nonzero nilpotent orbit in $\g$. It consists of all nonzero $e\in\g$ with the property that $[e,[e,\g]]=\Bbbk e$.
\begin{cor}\label{max-root}
Let $M$ be a maximal Lie subalgebra of $\g$ and denote by $N$ the nilradical of $M$. Suppose $N\ne 0$ and let $R$ be any Lie subalgebra of $M$ whose derived ideal $[R,R]$ consists of nilpotent elements of $\g$. Then
the centraliser $\c_\g(N)$ is an ideal of $M$ and there exists $e\in \c_\g(N)\cap\OO_{\rm min}$ such that $[R,e]\subseteq \Bbbk e$.
\end{cor}
\begin{proof}
Since $N$ is nilpotent we have that $0\ne
\z(N)\subseteq \c_\g(N)$. Therefore, $\c_\g(N)$ is a nonzero Lie subalgebra of $\g$.
If $x\in M$, $c\in \c_\g(N)$ and $n\in N$ then $[[x,c],n]=[x,[c,n]]-[c,[x,n]]=0$. So $[M,\c_\g(N)]\subseteq \c_\g(N)$ which implies that $\widetilde{M}:=M+\c_\g(N)$ is a Lie subalgebra of
$\g$. Since $\g$ is a simple Lie algebra and $\widetilde{M}$ normalises $N$, it must be that $\widetilde{M}\ne \g$. As a result,
$\widetilde{M}=M$ forcing $\c_\g(N)\subseteq M$.

Let $\widetilde{R}=R+N$. Then it is immediate from Jacobson's formula for $p$-th powers that $[\widetilde{R},\widetilde{R}]\subseteq [R,R]+N$ consists of nilpotent elements of $\g$. By Lemma~\ref{borel},
there exists a Borel subalgebra  $\b$ of $\g$ containing $\widetilde{R}$. Let $B=T\cdot R_u(B)$ be the Borel subgroup of $G$ such that $\b=\Lie(B)$. The maximal torus $T$ of $B$ preserves the centre $\z(\n_+)$ of $\n_+:=\Lie(R_u(B))$.
This implies that $\z(\n_+)$ is spanned by root vectors relative to $T$. Since $p$ is a good prime for $G$
and $\n_+$ contains all simple root vectors with respect to $T$,
this yields that $\z(\n_+)=\Bbbk e$ where $e$ is a highest root vector of $\n_+$. It is well known (and easily seen) that the latter belongs to $\OO_{\rm min}$.
As $R\subseteq\widetilde{R}\subseteq \b$ and $\b$ normalises $\z(\n_+)$ it must be that $[R,e]\subseteq \Bbbk e$. Since $N\subseteq \n_+$ we also have that $e\in\c_\g(N)$. This completes the proof.
\end{proof}
\subsection{}
A Lie subalgebra of $\g$ is called {\it regular} if it contains a maximal toral subalgebra of the restricted Lie algebra $\g$.
\begin{lemma}\label{regsub}
Let $M$ be a maximal subalgebra of $\g$ with ${\rm nil}(M)\ne 0$. If $M$ is regular then it is a parabolic subalgebra of $\g$.
\end{lemma}
\begin{proof}
Let $\t$ be a maximal toral subalgebra of $\g$ contained in $M$. It follows from \cite[11.8]{Borel} that there exists a maximal torus $T$ in $G$ such that $\t=\Lie(T)$
(see also \cite[Theorem~13.3]{Hum}).
Since $p>3$, the toral subalgebra $\t$ is a classical Cartan subalgebra of $\g$ in the sense of Seligman and the Lie algebra $\g$ satisfies the Seligman--Mills axioms; see \cite[Ch.~II, \S\,3]{Sel}.
In particular, this means that all root spaces of $\g$ with respect  to $\t$ are $1$-dimensional.
As a consequence, $M$ is $(\Ad\, T)$-stable.

Let $\Phi$ be the root system of $G$ with respect to $T$. Given $\alpha\in\Phi$ we denote by $\g_\alpha$ the root subspace of $\g$ with respect to $T$. Then $\g_\alpha=\Bbbk e_\alpha$ for some root element $e_\alpha\in\g$. The preceding remark entails that there exists a subset
$\Psi$ of $\Phi$ such that $M=\t\oplus\sum_{\alpha\in\Psi}\,\Bbbk e_\alpha$.
Moreover, it follows from \cite[Ch.~II, \S\,4]{Sel}, for example, that if $\alpha,\beta\in\Psi$ and $\alpha+\beta\in\Phi$ then $\alpha+\beta\in\Psi$. In other words, $\Psi$ is a closed subset of $\Phi$ in the sense of \cite[Ch.~ VI, \S\,1, Sect.~7]{Bour}.

Let $\g_\Z$ be the Chevalley $\Z$-form of $\g$ associated with the root system $\Phi$. The above discussion shows that $M$ is obtained by base change from the $\Z$-subalgebra $\g_\Z(\Psi)$ of $\g_\Z$ spanned by all root vectors $E_\alpha\in \g_\Z$ with $\alpha\in\Psi$ and all commutators $[E_\beta,E_{-\beta}]$ with $\beta\in\Phi$. The maximality of $M$ in $\g$ now implies that
$\g_\Z(\Psi)\otimes_\Z\mathbb{C}$ is a maximal subalgebra of $\g_\mathbb{C}:=\g_\Z\otimes_\Z\mathbb{C}$. Then Morozov's theorem
on maximal subalgebras of $\g_\mathbb{C}$ yields that
the closed subset $\Psi$ of $\Phi$ is either  symmetric or parabolic. If $\Psi$ is symmetric then it is not hard to see that the restriction of the Killing form of $\g$ to
$M=\g_\Z(\Psi)\otimes_\Z\Bbbk$ is non-degenerate.
This, however, contradicts our assumption that
${\rm nil}(M)\ne 0$. Therefore, $\Psi$ must be a parabolic subset of $\Phi$. This completes the proof.
\end{proof}
\subsection{}
Given a finite dimensional semisimple Lie algebra $L$ we write $L_p$ for the $p$-closure of $L\cong \ad L$ in $\Der(L)$. This semisimple restricted Lie algebra is often referred to as the {\it minimal} $p$-{\it envelope} of $L$. To ease notation we shall often identify $L$ with $\ad L\subseteq \Der(L)$.
If $L=\bigoplus_{i\in\Z}\,L_i$ is a graded Lie algebra
then $\Der(L)$ has a natural grading, too, and Jacobson's formula for $p$-th powers implies that $L_p$ is spanned by the $p$-closure of $L_0$
in $\Der(L)$ and all $L_i^{p^j}$ with $i\ne 0$ and $j\in\Z_{\ge 0}$.
Recall that an $L$-module $V=\bigoplus_{i\in  \Z}\,V_i$ is called {\it graded} if $L_i\,.\, V_j\subseteq V_{i+j}$ for all $i,j\in\Z$. We say that
$V$ is {\it restrictable} if
there exists a restricted representation $\rho\colon\,L_p\to\gl(V)$ such that
$(\rho(x))(v)=x. v$
for all $x\in L$ and $v\in V$.

\begin{lemma}\label{graded}
Let $L$ be a finite dimensional semisimple graded Lie algebra over $\Bbbk$ and suppose that its zero component $L_0$ is a restricted subalgebra of
$L_p$. Let $V$ be a finite dimensional graded
$L$-module, write $\pi$ for the corresponding representation of $L$, and suppose that
$\pi_{\vert\,L_0}$ is a restricted representation of $L_0$. Then all composition factors of the $L$-module $V$ are graded and restrictable.
\end{lemma}
\begin{proof}
(a) We first show by induction on the composition length $l$ of $V$ that all composition factors of $V$ are graded. If $V$ is irreducible then, of course, there is nothing to prove. So suppose $V$ is reducible and let $L_{\pm }=\bigoplus_{i> 0}\,L_{\pm i}$. These graded Lie subalgebras of $L$ are $(\ad L_0)$-stable and act nilpotently on $V$. It follows that $V_{+}:=\{v\in V\,|\,\,L_{+}\cdot v=0\}$ is a nonzero graded subspace of $V$ and each graded component of $V_+$ is invariant under the action of
$L_0$. Let $r$ be the smallest integer such that $V_{+}\cap V_r\ne 0$ and put $W:=U(L_{-})\cdot (V_+\cap V_r)$. It is straightforward to see that $W$  is a graded subspace of $V$ invariant under the action of $L_{\pm}$ and $L_0$. Our choice of
$r$ then implies that $W$ is an irreducible
$L$-submodule of $V$. As a consequence, the quotient module $V/W$ has a natural structure of a graded $L$-module. Since the length of $V/W$ is smaller than that of $V$ and the composition factors of $V$ are independent of the choice of a composition series (up to isomorphism), the statement follows by induction on $l$.

(b) Next we show that each composition factor of $V$
is restrictable. In view of part~(a) no generality will be lost by assuming that $V$ is an irreducible
$L$-module. Let $\mathcal{L}$ be the $p$-envelope of $L$ in the universal enveloping algebra $U(L)$.
This is an infinite dimensional restricted Lie algebra with an enormous centre
$\z(\mathcal{L})$. The action of $U(L)$ gives $V$
a natural structure of a restricted $\mathcal{L}$-module.
Let $\rho\colon\,\mathcal{L}\to \gl(V)$ denote the corresponding representation of $\mathcal{L}$.
Since $V$ is an irreducible $L$-module,
 $\z(\mathcal{L})$ acts on $V$ by scalar linear operators. Now
define
$\z'(\mathcal{L}):=\z(\mathcal{L})\cap(\ker\rho)$
and put $\tilde{L}_p:=\mathcal{L}/\z'(\mathcal{L})$.
As $L$ is semisimple, $\tilde{L}_p$ is a $p$-envelope of $L$ and its centre $\z=\z(\mathcal{L})/\z'(\mathcal{L})$ has dimension $\le 1$. Moreover, $V$ is a restricted $\tilde{L}_p$-module
and $\mathcal{L}/\z(\mathcal{L})\cong L_p$ as restricted Lie algebras; see \cite[Theorem~1.1.7]{Str04}. 

We write $x\mapsto x^{[p]}$ for the $[p]$-power map of $\mathcal{L}$.
The grading of $L$ induces that on $U(L)$ and its Lie subalgebra $\mathcal{L}=
\sum_{i\in \Z}\sum_{j\ge 0}L_i^{[p]^j}$. We stress that each $L_i^{[p]^j}$ consists of graded elements of degree
$ip^j$. It is straightforward to check that $\z(\mathcal{L})$
is a graded subspace of $\mathcal{L}$. Since $V$ is a graded
$U(\mathcal{L})$-module the subspace $\z'(\mathcal{L})$ must be graded as well. As all graded components $\z(\mathcal{L})_i$ with $i\ne 0$ act on $V$ by nilpotent transformations they  
lie in $\z'(\mathcal{L})$ due to the irreducibility of $V$. This yields that $\z=\Bbbk \rho(z)$ for some $z\in \z(\mathcal{L})_0$. The above description of 
$\z(\mathcal{L})_0\subseteq \sum_{j\ge 0}L_0^{[p]^j}$ then 
shows that $z=z_0+z_1^{[p]}+\cdots+z_k^{[p]^k}$ for some $z_0,z_1,\ldots, z_k\in L_0$. Since $L_0$ is a restricted Lie subalgebra of $L_p$ and $V$ is a restricted $L_0$-module
we have that $\rho(z_i^{[p]^i})=\rho(z_i^{p^i})$ for all $0\le i\le k$ where $x\mapsto x^p$ stands for the $p$-th power map
of $L_p$. As each element $z_i^{[p]^i}-z_i^{p^i}\in\mathcal{L}$ is central, it lies in $\z'(\mathcal{L})$. As $z$ is a central element of $\mathcal{L}$, so too is $z':=\sum_{i=0}^kz^{p^i}$. Since $z'\in L$ and $L$ is semisimple this forces $z'=0$ 
But then $\rho(z)=\rho(z')=0$ implying $\z'(\mathcal{L})=\z(\mathcal{L})$. As a result $\tilde{L}_p\cong L_p$ as restricted Lie algebras. This completes the proof.
\end{proof}
\subsection{} In order to deal with Hamiltonian Lie algebras
which will come to complicate things in the next section
we must look very closely at the $G$-orbits of certain special toral elements of $\g$.
\begin{defn} Let $d$ be a positive integer. A toral element $h\in\g$ is said to be
$d$-{\it balanced} if $\dim \g(h,i)=\dim \g(h,j)$ for all $i,j\in\F_p^\times$ and all eigenspaces $\g(h,i)$ with $i\ne 0$ have dimension divisible by $d$.
\end{defn}
If $h$ is a $d$-balanced toral element of $\g$ then
\begin{equation}\label{killing}
\kappa(h,h)\,=\,{\rm tr}(\ad h)^2\,=\,\textstyle{\sum}_{i=1}^{p-1}\,\dim\g(h,i)\cdot i^2\,=\,\dim\g(h,1)\cdot {(p-1)p(2p-1)}/6\,=\,0.
\end{equation}
\begin{prop}\label{balanced}
Let $G$ be an exceptional algebraic $\Bbbk$-group, where $p={\rm char}(\Bbbk)$ is a good prime for $G$, and let $h$ be a nonzero $p$-balanced toral element of $\g=\Lie(G)$. Then one of the following cases occurs:
\begin{itemize}
\item[(i)\ ] $G$ is of type $\rm E_6$, $p=5$,\,\, the root system of $G_h$ has type ${\rm A}_3$, and $\dim \g_h=18$;

\smallskip

\item[(ii)\,] $G$ is of type ${\rm E}_7$, $p=5$, \
the root system of $G_h$ has type ${\rm D}_4{\rm A}_1$, and $\dim\g_h=33$;

\smallskip

\item[(iii)\,] $G$ is of type ${\rm E}_8$, $p=7$, \
the root system of $G_h$ has type ${\rm E}_6$, and $\dim\g_h=80$;

\smallskip

\item[(iv)\,] $G$ is of type ${\rm E}_8$, $p=7,\,$ \
the root system of $G_h$ has type ${\rm D}_4{\rm A}_2$, and $\dim\g_h=38$;

\smallskip

\item[(v)\,] $G$ is of type ${\rm E}_8$, $p=11,\,$
the root system of $G_h$ has type ${\rm A}_4$, and $\dim\g_h=28$.

\smallskip

\item[(vi)\,] $G$ is of type ${\rm F}_4$, $p=7$,\ \,\,
the root system of $G_h$ has type ${\rm A}_2$, and $\dim\g_h=10$.

\smallskip

\item[(vii)\,] $G$ is of type ${\rm F}_4$, $p=5$,\ \,\,
the root system of $G_h$ has type ${\rm B}_2$, and $\dim\g_h=12$.
\end{itemize}
Furthermore, in all cases except (vi) the nonzero $p$-balanced toral elements
of $\g$ are conjugate under the adjoint action of $G$, whereas in case (vi) there are two $p$-balanced $G$-orbits.
\end{prop}
\begin{proof}
(a) We may assume without loss of generality that $G$ is a group of adjoint type.
Let $\Phi$ be the root system of $G$ with respect to a maximal torus $T$ of $G$ and let $\Pi=\{\alpha_1,\ldots,\alpha_\ell\}$ be a basis of simple roots
of $\Phi$, so that $\ell=\dim T$. Let $\Phi^+$ be the positive system associated with $\Pi$. Any root $\gamma\in\Phi^+$ can be uniquely expressed as   $\gamma=\sum_{i=1}^\ell\,\nu_i(\gamma)\alpha_i$ for some $\nu_i(\gamma)\in\Z_{\ge 0}$. In what follows we always use Bourbaki's numbering of simple roots in $\Pi$; see \cite[Planches~I--IX]{Bour}.

Since $h$ is contained in a maximal toral subalgebra of $\g$ and all such subalgebras are
$(\Ad\, G)$-conjugate, we may assume that $h\in\Lie(T)$; see \cite[11.8]{Borel} of \cite[Theorem~13.3]{Hum}.
Then $\g_h$ is $(\Ad\, T)$-stable. Since $p$ is a good prime for $G$ we may assume further that $\g_h$
is the standard Levi subalgebra of $\g$ associated with a subset $\Pi_0$ of
$\Pi$. Furthermore, there exist $t_1,\ldots, t_\ell\in \t$ such that $({\rm d}\alpha_i)(t_j)=\delta_{i,j}$ for all $1\le i,j\le\ell$. These elements form a $\Bbbk$-basis of $\t$.
As a consequence, the centre of $\g_h$ has a $\Bbbk$-basis consisting
of all $t_i$ with $1\le i\le \ell$ such that $({\rm d}\beta)(t_i)=0$ for all $\beta\in\Pi_0$.

Let $\Psi$ be the root system of $\g_h$ relative to $T$ and put $\Psi^+:=\Psi\cap \Phi^+$. Then $\dim \g_h=\ell+2|\Psi^+|$ and
\begin{equation}\label{equiv}\dim\g-\dim\g_h
\,=\,\textstyle{\sum}_{i\in\F_p^\times}\,\dim\g(h,i)\,\equiv\, 0
\mod p(p-1),\end{equation}
by our assumption on $h$. Hence $|\Phi^+|\equiv |\Psi^+|\mod p(p-1)/2$.

(b) Let $G_{\mathbb C}$ be the complex group of adjoint type with root system $\Phi$. We may assume without loss of generality that $T$
is obtained by base change from a maximal torus $T_{\mathbb C}$ in $G_{\mathbb C}$.
In what follows we are going to rely on the well known analogy between toral elements of $\g$ and elements of order $p$ in $G_{\mathbb C}$; see \cite{Serre}.
This will enable us to use extended Dynkin diagrams and Kac coordinates for labelling toral $G$-orbits in $\g$.

More precisely, let $\zeta$ be a primitive $p$-th root of unity in $\mathbb C$.
Any toral element  of $\t$ is a linear combination of semisimple root vectors
$({ \rm d}\beta^\vee)(1)\in \Lie(\beta^\vee(\Bbbk^\times))\subset T$ with coefficients in $\F_p=\Z/p\Z$. If $i\in\F_p^\times$ then $\g(h, i)$ is spanned by the root spaces
$\g_\gamma$ with $({\rm d}\gamma)(h)=i$.
Let $\g_\Z=\t_\Z\oplus\big(\bigoplus_{\gamma\in\Phi}\, \Z e_\gamma\big)$ be an admissible $\Z$-lattice in $\g_\mathbb{C}$ such that $\t=\t_\Z\otimes_\Z\Bbbk$ and $\g_\gamma=\Bbbk (e_\gamma\otimes1)$ for all $\gamma\in\Phi$.
Since $G_{\mathbb C}$ is a group of adjoint type  there exists a unique $\sigma=\sigma(h)\in T_{\mathbb C}$ with
$\sigma^p=1$ such that $$\g_\mathbb{C}(\sigma,\zeta^i)\,=\,\g_\Z(\sigma, \zeta^i)\otimes_\Z\mathbb{C}\ \ \mbox{ and }\  \ \g(h,i)\,=\,\g_\Z(\sigma, \zeta^i)\otimes_\Z\Bbbk\ \qquad\ \big(\forall\, i\in(\Z/p\Z)^\times\big)$$
where $\g_\mathbb{C}(\sigma, \zeta^i)$ is the $\zeta^i$-eigenspace of $\sigma$ and $\g_\Z(\sigma, \zeta^i)$ is
$\Z$-span of all $e_\gamma$ with $\gamma(\sigma)=\zeta^i$.

Let $P(\Phi^\vee)$ be the weight lattice of the dual root system $\Phi^\vee$ and let  $W=N_G(T)/T\cong N_{G_{\mathbb C}}(T_{\mathbb C})/T_{\mathbb C}$ be the Weyl group of $G$. Since $G_{\mathbb C}$ is a group of adjoint type the lattice of cocharacters $T_{\mathbb C}$ coincides with $P(\Phi^\vee)$ which implies that the $\F_p W$-module of all
elements of order $p$ in $T_{\mathbb C}$ identifies with $P(\Phi^\vee)\otimes_\Z \F_p$. Since $p>3$ the latter space is $W$-equivariantly isomorphic to $\t^{\rm tor}$, the $\F_p$-subspace of all toral elements of $\t$.
From this it follows that there is a natural bijection between the toral $G$-orbits in $\g$ and the conjugacy classes of elements of order $p$ in $G_{\mathbb C}$; see \cite{Serre} for more detail.

Let $\widetilde{\alpha}=\sum_{\alpha\in\Pi} n_\alpha\alpha$
be the highest root in $\Phi^+$ and put $\widetilde{\Pi}:=\Pi\sqcup\{\alpha_0\}$ where $\alpha_0=-\widetilde{\alpha}$.
To any nonzero collection of non-negative integers ${\bf a}=(a_\alpha\,|\,\alpha\in\Pi)$ such that $a_\alpha<p$ for all $\alpha\in\Pi$ and $\sum_{\alpha\in\Pi}a_\alpha n_\alpha\le p$ we attach a unique element $\sigma_{\bf{a}}\in T_{\mathbb C}$ of order $p$ by imposing that $\sigma_{\bf{a}}(e_\alpha)=\zeta^{a_\alpha}e_\alpha$ for all
$\alpha\in \Pi$ (here $e_\gamma$ stands for a root vector of $\g_{\mathbb C}$ corresponding to $\gamma\in\Phi$).
It is well known that any element of order $p$ in $G$ is conjugate to one of the $\sigma_{\bf a}$'s. If $\sigma$ is conjugate to $\sigma_{\bf a}$ then the collection $(a_\alpha\,|\,\alpha\in \widetilde{\Pi})=(a_{\alpha_0},{\bf a})$
with $a_{\alpha_0}=p-\sum_{\alpha\in\Pi} a_\alpha n_\alpha$ is sometimes referred to as {\it Kac coordinates} of $\sigma$. It is known that two different Kac coordinates represent the same conjugacy class in $G_{\mathbb C}$ if and only if one can be obtained from the other by  applying a suitable element from the group ${\rm Stab}_W(\widetilde{\Pi})\cong Z(\widetilde{G}_{\mathbb C})$ (here $\widetilde{G}_{\mathbb C}$ stands for the simple simply connected algebraic group over $\mathbb C$ with root system $\Phi$).
The action of  ${\rm Stab}_W(\widetilde{\Pi})$ on
$\widetilde{\Pi}$
is described in \cite[Planches~I--IX]{Bour} (recall that our numbering of simple roots is compatible with that of {\it loc.\,cit.}).
It is immediate from the above that the set $\{\alpha\in\widetilde{\Pi}\,|\,a_\alpha=0\}$ forms a basis of simple roots of the fixed-point algebra
$\g_{\mathbb C}^{\sigma_{\bf a}}=\{x\in\g_{\mathbb C}\,|\,\, \sigma_{\bf a}(x)=x\}$. To ease notation we set $a_i:=a_{\alpha_i}$ for $0\le i\le \ell$.

(c) Suppose for a moment that $G$ is an arbitrary standard connected reductive group over $\Bbbk$. Recall that a nilpotent $G$-orbit in $\g=\Lie(G)$ is called {\it rigid} if it cannot be
obtained by Lusztig--Spaltenstein induction from a proper Levi subalgebra of $\g$. In order to reduce the number of cases that we have to investigate closely we shall rely on the the theory of sheets in $\g$ as presented in \cite{PSt}. Put $\l:=\g_h$ and let $L$ be the standard Levi subgroup of $G$ with $\l=\Lie(L)$. Write $L_{\mathbb C}$ for the standard Levi subgroup of $G_{\mathbb C}$
associated with the subset $\Pi_0$ of $\Pi$.
We adopt the notation introduced in \cite{PSt} and, in particular, write $\z(\l)_{\rm reg}$ for the set of all
elements $z$ in the centre $\z(\l)$ for which
$\g_z=\l$. This is a nonempty Zariski open subset of $\z(\l)$.
Since $\{0\}$ is a rigid nilpotent orbit
in $\l$ there exists a unique sheet $\mathcal{S}$ of $\g$ whose open decomposition class coincides with $\mathcal{D}(\l,0)=
(\Ad\, G)\cdot\z(\l)_{\rm reg}$; see \cite[Theorem~2.8]{PSt}.
Furthermore, in view of \cite[2.5]{PSt} the unique nilpotent orbit contained in $\mathcal{S}$ coincides with $\OO_0:={\rm Ind}_\l^\g\,\{0\}$, the Richardson orbit associated with the Levi subalgebra $\l$. It follows from \cite[Theorem~1.4]{PSt} that the Dynkin label of
$\OO_0$ coincides with that of the
Richardson orbit in $\g_{\mathbb C}$ associated with $\l_{\mathbb C}$. The latter can be read off from the tables in \cite{dG-E}.

Since $h$ is a semisimple element of $\g$ the adjoint $G$-orbit $\OO$ of $h$ is Zariski closed in $\g$. Let $\mathbb{K}\OO$ denote the cone associated with $\OO$, a $G$-stable, Zariski closed, conical subset of $\g$; see \cite[5.1]{PSk}, for example. It is well known that all irreducible components of $\mathbb{K}\OO$ have dimension equal to that of $\OO$.
Since $x^{[p]}-x=0$ for all $x\in\OO$ and
the $p$-th power map of $\g$ is a morphism given by
a collection of homogeneous polynomial functions of degree $p$ on $\g$ we have the inclusion $\Bbb{K}\OO\subseteq \N_p(\g)$. On the other hand,
it follows from \cite[Lemma~5.1]{PSk} that
$\Bbb{K}\OO\,=\,\overline{\Bbbk^\times\OO}\setminus \Bbbk^\times\OO$.
Since $\Bbbk^\times\OO\subseteq \mathcal{D}(\l,0)\subset \mathcal{S}$ and $\mathcal{S}$ contains a unique nilpotent orbit, we now deduce that
$\Bbb{K}\OO$ coincides with the Zariski closure of
the Richardson orbit $\OO_0$. As a consequence,
$\OO_0\subseteq \N_p(\g)$. We let $e$ be any element of $\OO_0$. Since $p$ is a good prime for $G$ we have that
$\dim \g_e=\dim G_e=\dim G_h=\dim \g_h$. In view of (\ref{equiv}) this shows that  $\dim \OO_0$ is divisible by $p(p-1)$.

(d) Suppose $G$ is of type ${\rm E}_6$ and let $h$, $L$, $\l$, $\Psi$, $\OO_0$, $e$ be as above. If
$p\ge 11$ then $p(p-1)>78=\dim \g$. So $\g$ cannot contain orbits of dimension divisible by $p(p-1)$.

Suppose $p=7$. Then it must be that $\dim \OO_0=42$. There is only one nilpotent orbit of dimension $42$ in $\g$ and its Dynkin label is ${\rm A}_2$. In view of
\cite[p.~267]{dG-E} this yields that $\Psi$ has type ${\rm A}_5$. But then $\l=\z(\l)\oplus [\l,\l]$ and $\z(\l)=\Bbbk h$. Since the restriction of $\kappa$ to $\l$ is non-degenerate and $\z(\l)\subset [\l,\l]^\perp$, this forces $\kappa(h,h)\ne 0$. Since this contradicts (\ref{killing}), the case $p=7$ cannot occur.

Suppose $p=5$. The $p(p-1)=20$. As there are no nilpotent orbits of dimension $20$ in $\g$ we have that $\dim \OO_0\in\{40, 60\}$. The only nilpotent orbit of dimension $40$ in $\g$ has Dynkin label $3{\rm A}_1$.
Since this orbit is rigid by \cite[Theorem~1.3]{PSt}, it must be that $\dim \OO_0=60$. There are two nilpotent orbits of dimension $60$ in $\g$ and their Dynkin labels are
${\rm A}_4$ and ${\rm D}_4$. Both orbits are Richardson by \cite[p.~267]{dG-E}. However, if $e\in\OO({\rm D}_4)$ then $e^{[5]}\ne 0$ by \cite[Theorem~35]{McN03} and \cite[Table~6]{Law}. So it must be that $e\in\OO({\rm A}_4)$. Then \cite[p.~267]{dG-E} shows that
$\Psi$ has type ${\rm A}_3$.
Let $\sigma=\sigma(h)\in G_{\mathbb C}$ be the element of order $p$ attached to $h$ in part~(b) and let $\sigma_{\bf a}\in T_{\mathbb C}$ be a canonical representative of the conjugacy class of $\sigma$. Our discussion at the end of part~(b) shows that
the set $\{\alpha\in\widetilde{\Pi}\,|\,\,a_\alpha=0\}$
forms a subdiagram of type ${\rm A}_3$ of the extended Dynkin diagram  $\widetilde{{\rm E}}_6$.
Since $a_0=5-a_1-2a_2-2a_3-3a_4-2a_5-a_6\ge 0$ and $a_i\in \Z_{\ge 0}$ for all $i$, we essentially have only one option here, namely,
\begin{equation}\label{e6}(a_0,a_1,a_2, a_3, a_4, a_5, a_6)\,=\,(1,1,1,0,0,0,1).\end{equation}
The other two options that we have can be obtained from this one
by applying a suitable symmetry of $\widetilde{\Pi}$ coming from the group ${\rm Stab}_W(\widetilde{\Pi})\cong \Z/3\Z$. The uniqueness of $\sigma_{\bf a}$ implies that all powers $\sigma^i$  with $1\le i\le 4$ are $G$-conjugate forcing
$\dim_{\mathbb C}\,\g_{\mathbb C}(\sigma, \zeta^i)=\dim_{\mathbb C}\,\g_{\mathbb C}(\sigma, \zeta^j)=15$ for all $1\le i,j\le 4$. Our discussion at the beginning of part~(b) now shows that in characteristic $5$ the Lie algebra $\g$ admits a unique $G$-orbit consisting of $p$-balanced toral elements. This $p$-balanced orbit is listed in part~(i) of the proposition.

(e) Suppose $G$ is of type ${\rm E}_7$. If
$p\ge 13$ then $p(p-1)>133=\dim \g$. So $\g$ cannot contain orbits of dimension divisible by $p(p-1)$.

If $p=11$ then $\dim \OO_0=110$ and $\dim \g(h,i)=11$ for all $i\in\F_p^\times$. There are two nilpotent orbits of dimension $110$ in $\g$, but only one of them is Richardson; see \cite[p.~269]{dG-E}. So $e$ must have Dynkin label ${\rm E}_6({\rm a}_3)$ and therefore $\Psi$ must be of type ${\rm A}_3{\rm A}_1^2$. Replacing $h$ by $w(h)$ for some $w\in W$ if necessary we may assume that
$h=\lambda t_3+\mu t_6$ for some $\lambda,\mu\in \F_p^\times$. Set $$\Phi^+_{3,6}\,:=\,\{\gamma\in\Phi^+\,|\,\,
\nu_3(\gamma)=1 \mbox{ and }\,\, \nu_6(\gamma)=0\}.$$
It is straightforward to check that $|\Phi_{3,6}^+|=12$. Since the $\Bbbk$-span of the root vectors $e_\gamma$ with $\gamma\in\Phi_{3,6}^+$ is contained in a single
eigenspace $\g(h,k)$ with $k\in\F_p^\times$ we see
that $h$ cannot be $p$-balanced. So the case where $p=11$ cannot occur.

Suppose $p=7$. Then $\dim\OO_0\in\{42, 84, 126\}$.
If $\dim \OO_0=126$ then $\OO_0=\OO_{\rm reg}$, the regular nilpotent orbit of $\g$. Since $\OO_{\rm reg}\not\subset\N_p(\g)$ by \cite[Theorem~35]{McN03} and \cite[Table~8]{Law}, this case cannot occur.
As $\N(\g)$ has no orbits of dimension $42$  it must be that $\dim\OO=84$ and $\dim\g(h,i)=14$ for all $i\in\F_p^\times$. There are three orbits of this dimension in $\N(\g)$, but only two of them are Richardson; see \cite[p.~270]{dG-E}. Their Dynkin labels are
$2{\rm A}_2$ and ${\rm A}_2+3{\rm A}_1$, In the first case $\Psi$ has type ${\rm D}_5A_1$ and hence $\z(\l)=\Bbbk h$. Since $p\ne 2$ we have a direct sum decomposition $\l=\z(\l)\oplus [\l,\l]$. Arguing as at the beginning of part~(d) it is easy to see that
$\kappa(h,h)\ne 0$. Since this contradicts (\ref{killing}) we deduce that the present case is impossible. If $\OO_0=\OO({\rm A}_2+3{\rm A}_1)$ then $\Psi$ has type ${\rm A}_6$. In this case we also have the equality $\z(\l)=\Bbbk h$, but cannot argue as before because $ [\l,\l]\cong \sl_7$ contains $\z(\l)$. Instead, we observe that $h=\lambda
t_2$ for some $\lambda\in \F_p^\times$. Since the set
$\{\gamma\in\Phi^+\,|\,\,
\nu_2(\gamma)=1\}$ has cardinality $35>14$ we see that $h$ cannot be $p$-balanced in $\g$. Therefore, $p\ne 7$.

Suppose $p=5$. Then $\dim\OO_0\in\{20, 40, 60, 80,100, 120\}$. Since there are no orbits of dimension $20$, $40$, $60$ or $80$ in $\N(\g)$ it must be that $\dim \OO_0=100$ or $\dim\OO_0=120$. If $\dim \OO_0=120$ then
$\OO_0=\OO({\rm E}_7({\rm a}_3))$ or $\OO_0=\OO({\rm E}_6)$. Since neither of these orbits lies in $\N_p(\g)$ by \cite[Theorem~35]{McN03} and \cite[Table~8]{Law}, we obtain that $\dim \OO_0=100$ and $\dim \g(h,i)=20$ for all $i\in\F_p^\times$. There are two orbits of that dimension in $\g$ and their Dynkin labels are
${\rm A}_4$ and ${\rm A}_3+{\rm A}_2+{\rm A}_1$.
If $\OO_0=\OO({\rm A}_3+{\rm A}_2+{\rm A}_1)$ then \cite[p.~269]{dG-E} yields that $\Psi$ has type ${\rm A}_4{\rm A}_2$ and $\z(\l)=\Bbbk t_5$. As a consequence, $h=\lambda t_5$ for some $\lambda\in\F_p^\times$. Since the set $\{\gamma\in\Phi^+\,|\,\,\nu_5(\gamma)=1\}$ has cardinality $30>20$, the element $h$ cannot be
$p$-balanced in the present case.
Now suppose $\OO_0=\OO({\rm A}_4)$. Then \cite[p.~267]{dG-E} shows that $\Phi$ is of type
${\rm D}_4{\rm A}_1$.
Let $\sigma=\sigma(h)\in G_{\mathbb C}$ be the element of order $p$ attached to $h$ in part~(b) and let $\sigma_{\bf a}\in T_{\mathbb C}$ be a canonical representative of the conjugacy class of $\sigma$. Our discussion at the end of part~(b) shows that
the set $\{\alpha\in\widetilde{\Pi}\,|\,\,a_\alpha=0\}$
forms a subdiagram of type ${\rm D}_4{\rm A}_1$ on the extended Dynkin diagram $\widetilde{{\rm E}}_7$.
Since $a_0=5-2a_1-2a_2-3a_3-4a_4-3a_5-2a_6-a_7\ge 0$ and $a_i\in \Z_{\ge 0}$ for all $i$, we essentially have only one option here, namely,
\begin{equation}\label{e7}(a_0,a_1,a_2, a_3, a_4, a_5, a_6,a_7)\,=\,(1,1,0,0,0,0,1,0).\end{equation}
The other option that we have can be obtained
by applying a suitable symmetry of $\widetilde{\Pi}$ coming from the group ${\rm Stab}_W(\widetilde{\Pi})\cong \Z/2\Z$. As at the end of part~(d), the uniqueness of $\sigma_{\bf a}$ implies that all powers $\sigma^i$  with $1\le i\le 4$ are $G$-conjugate. Then
$\dim_{\mathbb C}\,\g_{\mathbb C}(\sigma, \zeta^i)=\dim_{\mathbb C}\,\g_{\mathbb C}(\sigma, \zeta^j)=20$ for all $1\le i,j\le 4$ (of course, one can also check this by a direct computation). Our discussion at the beginning of part~(b) now shows that in characteristic $5$ the Lie algebra $\g$ admits a unique $G$-orbit consisting of $p$-balanced toral elements. This proves statement~(ii).

(f) Suppose $G$ is of type ${\rm E}_8$. If $p\ge 17$ then $p(p-1)>248$, hence $\g$ does not have orbits of dimension divisible by $p(p-1)$. If $p=13$ then it must be that $\dim \OO_0=156$. There is only one orbit of dimension $156$ in $\N(\g)$ and its Dynkin label is $2{\rm A}_2$. Then \cite[p.~275]{dG-E} yields that
$\Psi$ has type ${\rm D}_7$ and $\z(\l)=\Bbbk t_1$. Since $\l=\z(\l)\oplus [\l,\l]$ we then have
$\kappa(h,h)\ne 0$ contrary to (\ref{killing}). Therefore, $p\le 11$.

Suppose $p=11$. Then $\dim\OO_0=220$ and $\dim \g(h,i)=22$ for all $i\in\F_p^\times$ because $\N(\g)$ has no orbits of dimension $110$. There are two orbits of dimension $220$ in $\N(\g)$ and their Dynkin labels are ${\rm E}_7({\rm a}_3)$ and  ${\rm E}_8({\rm b}_6)$.
Both orbits are Richardson and it follows from \cite[p.~272]{dG-E} that in the second case
$\Psi$ has type ${\rm A}_3{\rm A}_2{\rm A}_1$ and $h=\lambda t_4+\mu t_8$ for some $\lambda,\mu\in\F_p^\times$. Set $$\Phi^+_{4,8}\,:=\,\{\gamma\in\Phi^+\,|\,\,
\nu_4(\gamma)=1 \mbox{ and }\,\, \nu_8(\gamma)=0\}.$$
It is easy to see that $|\Phi_{4,8}^+|=24>22$. Since the $\Bbbk$-span of the root vectors $e_\gamma$ with $\gamma\in\Phi_{4,8}^+$ is contained in a single
eigenspace of $\ad h$ we deduce that the case where  $\OO_0=\OO({\rm E}_8({\rm b}_6))$ cannot occur.
If $\OO_0=\OO({\rm E}_7({\rm a}_3))$ then
\cite[p.~272]{dG-E} shows that $\Psi$ has type ${\rm A}_4$.
Let $\sigma=\sigma(h)\in G_{\mathbb C}$ be the element of order $p$ attached to $h$ in part~(b) and let $\sigma_{\bf a}\in T_{\mathbb C}$ be a canonical representative of the conjugacy class of $\sigma$. We know from part~(b) that the set $\{\alpha\in\widetilde{\Pi}\,|\,\,a_\alpha=0\}$
forms a subdiagram of type ${\rm A}_4$ of the extended Dynkin diagram $\widetilde{\rm E}_8$.
Since $a_0=11-2a_1-3a_2-4a_3-6a_4-5a_5-4a_6-3a_7-2a_8\ge 0$ and $a_i\in \Z_{\ge 0}$ for all $i$, we have only one option here, namely,
\begin{equation}\label{e8-11}(a_0,a_1,a_2, a_3, a_4, a_5, a_6,a_7,a_8)\,=\,(1,1,1,0,0,0,0,1,1).\end{equation}
The uniqueness of $\sigma_{\bf a}$ implies that all powers $\sigma^i$  with $1\le i\le 10$ are $G$-conjugate implying that
$\dim_{\mathbb C}\,\g_{\mathbb C}(\sigma, \zeta^i)=\dim_{\mathbb C}\,\g_{\mathbb C}(\sigma, \zeta^j)=22$ for all $1\le i,j\le 10$. Arguing as before we now conclude that in characteristic $11$ the Lie algebra $\g$ admits a unique $G$-orbit consisting of $p$-balanced toral elements. This proves~(v).

Suppose $p=7$. Since $\N(\g)$ has no orbits of dimension $42$, $84$ and $126$, either $\dim\OO_0=168$ or $\dim\OO_0=210$. There are two orbits of dimension $168$ in $\N(\g)$, but one of them is rigid by
\cite[Theorem~1.3]{PSt}. So if $\dim\OO_0=168$ then $\dim\g(h,i)=28$ for all $i\in\F_p^\times$ and $\OO_0=\OO({\rm D}_4)$. In that case \cite[p.~275]{dG-E} yields that $\Psi$ has type ${\rm E}_6$ and $h=\lambda t_7+\mu t_8$ for some $\lambda,\mu\in\F_p^\times$.
Let $\sigma=\sigma(h)\in G_{\mathbb C}$ be the element of order $p$ attached to $h$ in part~(b) and let $\sigma_{\bf a}\in T_{\mathbb C}$ be a canonical representative of the conjugacy class of $\sigma$.  From part~(b) we know that the set $\{\alpha\in\widetilde{\Pi}\,|\,\,a_\alpha=0\}$
forms a subdiagram of type ${\rm E}_6$ of the extended Dynkin diagram $\widetilde{{\rm E}}_8$.
Since $a_0=7-2a_1-3a_2-4a_3-6a_4-5a_5-4a_6-3a_7-2a_8\ge 0$ and $a_i\in \Z_{\ge 0}$ for all $i$, we have only one option here, namely,
\begin{equation}\label{e8-7-1}(a_0,a_1,a_2, a_3, a_4, a_5, a_6,a_7,a_8)\,=\,(2,0,0,0,0,0,0,1,1).\end{equation}
The uniqueness of $\sigma_{\bf a}$ implies that all $\sigma^i$  with $1\le i\le 6$ are $G$-conjugate forcing
$\dim_{\mathbb C}\,\g_{\mathbb C}(\sigma, \zeta^i)=\dim_{\mathbb C}\,\g_{\mathbb C}(\sigma, \zeta^j)=28$ for all $1\le i,j\le 6$. Therefore, in the present case $h$ is a $p$-balanced toral element of $\g$. Statement~(iv) follows.

It remains to consider the case where $p=7$ and $\dim\OO_0=210$.
There are two orbits of that dimension in $\N(\g)$ and their Dynkin labels are ${\rm A}_6$ and ${\rm D}_4({\rm a}_1)$. Both orbits  are Richardson by [dGE, p.~273].
As $\OO_0=\OO({\rm D}_4({\rm a}_1))\not\subset \N_p(\g)$ by \cite[Theorem~35]{McN03} and \cite[Table~9]{Law}
it must be that $\OO_0=\OO({\rm A}_6)$. Then \cite[p.~273]{dG-E} in conjunction with \cite[Theorem~1.4]{PSt} shows that $\Psi$ is of type ${\rm D}_4{\rm A}_2$ and
$h=\lambda t_1+\mu t_6$ for some $\lambda,\mu\in\F_p^\times$. Let  $\sigma_{\bf a}\in T_{\mathbb C}$ be a canonical representative of the conjugacy class of $\sigma(h)$.  By part~(b), the set $\{\alpha\in\widetilde{\Pi}\,|\,\,a_\alpha=0\}$
forms a subdiagram of type ${\rm D}_4{\rm A}_2$ of the extended Dynkin diagram $\widetilde{{\rm E}}_8$.
Since $a_0=7-2a_1-3a_2-4a_3-6a_4-5a_5-4a_6-3a_7-2a_8\ge 0$ and $a_i\in \Z_{\ge 0}$ for all $i$, the only  option we have here is
\begin{equation}\label{e8-7-2}(a_0,a_1,a_2, a_3, a_4, a_5, a_6,a_7,a_8)\,=\,(1,1,0,0,0,0,1,0,0).\end{equation}
Once again the uniqueness of $\sigma_{\bf a}$ implies that all powers $\sigma^i$  with $1\le i\le 6$ are 
$G_{\mathbb C}$-conjugate. Hence
$\dim_{\mathbb C}\,\g_{\mathbb C}(\sigma, \zeta^i)=\dim\g(h, i)=35$ for all $i\in \F_p^\times$. Therefore, $h$ is a $p$-balanced toral element of $\g$, proving (v).

(g) Suppose $G$ is of type ${\rm F}_4$. If $p>7$ then
$p(p-1)>52=\dim\g$, hence $\g$ has no orbits of dimension divisible by $p(p-1)$. If $p=7$ then it must be that $\dim\OO_0=42$ and $\dim\g(h,i)=7$ for all $i\in \F_p^\times$. There are two orbits of dimension $42$ in $\N(\g)$ and their Dynkin labels are ${\rm B}_3$ and ${\rm C}_3$. By \cite[p.~276]{dG-E} (which is applicable in view of
\cite[Theorem~1.4]{PSt}), both orbits are Richardson and  $\Psi$ has type ${\rm A}_2$ in both cases (in the ${\rm B}_3$-case $\Psi$ consists of short roots whereas in the ${\rm C}_3$-case all roots in $\Psi$ are
long).
In both cases there exists a unique collection ${\bf a}=(a_0,a_1,a_2,a_3,a_4)$ with $a_i\in\Z_{\ge 0}$ and
$a_0=7-2a_a-3a_2-4a_3-2a_4$ such that  the set $\{\alpha\in\widetilde{\Pi}\,|\,\,a_\alpha=0\}$ forms a basis of $\Psi$. Specifically, in the ${\rm B}_3$-case we take
${\bf a}=(2,1,1,0,0)$ whilst in the ${\rm C}_3$-case we take ${\bf a}=(1,0,0,1,1)$ (these are the only options available). Arguing as before it is now easy to observe that each of these two cases gives rise to a unique $p$-balanced conjugacy class in $\g$. This proves (vii).

Suppose $p=5$. Since there are no orbits of dimension $20$ in $\g$ and there exists only one orbit of  dimension $40$ in $\N(\g)$
it must be that $\OO_0=\OO({\rm F}_4({\rm a}_3))$.
Thanks to \cite[p.~276]{dG-E} this yields that $[\l,\l]$ has type ${\rm A}_1{\rm A}_2$ or ${\rm B}_2$. In the first case $h$ spans the centre of $\l$. Since $p=5$ we have that
$\l=\z(\l)\oplus[\l,\l]$ which forces $\kappa(h,h)\ne 0$ contrary to (\ref{killing}). Therefore, this case cannot occur. In the second case, Kac coordinates for the $G$-orbit of $h$ should verify $a_2=a_3=0$, $a_i>0$ for $i\in\{0,1,4\}$, and $ a_0+2a_1+2a_4=5$. Hence
\begin{equation}\label{f4-5-1} (a_0,a_1,a_2,a_3,a_4)=(1,1,0,0,1).\end{equation} As before, the uniqueness of ${\bf a}$ implies that
all elements $\sigma(h)^i$ with $1\le i\le 4$ are $G_{\mathbb C}$-conjugate. Consequently, $h$ is $p$-balanced in $\g$.

(h) If $G$ is of type ${\rm G}_2$ then $p(p-1)>\dim\g=14$ for any good prime $p$. This means that $\g$ does not admit any $p$-balanced toral elements.
This completes the proof of the proposition.
\end{proof}
\subsection{}\label{2.8} In the final stages of the
proof of Theorem~\ref{thm:main} we shall require
some information on the gradings of non-restricted Hamiltonian algebras $H(2;\underline{n};\Phi)^{(1)}$.

Given $m\in\Z_{\ge 1}$ we denote by $O(\!(m)\!)$ the full divided power algebra
in $m$ variables over $\Bbbk$. Its elements are the infinite series  $\sum_{\bf r}\,\lambda_{\bf r} x^{\bf r}$ with $\lambda_{\bf r}\in \Bbbk$. Here $x^{\bf r}= x_1^{(r_1)}\cdots x_m^{(r_m)}$, where
$r_i\in\Z_{\ge 0}$, and the product is induced by the rule
$$x^{\bf a}\cdot x^{\bf b}\,=\,
\Big(\textstyle{\prod_{i=1}^m{{a_i+b_i}\choose{a_i}} \Big)}\cdot x^{{\bf a}+{\bf b}}.$$
The maximal ideal $\m$ of the  ring $O(\!(m)\!)$ is equipped with a system of divided powers $f\mapsto f^{(k)}\in O(\!(m)\!)$ where $k\in\Z_{\ge 0}$; see \cite[2.1]{Str04} for more detail.
Given $f\in\m$ we set
$$
\exp(f):=\textstyle{\sum}_{k\ge 0}\,f^{(k)}.
$$
This element is well-defined in the linearly compact ring
$O(\!(m)\!)$ and using the axioms of divided powers it is straightforward to check that $\exp(f)^{-1}=\exp(-f).$

A continuous derivation $D$ of the topological algebra $O(\!(m)\!)$ is called {\it special}
if $D(f^{(k)})=f^{(k-1)}D(f)$ for all $f\in \m$ and
$k>0$. The special derivations of $O(\!(m)\!)$
form a Lie subalgebra of $\Der\big(O(\!(m)\!)\big)$ denoted $W(\!(m)\!)$. It is well known that this algebra
is a free $O(\!(m)\!)$-module of rank $m$  with a free basis consisting of the special partial derivatives
$\partial_1,\ldots,\partial_m$. Recall that
$\partial_i(x^{\bf r})=x^{{\bf r}-\epsilon_i}$ where
$\epsilon_i=(\delta_{i,1},\ldots, \delta_{i,m})$ and $1\le i\le m$. We denote by ${\rm Aut}_c\, O(\!(m)\!)$
the group of all continuous automorphisms $g$ of the topological $\Bbbk$-algebra $O(\!(m)\!)$ such that $g(f^{(k)})=
(g(f))^{(k)}$ for all $f\in \m$ and
$k>0$.

If $\underline{n}=(n_1,\ldots,n_m)\in \Z_{\ge 1}^m$
then the $\Bbbk$-span $O(m;\underline{n})$ of
all $x^{\bf r}$ with $0\le r_i\le p^{n_i}-1$ is a subalgebra of dimension $p^{n_1+\cdots+n_m}$ in $O(\!(m)\!)$ invariant under all special partial derivatives $\partial_i$ with $1\le i\le m$. The {\it general Cartan type Lie algebra} $W(m;\underline{n})$ is the normaliser of $O(m;\underline{n})$ in  $W(\!(m)\!)$. It is a free $O(m;\underline{n})$-module with basis $\partial_1,\ldots,\partial_m$.

We denote by $\Omega(\!(m)\!)=\bigoplus_{i=0}^m\,\Omega^i(\!(m)\!)$ the module of K{\"a}hler differentials over $O(\!(m)\!)$.
It is known that any Hamiltonian  algebra
$H(2;\underline{n};\Phi)^{(1)}\subset W(2;\underline{n})$ has Cartan type $S$ and stabilises the volume form
$J(\Phi) \omega_S\in\Omega^2(\!(2)\!)$ where $\omega_S={\rm d}x_1\wedge{\rm d}x_2$. Here $J(\Phi)\in O(\!(2)\!)^\times$ is the Jacobian of the admissible automorphism
$\Phi\in {\rm Aut}_c\,O(\!(2)\!)$; see \cite[p.~301]{Str04}. Moreover, thanks to \cite[Theorem~6.3.8]{Str04} we may assume that $\Phi=\Phi(\tau)$ or
$\Phi=\Phi(l)$ where
$\tau=\tau(\underline{n})=(p^{n_1}-1,p^{n_2}-1)$ and $l=1,2$. Recall that
$J\big(\Phi(\tau)\big)=1+x^{\tau(\underline{n})}$ and
$J\big(\Phi(l)\big)=\exp\big(x_l^{(p^{n_l})}\big)$; see \cite[p.~309]{Str04}.
\begin{lemma}\label{non-grad}
Let $S=H(2;\underline{n};\Phi)^{(1)}$ where $\Phi$ is one of $\Phi(\tau)$, $\Phi(1)$, $\Phi(2)$. Then all maximal tori of the algebraic group ${\rm Aut}(S)$ are $1$-dimensional.
\end{lemma}
\begin{proof} Let ${\bf G}={\rm Aut}(S)$ and denote by $\tilde{\bf G}$ the normaliser of $O(2;\underline{n})$ in
${\rm Aut}_c\,O(\!(2)\!)$.
It follows from \cite[Theorem~10.8]{Skr91} that the group $\bf G$ consists of all $g\in \tilde{\bf G}$ such that
$g\big(J(\Phi) \omega_S\big)\in \Bbbk^\times J(\Phi) \omega_S$ (see also \cite[Theorems~7.3.2]{Str04}).
Any element $g\in \tilde{\bf G}$ is uniquely determined by its effect on the divided power generators $x_1,\,x_2 \in O(\!(2)\!))$
and $g(x_1),\,g(x_2)\in O(2;\underline{n})$. From this it is immediate that $\tilde{\bf G}/R_u(\tilde{\bf G})$ is isomorphic to a parabolic subgroup of $\GL_2(\Bbbk)$.
It follows that all maximal tori of $\tilde{\bf G}$ have dimension $2$. One such torus, $T_0$, consists of all
automorphisms $g(t_1,t_2)$ with $t_1,\,t_2\in\Bbbk^\times$ such that
$\big(g(t_1,t_2)\big)(x_i)=t_ix_i$ for $i=1,2$.

Suppose ${\bf G}$ contains a $2$-dimensional torus.
Then it is a maximal torus of $\tilde{\bf G}$ and hence is $\tilde{\bf G}$-conjugate to $T_0$. Replacing $\Phi\in{\rm Aut}_c\,O(\!(2)\!)$ by a $g \Phi g^{-1}$ for a suitable
$g\in \tilde{\bf G}$ we  many assume $T_0\subset {\bf G}$.
Then $J(\Phi){\rm d}x_1\wedge{\rm d}x_2$ is a weight vector for the action of $T_0$ on $\Omega^2(\!(2)\!)$.
Note that $O(\!(2)\!)$ decomposes into an infinite direct sum of weight spaces for $T_0$.
Since $\big(g(t_1,t_2)\big)(x^{\bf a})=t_1^{a_1}t_2^{a_2}\cdot x^{\bf a}$, all
$T_0$-weight spaces of $O(\!(2)\!)$ are $1$-dimensional and spanned by monomials $x^{\bf a}$.
As a consequence the $T_0$-weight spaces of
$\Omega^2(\!(2)\!)$ have the form $\Bbbk x^{\bf a}{\rm d}x_1\wedge{\rm d}x_2$ for ${\bf a}\in\Z_{\ge 0}^2$.
Since $J(\Phi)$ is invertible in $O(\!(2)\!)$ and $J(\Phi){\rm d}x_1\wedge{\rm d}x_2=\lambda
x^{\bf a}{\rm d}x_1\wedge{\rm d}x_2$ for some $\lambda\in\Bbbk^\times$ and ${\bf a}
\in\Z_{\ge 0}^2$, it must be that ${\bf a}=(0,0)$.
But then $H(2;\underline{n};\Phi)^{(1)}\subseteq H(2;\underline{n})$. The simplicity of $H(2;\underline{n};\Phi)^{(1)}$ now forces $H(2;\underline{n};\Phi)^{(1)}\subseteq H(2;\underline{n})^{(2)}$. Since
$p^{n_1+n_2}-2=\dim  H(2;\underline{n})^{(2)}\le
\dim H(2;\underline{n};\Phi)^{(1)}\le {p^{n_1+n_2}}-1$
this is impossible.
This contradiction shows that the maximal tori of $\bf G$ are at most $1$-dimensional.

On the other hand, using  the above expressions for $J(\Phi)$ it is straightforward, in each case, to produce a $1$-dimensional subtorus $T_\Phi$ of $T_0$ contained in $\bf G$. If $\Phi=\Phi(\tau)$ we take for $T_\Phi$ the identity component of $\{g(t_1,t_2)\in T_0\,|
\,\,t_1^{p^{n_1}-1} t_2^{p^{n_2}-1}=1\}$ and if $\Phi=\Phi(l)$, where $l\in\{1,2\}$,  we take $T_\Phi=\{g(t_1,t_2)\in T_0\,|\,\,t_l=1\}$. This completes the proof.
\end{proof}
\section{Proof of the main theorem}
\subsection{}\label{3.1} Let $M$ be as in the statement of Theorem~\ref{thm:main}. Our proof will make essential use of Weisfeiler's theorem; see \cite{W2} and \cite[Theorems~3.5.6, 3.5.7 ad 3.5.8]{Str04}. Let $O(m;\underline{n})$ and $W(m,\underline{n})$  be as in (\ref{2.8}). We shall often encounter the case where
$\underline{n}=\underline{1}$, that is $n_i=1$ for all $i$.
In this case all derivations of $O(m;\underline{n})$ are special and
$W(m;\underline{1})$ is the full derivation algebra  of
$O(m;\underline{1})\cong \Bbbk[X_1,\ldots,X_m]/(X_1^p,\ldots,X_m^p)$.
For $1\le i\le m$ we let $x_i$ stand for the divided-power generator
$x^{(\epsilon_i)}$ of $O(m;\underline{n})$ and denote by $O(x_{i_1},\ldots,x_{i_s})$ the divided power subalgebra of $O(m;\underline{n})$ generated by $x_{i_1},\ldots,x_{i_s}$ (as a $\Bbbk$-algebra $O(x_{i_1},\ldots,x_{i_s})$ is
isomorphic to a truncated polynomial ring in $n_{i_1}+\cdots+n_{i_s}$ variables). A Lie subalgebra $\mathcal{D}$ of
$W(m;\underline{n})\,=\,\sum_{i=1}^m\,O(m;\underline{n})\, \partial_i$ is called {\it transitive} if $O(m;\underline{n})$ does not contain nonzero proper $\mathcal{D}$-invariant ideals.
\subsection{}\label{3.2}
Given a non-empty subset $X\subseteq\g$ and $k\in \Z_{\ge 2}$ we let
$X^k$ denote the linear span of all elements
$\big((\ad\, x_1)\circ\cdots \circ(\ad\,x_{k-1})\big)(x_k)$ with $x_i\in X$.
Put $M_{(0)}:=M$ and choose an $(\ad\,M_{(0)})$-stable subspace $M_{(-1)}$ in $\g$
such that $M_{(0)}\subsetneq M_{(-1)}$ and $M_{(-1)}/M_{(0)}$ is an irreducible $M_{(0)}$-module. Given $i\in\Z_{\ge 2}$ we define a subspace $M_{(-i)}\subseteq \g$ recursively by setting
$M_{(-i)}\,=\,M_{(-1)}^i+M_{(-i+1)}.$
Since $M$ is a maximal subalgebra of $\g$ there exists a positive integer $q$ such that $\g=M_{(-q)}\supsetneq M_{(-q+1)}$. Set $M_{(1)}:=\{x\in M_{(0)}\,|\,\,[x,M_{(-1)}]\subseteq M_{(0)}$ and for $i\ge 2$
define $M_{(i)}:=\{x\in M_{(i-1)}\,|\,\,[x,M_{(-1)}]\subseteq M_{(i-1)}\}$. Since $\g$ is a simple Lie algebra there is a non-negative integer $r$ such that $0=M_{(r+1)}\subsetneq M_{(r)}$.
It is well known that
the chain of subspaces
$$\g=M_{(-q)}\supset\cdots\supset M_{(0)}\supseteq \cdots \supseteq M_{(r)}\supset 0$$ is a Lie algebra filtration, that is, $[M_{(i)},M_{(j)}]\subseteq M_{(i+j)}$ for all $i,j\in\Z$.
By maximality, $M$ is a restricted sublagebra of $\g$ and one can see by induction on $k$ that $M_{(k)}=\{x\in M_{(0)}\,|\,\,\big(\ad M_{(-1)}^k\big)(x)\subseteq M_{(0)}\}$ for all $k>0$.
From this it is immediate that
$M_{(k)}^{[p]}\subseteq M_{(pk)}$ for such $k$. The decreasing filtration of $\g$ thus obtained is referred to as the {\it Weisfeiler filtration} associated with the pair $(M_{(-1)}, M_{(0)})$.

Since $M_{(-1)}/M_{(0)}$ is an irreducible module over $M_{(0)}=M$ we have that
${\rm nil}(M)\subseteq M_{(1)}$. If $x\in M_{(1)}$ then $(\ad\,x)^{q+r+1}(M_{(-q)})=0$. So the ideal $M_{(1)}$ of $M$ consists of nilpotent elements of $\g$. This shows that $M_{(1)}={\rm nil}(M)$. Since we are assuming that ${\rm nil}(M)\ne 0$, Weisfeler's filtration does not collapse completely in our case. In other words, $r>0$.

We denote by $\G$ the corresponding graded Lie algebra
$$\G={\rm gr}(\g)\,:=\,\textstyle{\bigoplus}_{i=-q}^r\,\G_i,\qquad\ \G_i={\rm gr}_i(\g)=M_{(i)}/M_{(i+1)},$$ a graded vector space whose Lie product is induced by that of $\g$.
We can give $\G_{\ge 0}:=\bigoplus_{i\ge 0}\,\G_i$ a restricted Lie algebra structure by setting
$
\big({\rm gr}_i(x)\big)^{[p]}:=\,{\rm gr}_{pi}\big(x^{[p]}\big)$ for all $x\in M_{(i)}$
where  $i\ge 0$.
By construction, $\G_0\cong M/{\rm nil}(M)$ as Lie algebras, $\G_{-1}$ is an irreducible $\G_0$-module, and the Lie subalgebra $\G_{-}:=\bigoplus_{i<0}\,\G_i$ is generated by $\G_{-1}$.
\subsection{}\label{3.3}
Let $N(\G)=\bigoplus_{i<0}\,N_i(\G)$ be
the largest graded ideal of $\G$ contained in $\G_{-}$ and put $\bar{\G}=\G/N(\G)$ so that $\bar{\G}=\bigoplus_{i=-q}^r\,\bar{\G}_i$ where $\bar{\G}_i=\G_i/N_i(\G)$.
Since $N(\G)\cap \G_{-1}=0$ by the irreducibility of $\G_{-1}$ we have that $N(\G)\subseteq \bigoplus_{i\le -2}\,\G_i$ (this means that $\bigoplus_{i\ge -1}\,\G_i\cong\, \bigoplus_{i\ge -1}\,\bar{\G}_i$ as graded vector spaces).

Since $r>0$, the first part of Weisfeiler's theorem \cite{W2} says that
$N(\G)$ coincides with the radical of $\G$ and the semisimple Lie algebra
$\bar{\G}$ has a unique minimal ideal which contains $\bar{\G}_{-}:=\bigoplus_{i\le -1}\bar{\G}_i$. This ideal, denoted
$A(\bar{\G})$, is isomorphic to $S\otimes O(m;\underline{n})$ for some simple Lie algebra $S$ and some divided power algebra $O(m;\underline{n})$.
The restricted Lie algebra structure of $\G_{\ge 0}$
induces that on $\bar{\G}_{\ge 0}:=\bigoplus_{i\ge 0}\,\bar{\G_i}$.

Thanks to Block's theorem, the adjoint action of $\bar{\G}$ on $A(\bar{\G})$ gives rise to inclusions
\begin{eqnarray}\label{Eq1}
S\otimes O(m;\underline{n})\subset \bar{\G}\subseteq \big((\Der(S)\otimes O(m;\underline{n})\big)\rtimes \big({\rm Id}_S\otimes W(m;\underline{n})\big)
\end{eqnarray}
and the canonical projection $$\pi\colon\big((\Der(S)\otimes O(m;\underline{n})\big)\rtimes \big({\rm Id}_S\otimes W(m;\underline{n})\big)\twoheadrightarrow W(m;\underline{n})$$ maps $\bar{\G}$ onto a transitive subalgebra of $W(m;\underline{n})$; see \cite[Theorem~3.3.5]{Str04}, for example.
Since $\bar{\G}$ is semisimple (hence centreless) and acts faithfully on $A(\bar{\G})$, the restricted Lie algebra $\bar{\G}_{\ge 0}\,\cong {\ad}_{A(\bar{\G})}\big(\bar{\G}_{\ge 0}\big)$ identifies naturally with a restricted subalgebra of $\Der(A(\bar{\G}))$.
\subsection{}\label{3.4}
The uniqueness of  $A(\bar{\G})$ implies that it is a graded ideal of $\bar{\G}$.
The grading of $A(\bar{\G})$ is completely determined by the second part Weisfeiler's theorem which states that only one of the following two cases can occur:

{\it Degenerate Case.} If $A(\bar{\G})$ intersects trivially with $\bar{\G}_+:=\bigoplus_{i>0}\,\bar{\G}_i$ then necessarily $m\ge 1$ and the grading of $A(\bar{\G})$ is induced by that of the associative algebra $O(m;\underline{n})$. The latter is given by fixing a positive integer $s\le m$ and assigning to the divided power generators $x_1,\ldots, x_s$ and $x_{s+1},\ldots, x_m$ degrees $-1$ and $0$, respectively.
More precisely, for $i\le 0$ the graded component of $A(\bar{\G})$ has the form $A_{i}(\bar{\G})=S\otimes O(m;\underline{n})[i;s]$ where $O(m;\underline{n})[i;s]$ is the linear span of all monomials
$\prod_{i=1}^m x_i^{(a_i)}$ with $0\le a_i\le p^{n_i}-1$ and $a_1+\cdots +a_s=-i$.
Moreover, in this case $[[\G_{-1},\G_1],\G_1]=0$ and $\G_i=0$ for all $i\ge 2$. Finally,
$\G_1$ identifies with a nonzero subspace of $\sum_{i=1}^sO(x_{s+1},\ldots, x_m)\del_i$ through an embedding described in (\ref{Eq1}) and $$S\otimes O(x_{s+1},\ldots,x_m)\subset \G_0\subseteq\big(\Der(S)\otimes O(x_{s+1},\ldots,x_m) \big)\rtimes\big({\rm Id}_S\otimes W(m,\underline{n})[0;s]\big)$$
where $W(m;\underline{n})[0;s]\,=\,\textstyle{\sum}_{i>s} O(x_{s+1},\ldots,x_m)\del_i\oplus
\textstyle{\sum}_{i,j\le s}\, O(x_{s+1},\ldots,x_m)(x_i\del_j)$.

{\it Non-degenerate Case.} The simple Lie algebra
$S$ is $\Z$-graded in such a way that $S_{\pm 1}\ne 0$ and
$A_i(\bar{\G})=S_i\otimes O(m;\underline{n})$ for all $i\in\Z$.
The grading of $S$ induces that on $\Der(S)$ and $\G_0$ is sandwiched between $S_0\otimes O(m;\underline{n})$ and
$\big(\Der_0(S)\otimes O(m;\underline{n})\big)\rtimes\big({\rm Id}_S\otimes W(m;\underline{n})\big)$.
Since $\G_{-1}=S_{-1}\otimes(m;\underline{n})$ is a faithful $\G_0$-module and $S_0\otimes 1\subseteq S_0\otimes O(m;\underline{n})\subseteq \G_0$, the Lie subalgebra $S_0$ of $S$ acts faithfully on $S_{-1}$.
Although $S_{-1}$ does not have to be irreducible over $S_0$, in general, it
follows from \cite[Theorem~3.5.7(6)]{Str04} that
this module is semisimple and isogenic when
$\Der(S)\,=\,\ad S$.

Let $M$ be a counterexample to Theorem~\ref{thm:main}. By Lemma~\ref{regsub}, $M$ is not regular  in $\g$.
We set $M:=M_{(0)}$, choose a subspace $M_{(-1)}$ as in (\ref{3.2}), and consider the Weisfeiler filtration
of $\g$ associated with the pair $(M_{(-1)},M_{(0)})$.
Write $\G$ for the corresponding graded Lie algebra, and let $N(\G)$, $\bar{\G}$ and $A(\bar{\G})=S\otimes O(m;\underline{n})$ be as in (\ref{3.3}).
\subsection{}\label{3.5} In the next three subsections we assume that the degenerate case of Weisfeiler's theorem holds for $\bar{\G}$. In particular, this means that $m\ge 1$. If $m=1$ then
necessarily $s=1$. As a consequence,
$\G_1=\Bbbk\del_1$ is $1$-dimensional. Since $\G_2=0$ by (\ref{3.4}), it must be that $M_{(2)}=0$. So
${\rm nil}(M)=M_{(1)}$ is isomorphic to $\G_1$.
But then ${\rm nil}(M)=\Bbbk e$ for some nonzero nilpotent element $e\in\g$ and $M\subseteq \n_\g(\Bbbk e)$. This, however, contradicts our assumption on $M$ because $\n_\g(\Bbbk e)$ is contained in a proper parabolic subalgebra of $\g$
(see the proof of Corollary~\ref{max-root} for detail).

Now suppose $m\ge 2$. Since $S$ is a simple Lie algebra and $\dim\,O(m;\underline{n})\ge p^m$, we have the inequality $(\dim\,\g)/p^m>\dim\,S\ge 3$. Since $\g$ is exceptional and $p$ is a good prime for $\g$, this forces $m=2$, $\underline{n}=(1,1)$, and rules out the cases where $\g$ is of type ${\rm G}_2$ or ${\rm F}_4$.
Furthermore, if $\g$ is of type ${\rm E}_6$ then $p=5$ and $S\cong\sl_2$,
if $\g$ is of type ${\rm E}_7$ then $p=5$ and $S$ is either $\sl_2$ or $W(1;\underline{1})$, and if $\g$ is of type ${\rm E}_8$ then $p=7$ and $S\cong \sl_2$. In any event, all derivations of $S$ are inner which implies that
$$S\otimes O(2;\underline{1})\subset \bar{\G}\,=\, \big(S\otimes O(2;\underline{1})\big)\rtimes\big({\rm Id}_S\otimes\mathcal{D}\big)$$ for some transitive Lie subalgebra of $W(2;\underline{1})$.
\subsection{}\label{3.6}
Suppose $s=m=2$. Then $\G_1\subseteq \Bbbk\del_1\oplus \Bbbk\del_2$ and $$
\G_0\cong \bar{\G}_0\,=\big(S\otimes 1\big)
\,\textstyle{\oplus}\,\big({\rm Id}_S
\otimes \mathcal{D}_0\big)\subseteq\big(S\otimes 1\big)
\,\textstyle{\oplus}\,
\Big(\textstyle{\sum}_{i,j=1}^2\,\Bbbk(x_i\del_j)\big)$$
where $\mathcal{D}_0=\mathcal{D}\cap W(2;\underline{1})[0].$
If $\dim\,\G_1=1$ we can argue as  in the previous paragraph to conclude that $M$ is contained in a proper parabolic subalgebra of $\g$. So assume that
$\G_1=\Bbbk\del_1\oplus\Bbbk\del_2$. Then $\mathcal{D}_0\subseteq\textstyle{\sum}_{i,j=1}^2\,\Bbbk(x_i\del_j)\cong\gl_2$ acts faithfully
on $\G_1$. Let $Q$ denote the centraliser of ${\rm nil}(M)$ in $\g$. This is a restricted ideal of $M$. Since $\G_2=0$, the ideal
${\rm nil}(M)$ is abelian. Hence ${\rm nil}(M)\subseteq Q$.
Write $\bar{Q}$ for the image
of $Q$ in $\G_0=M/{\rm nil}(M)$. Since $S\otimes\Bbbk\subset \G_0$ commutes with $\G_1$
and $\mathcal{D}_0$ acts faithfully on $\G_1$, we now deduce that $\bar{Q}=S\otimes\Bbbk$.

Our next goal is to produce a nice subalgebra $R$ of $M$ satisfying the conditions of Lemma~\ref{borel}.
If $S\cong\sl_2$ we choose a standard basis
$\{e_0,h_0,f_0\}$ of $S$. If $S\cong W(1;\underline{1})$ we choose an $\sl_2$-triple
$\{e_0,h_0,f_0\}\subset S\setminus\{0\}$ such that $f_0\in \Bbbk \del$,
$h_0=x\del$ and $e_0\in\Bbbk (x^2\del)$.
This choice ensures that $\mathfrak{c}_S(f_0)=\Bbbk f_0$.
To ease notation we identify $S$ with $S\otimes\Bbbk\subset \bar{\G}_0$ and regard $e_0, h_0, f_0$ as elements of $\bar{\G}_0$.
Let $\tilde{h}_0\in Q$ be a preimage of $h_0$ under
the canonical homomorphism  $Q\twoheadrightarrow
\bar{Q}=Q/M_{(1)}$. By construction, $h_0$ is a toral element of $\G_0$. Replacing $\tilde{h}_0$ by its sufficiently large $[p]$-th power in $Q$ we may assume that $\tilde{h}_0$ is a semisimple element of $\g$.
Then the inverse image of $f_0$ under the canonical homomorphism $Q\twoheadrightarrow
\bar{Q}=Q/M_{(1)}$ contains an eigenvector for $\ad\,\tilde{h}_0$; we call it $\tilde{f}_0$. Since $f_0^{[p]}=0$ in $S$ it must be that $\tilde{f}_0^{[p]}\in M_{(1)}$. Therefore, $\tilde{f}_0$ is a nilpotent element of $\g$.
We now set $R:=\Bbbk\tilde{h}_0\oplus\Bbbk\tilde{f}_0\oplus {\rm nil}(M)$. Clearly, this is a Lie subalgebra of $M$ and $[R,R]\subseteq \Bbbk\tilde{f}_0\oplus {\rm nil}(M)$. So $[R,R]$ consists of nilpotent elements of $\g$ by Jacobson's formula for $p$-th powers.
By Corollary~\ref{max-root}, there exists a nonzero $e\in Q$ such that $[R,e]\subseteq \Bbbk e$ and $[e,[e,\g]]=\Bbbk e$. Let $\bar{e}$ denote the image of
$e$ in $\bar{Q}$.

Suppose $\bar{e}\ne 0$. As $[f_0,\bar{e}]=0$ we then have
$\bar{e}\in \Bbbk^\times f_0$. So our assumption on $\bar{e}$ entails $\rk\,(\ad\,f_0)^2\le 1$.
However, the restriction of $(\ad\,f_0)^2$ to every subspace $\bar{\G}_{-i}=
S\otimes O(2;\underline{1})[i;2]$ with $1\le i\le 2(p-1)$ is nonzero, forcing   $\rk\,(\ad\,e)^2\ge \rk\,(\ad\,f_0)^2\ge 2p-2$. This contradiction shows that $\bar{e}=0$. As a consequence,
$e\in M_{(1)}$. Since $\G_1=\Bbbk\del_1\oplus\Bbbk\del_2$, we may assume after a suitable linear substitution of $x_1,x_2$ that the image of $e$ in $\bar{\G}$ equals $\del_1$. But since $(\ad\,e)^3=0$ this would entail
$(\ad\,\del_1)^3=0$ which is false because $(\ad \partial_1)^{p-1}\ne 0$ and $p>3$. We thus conclude that the present case cannot occur.
\subsection{}\label{3.7}
Now suppose $m=2$ and $s=1$, so that again $S$ is either $\sl_2$ or $W(1;\underline{1})$. This case is quite similar to the previous one, but we need to choose a solvable subalgebra $R$ more carefully.
By (\ref{3.4}),  $\G_1\subseteq O(x_2)\del_1$ and
$\G_0\cong \bar{\G}_0=\big(S\otimes O(x_2)\big)
\,\textstyle{\oplus}\,\big({\rm Id}_S\otimes\mathcal{D}_0\big)$
where  $$\mathcal{D}_0=\mathcal{D}\cap W(2;\underline{1})[0;1]\,\subseteq\, O(x_2)\del_2\,\textstyle{\oplus}\, O(x_2)(x_1\del_1).$$
Let $\m_2$ be the maximal ideal of the local ring $O(x_2)$. Thanks to the transitivity of $\mathcal{D}\subseteq\mathcal{D}_0\oplus \G_1$ and the inclusion $\G_1\subseteq O(x_2)\del_1$ there exists
an element $d\in \bar{\G}_0$ such that $\pi(d)=\phi(x_2)\del_2+\psi(x_2)(x_1\del_1)\in\mathcal{D}_0$ has the property that $\phi(x_2)\in O(x_2)\setminus\m_2$. Since all derivations of $S$ are inner, our discussion in (\ref{3.4}) yields $d-\pi(d)\in S\otimes O(x_2)$.

As before, we may assume that $\dim \G_1\ge 2$. Indeed, otherwise ${\rm nil}(M)=M_{(1)}$ is $1$-dimensional implying that $M=\n_\g(M_{(1)})$ is contained in a parabolic subalgebra of $\g$.
We choose $e_0,f_0, h_0\in S$ as in (\ref{3.6}) and define $\bar{R}:=\Bbbk d\oplus\Bbbk (h_0\otimes 1)\oplus (f_0\otimes O(x_2))$. Then $[\bar{R},\bar{R}]\subseteq
f_0\otimes O(x_2)$ yielding $[\bar{R},\bar{R}]^{[p]}\subseteq
f_0^{[p]}\otimes O(x_2)=0$. Let $R$ be the preimage
of $\bar{R}$ under the canonical homomorphism
$M\twoheadrightarrow \G_0$. By construction, $R$ is a Lie subalgebra of $M$ with $[R,R]^{[p]}\subseteq M_{(1)}$. Therefore, it satisfies the assumptions of Lemma~\ref{borel}.

Now we need to locate $\bar{Q}$, the image of $Q=\c_\g(M_{(1)})$ in $\G_0$. It contains $S\otimes O(x_2)$ because the latter
commutes with $\G_1\subseteq O(x_2)\del_1$.
Therefore, $$\bar{Q}=\big({\rm Id}_S\otimes {\rm Ann}_{\mathcal{D}_0}(\G_1)\big)\,\textstyle{\oplus}\,\big(S\otimes O(x_2)\big).$$
Clearly, ${\rm Ann}_{\mathcal{D}_0}(\G_1)$
is an ideal of $\mathcal{D}_0$. Since $\G_1\subseteq O(x_2)\del_1$ is $d$-stable and $O(x_2)(x_1\del_1)$
preserves $\m_2\del_1\subset \G_1$, the subspace $\G_1$
contains an element of the form $a(x_2)\del_1$ with $a(x_2)\in O(x_2)\setminus\m_2$. Since $$[f(x_2)(x_1\del_1),a_2(x_2)\del_1]=-f(x_2)a(x_2)\del_1\qquad\quad\big(\forall\,f(x_2)\in
O(x_2)\big)$$
and $a(x_2)$ is invertible in $O(x_2)$, it must be that $\big(O(x_2)(x_1\del_1)\big)\bigcap{\rm Ann}_{{\mathcal D}_0}(\G_1)=0.$

Suppose ${\rm Ann}_{{\mathcal D}_0}(\G_1)\ne 0$. Then there exists $u=b(x_2)\del_2+c(x_2)(x_1\del_1)\in {\rm Ann}_{{\mathcal D}_0}(\G_1)$ such that $b(x_2),c(x_2)\in O(x_2)$ and
$b(x_2)\ne 0$. Since $O(x_2)(x_1\del_1)$ is an abelian ideal of the Lie algebra $O(x_2)\del_2\oplus O(x_2)(x_1\del_1)$ and ${\rm Ann}_{{\mathcal D}_0}(\G_1)$
is $(\ad d)$-stable, we may assume further that
$b(x_2)\not\in\m_2$. As $\dim\,\G_1\ge 2$ and $\m_2\del_1$ has codimension $1$ in $O(x_2)\del_1$
it must be that $(\m_2\del_1)\cap \G_1\ne 0$.
Let $v=g(x_2)\del_1$ be a nonzero vector in
$(\m_2\del_1)\cap \G_1$.
Then there is $k\in\{1,\ldots, p-1\}$ such that $g(x_2)\in\m_2^k\setminus\m_2^{k+1}$ and
$$[u,v]=[b(x_2)\del_2+c(x_2)(x_1\del_1),g(x_2)\del_1]=\big(b(x_2)g'(x_2)-c(x_2)g(x_2)\big)\del_1\not\in\m_2^k\del_1$$
(it is important here that $b(x_2)\not\in\m_2$). In particular, $[u,v]\ne 0$. Since this contradicts our choice of $u$ we now deduce that ${\rm Ann}_{{\mathcal D}_0}(\G_1)=0$.
This yields $\bar{Q}=S\otimes O(x_2)$.

By Corollary~\ref{max-root}, there exists a nilpotent element $e\in Q$ such that $[R,e]\subseteq \Bbbk e$ and $[e,[e,\g]]=\Bbbk e$. Let $\bar{e}$ be the image of $e$ in $\bar{Q}=S\otimes O(x_2)$. Since
$[f_0\otimes 1,\bar{e}]=0$  it must be that $\bar{e}=f_0\otimes q(x_2)$ for some $q(x_2)\in O(x_2)$. Suppose $\bar{e}\ne 0$. Note that
$O(x_2)(x_1\del_1)$ acts trivially
on $S\otimes O(x_2)$. Since $$[d,f_0\otimes q(x_2)]\in \Bbbk(f_0\otimes q(x_2))\bigcap\big(f_0\otimes \phi(x_2)q'(x_2)+S\otimes O(x_2)q(x_2)\big)$$ and $\phi(x_2)\not\in\m_2$, it must be that $q(x_2)\not\in\m_2$. From this it is immediate that the restriction of $(\ad \,\bar{e})^2$ to every subspace $\bar{\G}_{-i}=S\otimes O(2;\underline{1})[i;1]$ with $1\le i\le p-1$ is nonzero. But then
$\rk\,(\ad e)^2\ge \rk\,(\ad\,\bar{e})^2\ge p-1$. This contradiction shows that $e\in M_{(1)}$.
Identifying $M_{(1)}$ with $\G_1\subseteq O(x_2)\del_1$ we write $e=w(x_2)\del_1$ for some nonzero $w(x_2)\in O(x_2)$. Since $\phi(x_2)\not\in\m_2$ and
$$[d,w(x_2)\del_1]=\big(\phi(x_2)w'(x_2)-
\psi(x_2)w(x_2)\big)\del_1\in \Bbbk (w(x_2)\del_1)$$
it is straightforward to see that $w(x_2)\not\in\m_2$. But then
$(w(x_2)\del_1)^{p-1}$ is a nonzero linear operator on
$O(2;\underline{1})$ implying that $(\ad \,e)^3\ne 0$. Since this is false, we finally conclude that the degenerate case of Weisfeiler's theorem cannot occur for $\bar{\G}$.
\subsection{}\label{3.8}
From now on we may assume that the non-degenerate case of Weisfeiler's theorem holds for $\bar{\G}$,  that is $S=\bigoplus_{i\in\Z}\,S_i$ is $\Z$-graded,
$A_i(\bar{\G})=S_i\otimes O(m;\underline{n})$ for all $i\in\Z$, and $$S_0\otimes O(m;\underline{n})\subseteq \bar{\G}_0\subseteq \big(\Der_0(S)\otimes O(m;\underline{n})\big)\rtimes \big({\rm Id}_S\otimes\mathcal{D}\big)$$ where $\mathcal{D}=\pi(\bar{\G}_0)$ is a transitive subalgebra of $W(m;\underline{n})$. We shall often identify $\G_{\ge 0}$ with $\bar{\G}_{\ge 0}$.

First suppose that $S=\Lie(\mathcal{H})$ for some simple algebraic $\Bbbk$-group $\mathcal{H}$ (this excludes the case where $S\cong \psl_{kp}$ for some $k\ge 1$).
Then $\Der(S)=\ad S$; see \cite[Lemma~2.7]{BGP}, for example.
It follows that $\bar{\G}_i=S_i\otimes O(m;\underline{n})$ for all $i\ne 0$.
Our grading of $S$ is induced by the action of a $1$-dimensional torus $T_0$ of $({\rm Aut}\,S)^\circ=\Ad\,\mathcal{H}$. Differentiating this action we  find
a toral element $t_0\in \Lie(T_0)$ such that
$[t_0,x]=\bar{i}x$ for all $x\in S_i$ where $i\in \Z$.
Here and in what follows we write $\bar{i}$ for the image of $i$ in $\F_p\subset\Bbbk$.
The element $t_0\in\Der(S)$ is often referred to as the {\it degree derivation} of the graded Lie algebra $S$. In the present case the derivation $t_0$ is inner.

It is well known that $S_{\ge 0}=\bigoplus_{i\ge 0}\,S_i$ is a parabolic subalgebra of $S$ and $S_0=S^{T_0}$ is a Levi subalgebra of $S_{\ge 0}$. Obviously, $t_0\in S_0$
and we may assume without loss of generality
that $S_{\ge 0}$ is a standard parabolic
subalgebra of $S$.
Since $M_{(0)}=M$ is a restricted subalgebra of $\g$, there exists a toral element $\tilde{t}_0\in M_{(0)}$
which maps onto $t_0$ under the canonical homomorphism $M_{(0)}\twoheadrightarrow \G_0$. If $\c_{\G}(t_0)\subseteq \G_{\ge 0}$ then $\g_{\tilde{t}_0}\subseteq M_{(0)}$. Since $\g_{\tilde{t}_0}$ is a Levi subalgebra of $\g$, it contains a maximal torus of $\g$. This, however, contradicts our assumption that $M$ is not a regular subalgebra of $\g$. So it must be that $\c_{\G}(t_0)\cap \G_i\ne 0$ for some $i<0$.

Since $\Der(S)=\ad S$, our discussion at the end of (\ref{3.4}) shows that
the $S_0$-module $S_{-1}$ is semisimple and isogenic.
In particular, $\z(S_0)$ acts on $S_{-1}$ by scalar endomorphisms forcing $\z(S_0)=\Bbbk t_0$.
From this it follows that the parabolic subalgebra $S_{\ge 0}$ is maximal in $S$ and our grading
of $S$ is standard.
This means that there exist a maximal torus $T$ of $\mathcal{H}$ containing $T_0$, a basis of simple roots
$\{\alpha_1,\ldots,\alpha_\ell\}$ in the root system $\Phi(\mathcal{H},T)$, and a positive integer $d\le \ell$ such that $S_k$ with $k\ne 0$ is spanned by all root vectors $e_\gamma$ of $S$ with respect to $T$ such that $
\gamma=\sum_{i=1}^\ell m_i\alpha_i$ and $m_d=k$; see \cite[2.4]{BGP} for details.

Note that $p>3$ is good for $S$ provided that $S$ is not of type ${\rm E}_8$ and if $S$ is of type ${\rm E}_8$ then so is $\g$. So in any event $p$ is good for $S$.
Since our grading of $S$ is standard we now obtain that $\bar{\G}_+=\bigoplus_{i>0}\,\big(S_i\otimes O(m;\underline{n})\big)$ is generated by $\bar{\G}_1=S_1
\otimes O(m;\underline{n})$. Furthermore,  $\dim\, S_{i}=\dim\,S_{-i}$ for all $i\in\Z$, and
$S_i=0$ for all $i\ge p$.    This yields
$\c_{\bar{\G}}(t_0)\subseteq \bar{\G}_{\ge 0}.$ Since $\c_{\G}(t_0)\cap \G_i\ne 0$ for some $i<0$ it follows that $N(\G)\ne 0$.

Let $l$ be the smallest positive integer for which $N_{-l}(\G)\ne 0$. Then $2\le l\le p$ and $[N_{-l}(\G), \G_1]=0$.
Since $\g$ is generated by $M_{(-1)}$, the toral element $\tilde{t}_0$ preserves each $M_{(i)}$ with $i\in\Z$
and acts on $M_{(i)}/M_{(i+1)}$ as the scalar operator $\bar{i}\cdot {\rm Id}$. Since $\ad\,\tilde{t}_0$ is diagonalisable, for every $i\in\Z$ there exists a subspace $V_i\subset\g$
such that $[\tilde{t}_0,v]=\bar{i}v$ for all $v\in V_i$ and $M_{(i)}=V_i\oplus M_{(i+1)}$.
Given $a\in \Bbbk$ and an $(\ad\,\tilde{t}_0)$-invariant subspace $W$ of $\g$ we write
$W(a)$ for the set of all $w\in W$ with $[\tilde{t}_0,w]=aw$.

Let $\bar{u}\in N_{-l}(\G)\setminus\{0\}$ and pick $u\in V_{-l}$ which maps onto $\bar{u}$ under the canonical homomorphism $M_{(-l)}\twoheadrightarrow
\G_{-l}=
M_{(-l)}/M_{(-l+1)}$. Since
$[\bar{u}, \G_1]=0$ it must be that $[u,M_{(1)}]\subseteq
M_{(-l+2)}$. But then $$[u, V_1]\subseteq
M_{(-l+2)}(\overline{-l+1})\,=\,\big(V_{-l+2}\,\textstyle{\oplus}\,\cdots\,\textstyle{\oplus}\,
V_{-1}\textstyle{\oplus}\,V_0\,\textstyle{\oplus}\,
M_{(1)}\big)\cap \g(\overline{-l+1})\,=\,M_{(1)}(\overline{-l+1}).
$$
(It is of utmost importance here that the set of residues of $0,-1,\ldots, -l+2$ in $\F_p=\Z/p\Z$ does not contain the residue of $-l+1$). This shows that $[u,V_1]\subseteq M_{(1)}$. Since $V_1$ maps onto $\G_1$ under the canonical homomorphism $M_{(1)}\twoheadrightarrow \G_1$ and $\G_1$ generates the Lie algebra $\G_{+}$ by our remarks earlier in the proof, the nilpotent Lie algebra $M_{(1)}$ must be generated by $V_1$.
This gives $[u, M_{(1)}]\subseteq M_{(1)}$ forcing  $u\in\n_\g(M_{(1)})$. But then $u\in M$ by the maximality of $M$. This contradiction shows that the case where $S =\Lie(\mathcal{H})$ is impossible.
\subsection{}\label{3.9}
In the next four subsections we assume that $S=\psl_{kp}$ for some $k\in\Z_{>0}$. Recall that for a finite dimensional Lie algebra $L$ over $\Bbbk$ the {\it absolute toral rank} ${\rm TR}(L)$ denotes the maximal dimension of toral subalgebras in the restricted Lie algebras $\tilde{L}/\z(\tilde{L})$ where $\tilde{L}$ runs over the set of all finite dimensional $p$-envelopes of $L$. Since $p>2$, the restricted Lie algebra $\psl_{kp}$ contains a self-centralising torus of dimension $kp-2$. Similarly, $\g$ contains a self-centralising torus whose dimension equals $\rk(\g)$. Combining \cite{P87} and \cite{P90}
with \cite[Theorem~1.2.9]{Str04} one obtains that ${\rm TR}(\g)=\rk(\g)$ and $TR(\psl_{kp})=kp-2$. On the other hand, Skryabin's theorem says that
that ${\rm TR}(\g)\ge {\rm TR}(\G)$; see \cite[Theorem~5.1]{Skr98}. By \cite[Theorems~1.2.8(3) and ~1.2.7(1)]{Str04}, this gives
$$kp-2={\rm TR}(\psl_{kp})={\rm TR}(S)\le {\rm TR}(\bar{\G})\le {\rm TR}(\G)\le {\rm TR}(\g)=\rk(\g)\le 8.$$ Since $p>5$ when $G$ is of type ${\rm E}_8$, this yields $k=1$. Since $\dim \g> (p^2-2)p^{|\underline{n}|}$ where $|\underline{n}|=n_1+\cdots+n_m$, it must be that $m=0$ unless $\g$ is of type ${\rm E}_7$, $p=5$, and $(m,\underline{n})=(1,\underline{1})$.

The grading of $S$ is induced by the action of a $1$-dimensional torus
$\lambda(\Bbbk^\times)$
of $({\rm Aut}\,S)^\circ\cong {\rm PGL}_p$ on $S$ and
hence $S_{\ge 0}$ is a parabolic subalgebra of $S$.
However, it is no longer true in the present case that the degree derivation $t_0\in\Lie(\lambda(\Bbbk^\times))$  lies in $\ad S_0$. Essentially this is due to the fact that $\Der(\psl_p)\,\cong\, \mathfrak{pgl}_p$; see \cite[Lemma~2.7]{BGP}, for example.
Still
$\lambda(\Bbbk^\times)$ is contained in a maximal torus
$T$ of ${\rm PGL}_p$ and there is a basis of simple roots $\Delta$ in the root system of ${\rm Aut}\,S$ with respect to $T$ and a collection of non-negative integers $\{r_\alpha|\,\alpha\in\Delta\}$ such that
$\alpha(\lambda(t))=t^{r_\alpha}$ for all $\alpha\in\Delta$ and all $t\in\Bbbk^\times$.

If $m=0$ then \cite[Theorem~3.5.7.(6)]{Str04} shows that
the $S_0$-module $S_{-1}$ is semisimple and isogenic.
If $m>0$ then $S_0=\c_S(\lambda)$ is a proper standard Levi subalgebra of $S\cong \mathfrak{psl}_5$. Let $\Sigma_0$ be the root system of $S_0$ with respect to $T$.
Since $\Der(S)$ has type ${\rm A}_4$, the Lie algebra $S_0$ is either toral or
$\Sigma_0$ has type
${\rm A}_1$, ${\rm A}_1^2$, ${\rm A}_2$, ${\rm A}_2{\rm A}_1$ or ${\rm A}_3$. In any event, this implies that 
$S_{-1}$ is a semisimple $S_0$-module and all its irreducible submodules have dimension $1$, $2$, $3$, $4$ or $6$. Moreover, in the last two cases $S_{-1}$ is an irreducible $S_0$-module.
If $S_0$ is toral then $S_{-1}$ coincides with  the span of the roots vectors $e_{-\alpha}\in S$ such that $\alpha\in\Delta$ and $r_\alpha=1$. In particular, $\dim S_{-1}\le 4$. Taking all this into account it is not hard observe that
$S_{-1}$ always contains an irreducible $S_0$-submodule $U_0$ such that $\dim S_{-1}<p\dim U_0$.
Applying \cite[Theorem,~3.5.7(6)]{Str04} we now deduce that when $m>0$ the $S_0$-module $S_{-1}$ is semisimple and isogenic, too.
This implies that in all cases of interest the parabolic subalgebra $S_{\ge 0}$ is maximal in $S$. From this it is immediate that  $S=S_{-1}\oplus S_0\oplus S_1$.
\subsection{}\label{3.10}
The above discussion in (\ref{3.9}) also shows that $\Der(S)=\Bbbk t_0\oplus\ad S$ and
$$S_0\otimes O(m;\underline{n})\subseteq \G_0\subseteq \big(\Der_0(S)\otimes O(m;\underline{n})\big)\rtimes\big({\rm Id_S}\otimes \mathcal{D}\big)$$ for some transitive subalgebra $\mathcal{D}$ of $W(m;\underline{n})$.
Since all Cartan subalgebras of $\sl_p$ are toral the subalgebra $S_0$ of $S=\sl_p/\Bbbk 1$ contains a toral Cartan subalgebra of $S$, say $\t$. Clearly, $\t$ is a toral subalgebra of $\G_0=\bar{\G}_0$. Since the canonical homomorphism
$\eta\colon\,M_{(0)}\twoheadrightarrow \G_0$ is  restricted, there exists a toral subalgebra $\tilde{\t}$ in $M_{(0)}$ such that $\eta(\tilde{\t})\cong\t$; see \cite[Proposition~1.2.2]{Str04}, for example.

Since $\t\subset S_0$ is self-centralising we have that $\c_S(\t)\cap S_{-1}=0$. Consequently, $\c_{\bar{\G}}(\t)\cap \bar{\G}_{-1}=0$.
If $\G_{-2}=0$ then $\c_{\G}(\t)\subset \G_{\ge 0}$ forcing
$\c_\g(\tilde{\t})\subset M_{(0)}=M$. As this contradicts  our assumption that $M$ is non-regular we now deduce that $\G_{-2}=N_{-2}(\G)\ne 0$.
Since $N(\G)\subseteq \bigoplus_{i\le -2}\,N_i(\G)$ we have that $\G_1\cdot N_{-2}(\G)=0$.
The graded factor-space $\bar{N}(\G):=N(\G)/N(\G)^2$ carries a natural
$\bar{\G}$-module structure and $\bar{N}_{-i}(\G)\cong
N_{-i}(\G)$ for $i=2,3$ as vector spaces. Therefore, $\bar{N}_{-2}(\G)\ne 0$.
If $\bar{\G}_{-1}\cdot \bar{N}_{-2}(\G)=0$
then $\G_{-3}=[\G_{-1},\G_{-2}]=N_{-3}(\G)=0$
because $N_{-3}(\G)\cong \bar{N}_{-3}(\G)$.

Suppose $\bar{\G}\cdot\bar{N}(\G)=0$. Then the above shows that $N(\G)=\G_{-2}$ and $[\G_{\pm 1},N(\G)]=0$.
It is well known that the torus $\tilde{\t}\subset M_{(0)}$ is contained in a maximal torus of $\g$ and all such tori are
$(\Ad\,G)$-conjugate by \cite[11.8]{Borel} or \cite[Theorem~13.3]{Hum}.
Since $\t$ is a maximal torus of $S_0$ we have that $\dim\,\tilde{\t}=p-2$.
If $p\ge 7$ then $p-2>\frac{1}{2}\rk(\g)$.
In this case, the restriction of the (non-degenerate) Killing form of $\g$ to $\tilde{\t}$ is nonzero. Since $\tilde{\t}\subset M_{(0)}$ and $\t\subset S_0\subseteq [\G_{-1},\G_1]$, one of the composition factors
of the $S$-module $\bar{\G}=\G/\G_{-2}$ provides $S$ with a nonzero trace form. But since $S\cong\psl_p$ no such form can exist by \cite{Bl}
or \cite{Gar}.
This shows that if $\bar{\G}\cdot \bar{N}(\G)=0$ then $p=5$ and $\g$ is not of type ${\rm E}_8$. Arguing similarly one also observes that $G$ cannot be of type ${\rm G}_2$ or ${\rm F}_4$ (even when $p=5$).
\subsection{}\label{3.11}
Suppose $m=0$. Then $S_0\subseteq \G_0\subseteq
\Der_0(S)=S_0\oplus\Bbbk t_0$.
If $t_0\in \G_0$ then arguing as in (\ref{3.8}) we could find a toral element $\tilde{t}_0\in M_{(0)}$ with $\g_{\tilde{t}_0}\subseteq M$. Since this  contradicts our assumption that $M$ is non-regular, the equality $\bar{\G}=S$ must hold. In particular, this means that $\G_1=S_1$ and
$\G_0=S_0$ acts faithfully on $\G_1$.
From this it is immediate that the subalgebra $\c_\g({\rm nil}(M))$ of $M= M_{(0)}$ coincides with $M_{(1)}$. Then there exists a nonzero $e\in M_{(1)}$ such that $[e,[e,\g]]=\Bbbk e$; see
Corollary~\ref{max-root}. Its image $\bar{e}={\rm gr}_1(e)$ in $\G_1=S_1$ has the property that
$[\bar{e},[\bar{e}, \G_i]]=0$ for all $i\le -2$.

Since $\bar{\G}=S$ we have that $\bar{\G}=\bar{\G}^{(\infty)}$.
If all composition factors of the $\bar{\G}$-module
$\bar{N}(\G)$ are trivial then $\bar{\G}\cdot \bar{N}(\G)=\bar{\G}^{(\infty)}\cdot \bar{N}(\G)=0$. Then
$\bar{\G}_{-1}\cdot \bar{N}_{-2}(\G)=0$ and
our discussion in (\ref{3.10}) shows that $p=5$. Since
$\G_{-2}=[\G_{-1},\G_{-1}]$ is a homomorphic image of $\wedge^2\, \G_{-1}$
and $\dim\,\G_{-1}=k(5-k)$ for some $1\le k\le 4$
we have that $\dim\,\G_{-2}\le 15$. But then $\dim \g=
\dim\,S+\dim\,\G_{-2}\le 23+15=48$, a contradiction.

Let $V$ be a non-trivial composition factor of the $S$-module $\bar{N}(\G)$ and denote by $\rho\colon\,
S\to\gl(V)$ the corresponding representation of $S=\psl_p$. We shall regard $\rho$ as a representation of  $\sl_p=\Lie(\SL_p)$ by inflation. Our earlier remarks then show that $\rho(a)^2=0$ for some non-central element $a\in \sl_p$ (one can take $a=\bar{e}$ since $(\ad_{\G}\bar{e})^2=0$). It should be stressed at this point that the centre $\Bbbk I_p$ of $\sl_p$ acts trivially on $V$.
By Lemma~\ref{graded}, $\rho$ is a restricted representation of $S$. So, thanks to Curtis' theorem it can be obtained by differentiating an infinitesimally irreducible rational representation of $\SL_p$.
Since the $(\Ad\,\SL_p)$-stable set $\{X\in\sl_p\,|\,\,\rho(X)^2=0\}$ is non-central, Zariski closed and conical, it contains the orbit $\OO_{\rm min}$ of $\sl_p$.
This shows that $\rho(e_\alpha)^2=0$ for any root vector $e_\alpha\in \sl_p$.
Arguing as in
\cite[p.~72]{PSup} one then observes that $\rho$ is a fundamental representation of $\sl_p$, i.e. $V\cong L(\varpi_k)$ for some $1\le k\le p-1$.  But then the central element $I_p\in\sl_p$ acts on $V$ as $k\cdot {\rm Id}_V\ne 0$, a contradiction. As a result, $m\ge 1$.
\subsection{}\label{3.12}
Our discussion in (\ref{3.11})
implies that $p=5$, $\g$ is of type ${\rm E}_7$ and $O(m;\underline{n})=O(1;\underline{1})$. Then $$
\dim\,\bar{N}(\G)\le \dim \,N(\G)=\dim\,\g-\dim\,\bar{\G}<\dim\,\g-\dim\,S\otimes O(1;\underline{1})=133-115=18.
$$
We claim that $\G_{-3}=0$. Indeed, if this is not the case then $\bar{\G}_{-1}\cdot \bar{N}_{-2}(\G)\ne 0$.
Since $A(\bar{\G})=S\otimes O(1;\underline{1})$ is a perfect Lie algebra, this implies that one of the composition factors of the $A(\bar{\G})$-module $\bar{N}(\G)$ is non-trivial. This, in turn, shows that there is a composition factor $W$ of the $\bar{\G}$-module $\bar{N}(\G)$ such that $A(\bar{\G})\cdot W=W$. In this situation Block's theorem on derivation simple modules says that
there exists a faithful $S$-module $W_0$ such that $W=W_0\otimes O(1;\underline{1})$ as vector spaces; see \cite[Corollary~3.3.7]{Str04}.
As $\dim \,W\le \dim\,\bar{N}(\G)<18$ and $p=5$ we get $\dim\,W_0\le 3$. But then $23=\dim S\le \dim \gl(W_0)=9$. By contradiction the claim follows.

Let $M'$ denote the inverse image of $S_0\otimes O(1;\underline{1})$ under the canonical homomorphism $M_{(0)}\twoheadrightarrow \G_0$. Since $\dim\,S_{-1}=k(5-k)$ for some $1\le k\le 4$, we have that $\dim\,S_{\ge 0}\ge 23-6=17$. As a consequence, $\dim M'= \dim \G_{\ge 0}>\dim\big(S_{\ge 0}\otimes O(1;\underline{1})\big)\ge 85>
(\dim\,\g)/2.$ It follows that the restriction of the Killing form $\kappa$ of $\g$ to $M$ is nonzero.
Since $\ad M_{(1)}$ acts nilpotently on $\g$, it lies in the radical of the trace form
$\kappa_{\vert\, M'}$. Let $\kappa_L$ denote  the Killing form of a finite dimensional Lie algebra $L$. Since $M=M_{(0)}$ preserves each component $M_{(i)}$ of our filtration, we have that
\begin{equation}\label{kappa}\kappa(x,y)=\tr\big((\ad\,x)\circ(\ad\,y)\big)=
\kappa_\G({\rm gr}_0(x),{\rm gr}_0(y))\qquad\quad\ (\forall\,x,y\in M).\end{equation}
In view of the above this shows that the restriction of $\kappa_\G$ to $S_0\otimes O(1;\underline{1})$ is nonzero.  As $[\G_{\pm 1},\G_{-2}]=0$
by our earlier remarks and $S_0\otimes O(1;\underline{1})=[\G_{-1},\G_1]$, we have that $\kappa_\G(x,y)=
\kappa_{\bar{\G}}(x,y)$ for all $x,y\in S\otimes O(1;\underline{1})$ (as before we identify $\bar{\G}_0$ with $\G_0$).
This implies that the restriction of $\kappa_{\bar{\G}}$ to $A(\bar{\G})=S\otimes O(1;\underline{1})$ of $\bar{\G}$ is nonzero.
Since $A(\bar{\G})$ is the unique minimal ideal of $\bar{\G}$ and the form $\kappa_{\bar{\G}}$ is
$\bar{\G}$-invariant, the restriction
of $\kappa_{\bar{\G}}$ to $A(\bar{\G})$ must be
non-degenerate.
On the other hand, it is straightforward to see that
for any $s\in S$ and any $u\in A(\bar{\G})$ the linear operator  $$\ad_{A(\bar{\G})}\big(s\otimes x_1^{p-1}\big)\circ \ad_{A(\bar{\G})}u$$ is nilpotent and hence has zero trace. This contradiction shows that the case where $S=\psl_{kp}$ is impossible.
\subsection{}\label{3.13}
Suppose $S$ is a Lie algebra of Cartan type $W$. Since $\dim W(r;\underline{d})=rp^{|\underline{d}|}$ and $\dim S\le \dim \g$, the Lie algebra $S$ is in the following list:
$$W(1;\underline{1}), W(1;\underline{2}), W(1;\underline{3}), W(2;\underline{1}).$$
If $S=W(1;\underline{d})$ for $1\le |\underline{d}|\le 3$ then
we may assume that its grading is either natural or reverse thereof; see \cite[Theorem~4.7]{PS3}. In the second case,
$S_{-2}\ne 0$ and $\dim S_{-1}=1$. But then
$$0\ne S_{-2}\otimes O(m;\underline{n})=
\bar{\G}_{-2}=[\bar{\G}_{-1},\bar{\G}_{-1}]=
[S_{-1},S_{-1}]\otimes O(m;\underline{n})=0,$$
a contradiction. Thus, $S=\bigoplus_{i\ge -1}\,S_i$
and all graded components of $S$ are $1$-dimensional.
Moreover, there exists $t_0\in S_0$ such that $[t_0,x]=\bar{i}x$ for all $x\in S_i$.
It follows that $N(\G)=\bigoplus_{i\le -2}\,\G_i$, forcing $[\G_{-2}, \G_1]=0$. Equivalently,
$[M_{(-2)},M_{(i)}]\subseteq M_{(i-1)}$ for all $i\ge 0$.

The subspace $\Bbbk (t_0\otimes 1)$ is a $1$-dimensional torus contained in $\z(\G_0)$. As both $\G_0$ and $M_{(0)}$ are restricted, there exists a toral element $\tilde{t}_0\in M_{(0)}$ which maps onto $t_0$ under the canonical epimorphism $M_{(0)}\twoheadrightarrow \G_0$.
If $N(\G)=0$ then $\g_{\tilde{t}_0}\subset M_{(0)}$ contrary to our assumption that $M=M_{(0)}$ is
non-regular. This shows that $\G_{-2}\ne 0$.
As in (\ref{3.8}) we choose a subspace $V_{i}$ in
$M_{(i)}$ such that $M_{(i)}=V_i\oplus M_{(i+1)}$ and $[\tilde{t}_0, x]=\bar{i}x$ for all $x\in V_i$.
By our preceding remarks, $V_{-2}\ne 0$ and $[V_{-2},V_i]\subseteq M_{(i-1)}\cap M(\overline{i-2})\subset M_{(1)}$ for $i=1,2$.
Since $[V_{-2}, M_{(3)}]\subseteq M_{(1)}$ and $M_{(1)}=V_1\oplus V_2\oplus M_{(3)}$ we thus deduce
that $0\ne V_{-2}\subset \n_\g(M_{(1)})=M_{(0)}$. This contradiction shows that $S\ne W(1;\underline{d})$.
\subsection{}\label{3.14}
Suppose $S=W(2;\underline{1})$, a Lie algebra of dimension $2p^2$. Then $m=0$ and $\Der(S)=\ad S$ by \cite[Theorem~7.1.2(1)]{Str04}.
Hence $\bar{\G}=S$. In particular, this implies that $S_{-1}$ is an irreducible and faithful $S_0$-module.
By \cite[Theorem~4.7]{PS3}, any grading of $W(2;\underline{1})$ is given by assigning to
some generators  $u_1,u_2$ of the maximal ideal of $O(2;\underline{1})$ certain degrees $a_1,a_2\in \Z$. The element $u_1^{m_1}u_2^{m_2}\frac{\del}{\del u_k}\in W(2;\underline{1})$ with $k=1,2$ then acquires the degree $a_1m_1+a_2m_2-a_k$. The grading of $S$ obtained this way is said to have type $(a_1,a_2)$ with respect to $u_1,u_2$. To ease notation we assume that $u_i=x_i$ for $i=1,2$.
We write $S[i]$ for the $i$-th component of the standard grading of the Cartan type Lie algebra $S$ which has type $(1,1)$.

Suppose our grading of $S$ has type $(a_1,a_2)$. No generality will be lost by assuming that
$|a_1|\ge |a_2|$. Then $S_0$
is spanned by some $x_1^{i_1}x_2^{i_2}\del_k$
with $k=1,2$. We now follow very closely the argument in \cite[Lemma~4.15]{PS3}.
If $\{\del_1,\del_2,x_1\del_2\}\cap S_0=\emptyset$ then $$ \Bbbk x_1\del_1\oplus \Bbbk x_2\del_2\subseteq S_0\subseteq \Bbbk x_1\del_1\oplus \Bbbk x_2\del_2\oplus \Bbbk x_2\del_1\oplus \textstyle{\sum}_{i\ge 1}S[i].$$ In this case $S_0$ is solvable and $[S_0,S_0]\subseteq \Bbbk x_2\del_1\oplus\sum_{i\ge 1}S[i]$ acts nilpotently on $S$. Since  $S_{-1}$ is an irreducible and faithful $S_0$-module, this is impossible. Therefore,
$\{\del_1,\del_2,x_1\del_2\}\cap S_0\neq
\emptyset$, implying that either $a_2=0$ or
$a_1=a_2$. Since $S_{-1}\ne 0$ we now deduce that either
the grading of $S$ or its reverse has type $(1,1)$ or $(1,0)$.
As before, we set $Q:=\c_\g(M_{(1)})$, pick $e\in Q\cap\OO_{\rm min}$, and denote by
$\bar{Q}$ the image of $Q$ in $\G_0=M_{(0)}/M_{(1)}$.
Write $\bar{e}$ for the image of $e$ in $\bar{Q}$.

If the grading of $S$ has type $(-1,-1)$ then
$\G_1=S_{1}=\Bbbk\del_1\oplus\Bbbk\del_2$ and $S_i=0$ for $i\ge 2$. Furthermore, $\G_0=S_0=\sum_{i,j=1}^2\,\Bbbk(x_i\del_j)$ acts on $S_{1}$ faithfully.
From this it is immediate that $\bar{Q}=0$. But then $\bar{e}={\rm gr}_1(e)$ is a nonzero linear combination of $\del_1$ and $\del_2$. Since $(\ad\,\bar{e})^3=0$ and $p>3$, this is impossible. So this case cannot occur. If the grading of $S$ has type $(-1,0)$ then $S_1=\sum_{i=0}^{p-1}\,(\Bbbk x_2^i\del_1)$ and $S_i=0$ for $i\ge 2$. Moreover, the Lie algebra $$S_0\,=\,\textstyle{\sum}_{i=0}^{p-1}\,(\Bbbk x_2^i\del_2)\,\oplus\,\textstyle{\sum}_{i=0}^{p-1}\,\Bbbk x_2^i(x_1\del_1)$$ is isomorphic to $W(1;\underline{1})\ltimes O(1;\underline{1})$
where $\sum_{i=0}^{p-1}\,\Bbbk x_2^i(x_1\del_1)\cong O(1;\underline{1})$ is an abelian ideal of $S_0$.
It is straightforward to check that $\G_0=S_0$ acts faithfully on $\G_1=S_1$. So the equality $\bar{Q}=0$ must hold. Let
$\bar{R}=\Bbbk \del_2\oplus S_1$ and denote by $R$  the inverse image of $\bar{R}$ in $M_{(0)}$. Then $[R,R]\subseteq M_{(1)}$ consists of nilpotent elements of $\g$. In view of Corollary~\ref{max-root} we may assume that $[R,e]\subseteq \Bbbk e$.
Then $[\del_2,\bar{e}]=0$ forcing $\bar{e}=\lambda\del_1$ for some $\lambda\in\Bbbk$. This, however, contradicts the fact that $(\ad \del_1)^4\ne 0$.

As a result, we may assume that the grading of $S$ is either standard or has type $(1,0)$.
These are the reverse gradings of the ones considered earlier.
In particular, $\bar{\G}_{-2}=S_{-2}=0$. Let $\t=\Bbbk (x_1\del_1)\oplus\Bbbk(x_2\del_2)$, a maximal toral subalgebra of $\bar{\G}=S$ contained in $\G_0=S_0$. Since $\G_0$ is restricted, there exists a $2$-dimensional toral subalgebra $\tilde{\t}$ in $M_{(0)}$ which maps onto $\t$ under the canonical homomorphism $M_{(0)}\twoheadrightarrow \G_0$; see
\cite[Proposition~1.2.2]{Str04}. Since $M$ is non-regular and $\c_{\bar{\G}}(\t)=\c_S(\t)=\t$ it must be that $N(\G)\ne 0$. So the above yields $0\ne \G_{-2}\subseteq N(\G)$. Since $\Der(S)=\ad S$, the Lie algebra $\G_0$ contains the degree derivation $t_0$
associated with our present grading of $S$. As before, we find a toral element $\tilde{t}_0\in M_{(0)}$ which maps onto $t_0$ under the canonical
homomorphism $M_{(0)}\twoheadrightarrow \G_0$ and pick, for any $i\in\Z$, a subspace $V_i$ of $\g$ such that $M_{(i)}=V_i\oplus M_{(i+1)}$ and $[\tilde{t}_0,v]=\bar{i} v$ for all $v\in V_i$. Repeating verbatim the argument used in the $W(1;\underline{n})$-case we obtain that
$[V_{-2}, M_{(i)}]\subseteq M_{(1)}$ for $i=1,2$. It follows that $0\ne V_{-2}\subset \n_\g({\rm nil}(M))=M$ which is impossible. This enables us to conclude that the case where $S\cong W(r;\underline{d})$ cannot occur.
\subsection{}\label{3.15}
Suppose $S$ is of special Cartan type. Then it follows from  \cite[6.3]{Str04} that  $S$ is one of
$$S(r;\underline{d})^{(1)},\,\,S(r;\underline{d};\Phi(\tau))^{(1)},\,\,S(r;\underline{d};\Phi(l))^{(1)},$$ where $r\ge 3$ and $1\le l\le r$. The Lie algebras in this list have dimensions
$(r-1)(p^{|\underline{d}|}-1)$,
$(r-1)(p^{|\underline{d}|}-1)$ and $(r-1)p^{|\underline{d}|}$, respectively. As a consequence, $\dim S\ge 2(p^3-1)\ge 248$. Since $\dim\g\ge\dim S$ and $p>5$ if $\g$ is of type ${\rm E}_8$, we see that this case is impossible.
\subsection{}\label{3.16}
Suppose $S$ is a Hamiltonian Lie algebra, that is $S\cong H(2r;\underline{d};\Phi)^{(2)}$, where $r\ge 1$. This case is more complicated because several entirely different situations may occur here. We first note that $\dim S\ge \dim H(2r;\underline{d})^{(2)}$ because $H(2r;\underline{d};\Phi)^{(2)}$ is a filtered deformation of a graded Lie algebra sandwiched between the {\it graded} Hamiltonian algebras $H(2r;\underline{d})^{(2)}$ and $CH(2r;\underline{d})$; see \cite[Theorem~6.1.2]{Str04}, for instance. Hence
$\dim S\ge \dim H(2r;\underline{1})^{(2)}=
p^{2r}-2$. If $r\ge 2$ then $248\ge \dim \g\ge \dim S\ge 5^4-2$, a contradiction. In view of \cite[Theorem~6.3.10]{Str04} this shows that $S$ is  one of
$$H(2;\underline{d})^{(2)},\,\,
H(2;\underline{d};\Phi(\tau))^{(1)},\,\,H(2;\underline{d};\Phi(l)),\qquad l=1,2,$$
where $\underline{d}\in\{(1,1), (2,1), (1,2)\}$. Besides,
$H(2;(1,2);\Phi)^{(2)}\,\cong\,
H(2;(2,1);\Phi)^{(2)}$ unless $\Phi=\Phi(l)$.
\subsection{}\label{3.17}
We first assume that $A(\bar{\G})=S\otimes O(m;\underline{n})$ and $m\ge 1$. Then $\dim \G>(\dim S)p^{|\underline{n}|}\ge (p^2-2)p$. From this it is immediate that $p=5$, $O(m;\underline{n})=O(1;\underline{1})$, $\g$ is of type ${\rm E}_7$, and $S$ is one of $H(2;\underline{1})^{(2)}$, $H(2;\underline{1},\Phi(\tau))^{(1)}$, $H(2;\underline{1}; \Phi(1))$.

Let $\mathcal{M}(S)$ denote the standard maximal subalgebra of the Cartan type Lie algebra $S$. By Kreknin's theorem, $\mathcal{M}(S)$ is invariant under all automorphisms of $S$; see \cite[Theorem~4.2.6]{Str04}. Since our grading of $S$ is induced by the action of a $1$-dimensional torus  $T_0\subset {\rm Aut}(S)$, this implies that $\mathcal{M}(S)\,=\,\bigoplus_{i\in\Z}\,\big(\mathcal{M}(S)\cap S_i\big)$.
Since $\mathcal{M}(S)$ is maximal in $S$, it cannot consist of nilpotent elements of the minimal $p$-envelope $S_p$ of $S$. Since each $S_i\cap \mathcal{M}(S)$ with $i\ne 0$ acts nilpotently on $S$, it follows from the Engel--Jacobson theorem that $S_0\cap \mathcal{M}(S)$ contains a  
non-nilpotent element of $S_p$. On the other hand, it is immediate from \cite[Lemma~7.1.1.(3)]{Str04}
that $\mathcal{M}(S)$ is a restricted
subalgebra of $S_p$.
It follows that so is $\mathcal{M}(S)\cap S_0\,=\,\mathcal{M}(S)^{T_0}$.
As a consequence, $\mathcal{M}(S)\cap S_0$ contains a nonzero toral element, $h_0$ say. It is straightforward to see that
$S_p\otimes 1$ coincides with the $p$-envelope of $S\otimes 1$ in $\Der\big(S\otimes O(1;\underline{1})\big)$.

Let $C_0=\c_{\bar{\G}_0}(h_0\otimes 1)$.
Clearly, $h_0\otimes O(1;\underline{1})\subset
C_0$.
Recall that
$$S_0\otimes O(1;\underline{1})\subset \bar{\G}_0\subseteq \big(\Der_0(S)\otimes O(1;\underline{1})\big)\rtimes \big({\rm Id}_S\otimes\mathcal{D}\big)$$ for some transitive Lie subalgebra $\mathcal{D}$ of $W(1;\underline{1})$. Let $\m=O(1;\underline{1})x_1$ denote the maximal ideal of $O(1;\underline{1})=O(x_1)$. Since $\ad_{\bar{\G}}(h_0\otimes 1)$ is semisimple and $\mathcal{D}$ is transitive, $C_0$
contains an element
$c=(d\otimes f)+({\rm Id}_S\otimes D)$ with $d\in\Der_0(S)$, $f\in O(1;\underline{1})$,
$D\in W(1;\underline{1})$ such that $D(x_1)\not\in \m$.  Since $c\in C_0$ it must be that $[d,h_0]=0$.
By construction, $[c,h_0\otimes x_1]=h_0\otimes D(x_1)$. Since $h_0^{[p]}=h_0$ and $D(x_1)^p$ is a {\it nonzero}
scalar in $O(1;\underline{1})$, the $p$-closure of
$[c,h_0\otimes x_1]$ in the restricted Lie algebra $\bar{\G}_0$ contains
$h_0\otimes 1$.

Since the canonical homomorphism $\eta\colon\,M_{(0)}\twoheadrightarrow \bar{\G}_0$ is restricted, there is a toral element $\tilde{h}_0\in M_{(0)}$ such that $\eta(\tilde{h}_0)=h_0\otimes 1$. Let $\tilde{C}_0=\c_{M}(\tilde{h}_0)$. Since
$\ad \tilde{h}_0$ is semisimple, we have that $\eta(\tilde{C}_0)=C_0$ and
$\tilde{C}_0\cap\ker\eta=\c_{M_{(1)}}(\tilde{h}_0)$. Since both $c$ and $h_0\otimes x_1$ are in $C_0$, the above discussion implies  that the $p$-closure of $[\tilde{C}_0,\tilde{C}_0]$
in $M$ contains an element of the form
$\tilde{h}_0+y$ for some $y\in \tilde{C}_0\cap M_{(1)}$. As $(\tilde{h}_0+y)^{[p]^N}=\tilde{h}_0^{[p]^N}=
\tilde{h}_0$ for $N\gg 0$, it follows that the
$p$-closure of the derived subalgebra of $\g_{\tilde{h}_0}$ contains $\tilde{h}_0$. Since it follows from Jacobson's formula that for any Levi subalgebra $\l$ of $\g$ the Lie algebra $[\l,\l]$ is is closed under taking $p$-th powers in $\g$, this yields $\tilde{h}_0\in[\g_{\tilde{h}_0},\g_{\tilde{h}_0}]$.

We now wish to estimate $\dim \g(\tilde{h}_0,i)$ for all $i\in \F_p^\times$. First we note that since
$\ad\,\tilde{h}_0$ is toral and preserves all components $M_{(j)}$ of our filtration, the equality
$\dim \g(\tilde{h}_0,i)=\dim \G(h_0\otimes 1,i)$ must hold for all $i\in\F_p$. Next we observe that for any $k\ge 1$ the factor space $N(\G)^k/N(\G)^{k+1}$ is a graded $\bar{\G}$-module. This yields
\begin{eqnarray}\label{dim-gr}\qquad \dim \g(\tilde{h}_0,i)=\dim\bar{\G}(h_0\otimes 1,i)+
\textstyle{\sum}_{k\ge 1}\,\dim \big(N(\G)^k/N(\G)^{k+1}\big)(h_0\otimes 1,i).
\end{eqnarray}
Clearly, each $N(\G)^k/N(\G)^{k+1}$ is a graded $\bar{\G}$-module.
Since $\bar{\G}$ is semisimple and $\G$ is a restricted $(\ad\,\G_0)$-module, applying Lemma~\ref{graded} shows that all composition factors of the $(S\otimes 1)$-modules $\bar{\G}$ and $N(\G)^k/N(\G)^{k+1}$ with $k\ge 1$ are restrictable.

If $S$ is one of $H(2;\underline{1})^{(2)}$ or $H(2;\underline{1};\Phi(\tau))^{(1)}$ then any non-trivial irreducible restricted $S_p$-module $V$ is either isomorphic to $S$ (the adjoint $(\ad\,S)$-module) or is induced from an irreducible restricted $\mathcal{M}(S)$-module $V_0$; see \cite[\S\,5]{FSW}.
Note that $\mathcal{M}(S)/{\rm nil}(\mathcal{M}(S))$ is isomorphic to $\sl_2$ and $S/\mathcal{M}(S)$ is a
$2$-dimensional irreducible module over
$\mathcal{M}(S)/{\rm nil}(\mathcal{M}(S))$. Since $h_0\in \mathcal{M}(S)$ this implies that if $V$ is induced from $V_0$ then $\dim V(h_0,i)=p(\dim V_0)$ for all $i\in \F_p$. If $V\cong S$ then one derives that $\dim V(h_0,i)=p$ for all $i\in\F_p^\times$ by passing from the filtered Cartan type algebra $S$ to the corresponding graded Lie algebra ${\rm gr}(S)\supseteq H(2;\underline{1})^{(2)}$.
If $S= H(2;\underline{1};\Phi(1))$ then the description in \cite[10.4]{Str09} shows that $h_0\in S_0$ is contained in a $2$-dimensional torus of $S_p$.
Thanks to \cite[Lemmas~4.7.2]{BW} this yields that 
for any non-trivial restricted $S_p$-module $V$ all eigenspaces
$V(h_0,i)$ with $i\in\F_p^\times$ have the same dimension divisible by $p$. In conjunction with (\ref{dim-gr}) this entails that  in all cases
$\dim \g(\tilde{h}_0,i)=\dim \g(\tilde{h}_0,j)$ for all $i,j\in\F_p^\times$.

As a result, $\tilde{h}_0$ is a $p$-balanced toral element of $\g$.
Since in the present case $p=5$ and $G$ is of type ${\rm E}_7$, applying Proposition~\ref{balanced}(ii) yields that the Levi subalgebra $\g_{\tilde{h}_0}$ has type ${\rm D}_4{\rm A}_1$.
But then the derived subalgebra of $\g_{\tilde{h}_0}$ is semisimple and hence cannot contain its central element $\tilde{h}_0$. This contradiction shows that the case where $m\ge 1$ cannot occur.
\subsection{}\label{3.18}
In the next five subsections we assume that $A(\bar{\G})=S\cong H(2;\underline{1})^{(2)}$. In this case we may regard $\bar{\G}$ as a graded subalgebra of $\Der\big(H(2;\underline{1})^{(2)}\big)$.
Recall that $S$ is spanned by all $$D_H(f)\,:=\,\partial_1(f)\partial_2-
\partial_2(f)\partial_1$$ where $f\in O(2;\underline{1})$ is a linear combination of $x_1^{a_1}x_2^{a_2}$ with $0<a_1+a_2<2(p-1)$, and the standard maximal subalgebra $\mathcal{M}(S)$ has a basis consisting of all
$D_H(x_1^{a_2}x_2^{a_2})$ with $2\le a_1+a_2<2(p-1)$.
By \cite[Theorem~3.3]{PS2}, every grading of $\Der
\big(H(2;\underline{1})^{(2)}\big)$ is induced by a suitable $(a_1,a_2)$-grading of $W(2;\underline{1})$ (which contains the derivation algebra of $H(2;\underline{1})^{(2)}$).
Since $\bar{\G}_{-1}$ is an irreducible $\G_{0}$-module, no generality will be lost by assuming that $(a_1,a_2)\in\big\{(-1,-1), (-1,0), (1,0), (1,1)\big\}$; see \cite[Corollary~3.4(4)]{PS2}.
If $a_1=a_2=-1$ then $\G_i=0$ for $i\ge 2$, $\bar{\G}_1=
\Bbbk D_H(x_1)\oplus\Bbbk D_H(x_2)=\Bbbk \del_2\oplus \Bbbk\del_1$, and $$\sl_2\cong S_0\subseteq \bar{\G}_0\subseteq \gl_2$$ acts irreducibly on $\bar{\G}_1$.
From this it is immediate that $M_{(1)}\cong \bar{\G}_1$ and $\c_\g(M_{(1)})$ coincides with $M_{(1)}$ (recall that $\c_\g(M_{(1)})\subseteq M_{(0)}$ by the maximality of $M$).
Since $p\ge 5$, every nonzero element $D \in \Bbbk \del_1\oplus \Bbbk \del_2$ has the property that $(\ad D)^4(\bar{\G})\ne 0$. Arguing as before it is now easy to observe that this contradicts  Corollary~\ref{max-root}.
So the case where $(a_1,a_2)=(-1,-1)$ is impossible.
\subsection{}\label{3.19}
 If $(a_1,a_2)=(-1,0)$ then \cite[Corollary~3.4]{PS2} implies that $\bar{\G}_i=0$ for $i\ge 2$ and $\bar{\G}_0$ is sandwiched between $S_0=\bigoplus_{i=0}^{p-1}\,
 \Bbbk (ix_2^{i-1}x_1\del_1-x_2^i\del_2)\cong W(1;\underline{1})$ and $S_0\oplus\Bbbk z$ where $z=x_1\del_1$ acts on $\bar{\G}_1$ as a scalar operator. Also, $S_1=\bigoplus_{i=0}^{p-2}\,\Bbbk x_2^i\del_1\cong O(1;\underline{1})/\Bbbk 1$ has codimension $\le 1$ in $\bar{\G}_1\subseteq S_1\oplus \Bbbk x_2^{p-1}\del_1$ which is indecomposable as a $\Bbbk \del_2$-module. We
 now see that as in the previous case $M_{(1)}\cong\bar{\G}_1$ is abelian and $\c_\g(M_{(1)})=M_{(1)}$. We pick an element $v\in M_{(0)}$ which maps onto $\del_2\in S_0$ under the canonical homomorphism $M_{(0)}\to \G_{(0)}=M_{(0)}/M_{(1)}$ and set $R=\Bbbk v$. Since
 $\c_{\bar{\G}_1}(\del_2)=\Bbbk \del_1$ it follows from Corollary~\ref{max-root} that there exists an element $e\in\OO_{\rm min}\cap M_{(1)}$ which maps
 onto $\del_1$ under the epimorphism $M_{(1)}\twoheadrightarrow \G_{(1)}$. As $(\ad e)^3=0$ this contradicts the fact that $(\ad \del_2)^4\big(H(2;\underline{1})^{(2)}\big)\ne 0$. Hence the case where $(a_1,a_2)=(-1,0)$ cannot occur either.
 \subsection{}\label{3.20}
Now suppose $(a_1,a_2)=(1,0)$. Then
\cite[ Corollary~3.4(2)]{PS2} implies that $S=\bigoplus_{i=-1}^{p-2}\,S_i$ and $\bar{\G}=\bigoplus_{i=-1}^r\,\bar{\G}_i$ where
$r=p-2$ or $r=p-1$. Moreover, if $r=p-1$ then $\dim\G_r=\dim \bar{\G}_r=1$. Since
$\dim \G_r=\dim M_{(r)}$ the equality $r=p-1$ would mean that $M$ normalises a $1$-dimensional subspace $\Bbbk e\subset\N(\g)$.  As the latter would entail that $M\subseteq \n_\g(\Bbbk e)$ lies in the optimal parabolic subalgebra of $e$, it must be that $r=p-2$.

Since $\bar{\G}$ is a graded Lie algebra, the centre of $\bar{\G}_{>0}$ is a graded subalgebra of $\bar{\G}$.
Since $S\cong \ad S$ coincides with the unique minimal ideal of $\bar{\G}\subseteq \Der(S)$ and $S_{p-2}\ne 0$, this forces
$\z(\bar{\G}_{>0})=\bar{\G}_{p-2}$. From this it is immediate that
$\c_\g({\rm nil}(M))\,=\,\z(M_{(1)})\,=\,M_{(p-2)}$.

Applying \cite[Corollary~3.4(2)]{PS2} it is straightforward to see that
the endomorphism $\ad D_H(x_1)=\del_2\in \ad S_0$ acts on
$\G_{p-2}=\bar{\G}_{p-2}$ as a single Jordan block
and annihilates $D_H(x_1^{p-1})=-x_1^{p-2}\del_2\in S_{p-2}$.
Pick $e\in M_{(p-2)}$ which maps onto $D_H(x_1^{p-1})$
under the canonical epimorphism $M_{(p-2)}\twoheadrightarrow \bar{\G}_{p-2}$. Using Corollary~\ref{max-root} and arguing as in (\ref{3.19}) we now observe that $e\in\OO_{\rm min}$.

Let $E$ be a non-trivial restricted irreducible $S$-module.
We have already mentioned in (\ref{3.17}) that either $E\cong S$, the adjoint $S$-module, or
$E$ is induced from an
irreducible restricted module $E_0$ over $\mathcal{M}(S)/
{\rm nil}(\mathcal{M}(S))\cong\sl_2$.
Let $\rho\colon S\rightarrow \gl(E)$ denote the corresponding representation.  If $\rho=\ad$ then
$\rho\big(D_H(x_1^{p-1})\big)$ is injective on the span of $\big\{D_H(x_2^i),\,D_H(x_1x_2^i)\,|\,\,1\le i\le p-1\big\}$ implying that $\dim\,{\rm Im}\,\rho\big(D_H(x_1^{p-1})\big)\ge 2(p-1)$.

Now suppose $E\cong u(S)\otimes_{u(\mathcal{M}(S))}\,E_0$
and write $x\cdot v$ for $(\rho(x))(v)$ where $x\in S$ and $v\in E$.
Let $w \in E$ be such that $D_H(x_1^k)\cdot w=0$ for $3\le k\le p-1$. Recall that
$D_H(x_1^k)=kx_1^{k-1}\partial_2$. Also, $x_1\partial_2=\frac{1}{2}D_H(x_1^2)\in \mathcal{M}(S)$ acts nilpotently on the subspace $E_0$ of $E$ and commutes with $\partial_2$.
We wish to determine 
$D_H(x_1^{p-1})\cdot \partial_1^{p-2}\cdot w$. As
$D_H(x_1^{p-1})=-x_1^{p-2}\del_1$ we require an inductive formula for $(x_1^k\partial_2)\cdot \partial_1^k\cdot w$ valid for $k\in\{1,\ldots, p-2\}$.
Clearly, $(x_1\partial_2)\cdot\partial_1\cdot w=
-\partial_2\cdot w+\partial_1\cdot (x_1\partial_2)(w)$.
Suppose the equality
\begin{equation}\label{parr}
(x_1^k\partial_2)\cdot \partial_1^k\cdot w\,=\,
k!(-1)^k\partial_2\cdot w+\lambda_k\partial_1\cdot (x_1\partial_2)(w)
\end{equation}
holds for some  $\lambda_k\in \Bbbk$
(this is true for $k=1$ with $\lambda_1=1$).
Then
\begin{eqnarray*}
(x_1^{k+1}\partial_2)\cdot\partial_1^{k+1}\cdot w&=&
[x_1^{k+1}\partial_2,\partial_1]\cdot\partial_1^k\cdot w
+\partial_1\cdot(x_1^{k+1}\partial_2)\cdot\partial_1^k\cdot w\\
&=&-(k+1)(x_1^k\partial_2)\cdot\partial_1^k\cdot w+(k+1)!(-1)^k\partial_1\cdot(x_1\partial_2)(w)\\
&=&(k+1)!(-1)^{k+1}\partial_2\cdot w+\big((k+1)!(-1)^k-(k+1)\lambda_k\big)\partial_1\cdot(x_1\partial_2)(w).
\end{eqnarray*}
This shows that (\ref{parr}) holds for all $k\in\{1,\ldots, p-2\}$.  Since $(p-2)!=1$ in $\Bbbk$ we get
\begin{equation}\label{par-k} D_H(x_1^{p-1})\cdot\partial_1^{p-2}\cdot w=
\partial_2\cdot w- \lambda_{p-2}\partial_1\cdot(x_1\partial_2)(w).
\end{equation}
Next we show that the equality
\begin{equation}\label{par-kk} (x_1^k\partial_2)\cdot \partial_1^{k+1}\cdot w=(k+1)!(-1)^k\partial_1\partial_2\cdot w+\mu_k\partial_1^2\cdot(x_1\partial_2)(w)\ \mbox{ with }\ \mu_k\in\Bbbk\end{equation}
holds for all $k\in\{1,\ldots,p-2\}$.
Since $$(x_1\partial_2)\cdot\partial_1^2\cdot w=[x_1\partial_2,\partial_1]\cdot\partial_1\cdot w +\partial_1\cdot[x_1\partial_2,\partial_1]\cdot w+\partial_1^2\cdot(x_1\partial_2)(w)=-2\partial_1\partial_2\cdot w+\partial_1^2\cdot(x_1\partial_2)(w)$$ the statement holds for $k=1$. If it holds for $k=m$ then using (\ref{parr}) we get
\begin{eqnarray*}
(x_1^{m+1}\partial_2)\cdot\partial_1^{m+2}\cdot w
&=&\partial_1\cdot(x_1^{m+1}\partial_2)\cdot \partial_1^{m+1}\cdot w+[x_1^{m+1}\partial_2,\partial_1]\cdot\partial_1^{m+1}\cdot w\\
&=&(m+1)!(-1)^{m+1}\partial_1\partial_2\cdot w
+\lambda_m\partial_1^2\cdot(x_1\partial_2)(w)-
(m+1)(x_1^m\partial_2)\cdot\partial_1^{m+1}\cdot w\\
&=&\big((m+1)!(-1)^{m+1}-(m+1)!(-1)^m(m+1)\big)\cdot \partial_1\partial_2\cdot w\\
&+&
\big(\lambda_m-(m+1)\mu_m\big)\partial_1^2\cdot(x_1\partial_2)(w)\\
&=&(m+2)!(-1)^{m+1}\partial_1\partial_2\cdot w+\mu_{m+1}\partial_1^2\cdot(x_1\partial_2)(w).
\end{eqnarray*}
This proves (\ref{par-kk}). As a consequence,
\begin{equation}\label{parrr}
D_H(x_1^{p-1})\cdot\partial_1^{p-1}\cdot w=
-\partial_1\partial_2\cdot w- \mu_{p-2}\partial_1^2\cdot(x_1\partial_2)(w).
\end{equation}
Note that we can take for $w$ any vector of the form $\partial_2^{i}\otimes v$ with
$v\in E_0$ and $0\le i\le p-2$.
Applying (\ref{par-k}) and (\ref{parrr}) it is now straightforward to see that
$\rho\big(D_H(x_1^{p-1})\big)$ is injective on the subspace $\bigoplus_{i=0}^{p-2}\big((\del_1^{p-1}\del_2^i\otimes E_0)\oplus(\del_1^{p-2}\del_2^i\otimes E_0)\big)$ of $E$.
It follows that ${\rm Im}\,\rho\big(D_H(x_1^{p-1})\big)$ has dimension $\ge 2(p-1)(\dim E)/p^2$.

Let ${\rm gr}(\G)$ denote the $S$-module $\bar{\G}\oplus \sum_{k\ge 1}N(\G)^k/N(\G)^{k+1}$
and write $\varphi$ for the corresponding representation of $S$. Let $E_1,\ldots,E_r$ be all induced composition factors of the $S$-module
${\rm gr}(\G)$ and write $l$ for the multiplicity of $S$ in ${\rm gr}(\G)$. Put
$$s:=l+\textstyle{\sum}_{i=1}^r\,(\dim E_i)/p^2.$$
Repeating almost verbatim the argument used in (\ref{3.17}) one observes that $\g$ contains a $p$-balanced toral element $\tilde{h}_0$ such that $\dim\g(\tilde{h}_0,i)=ps$.
for all $i\in\F_p^\times$.
On the other hand, since $D_H(x_1^{p-1})={\rm gr}(e)$ for some $e\in M_{(p-2)}\cap \OO_{\rm min}$, the above discussion yields
\begin{equation}\label{image}
\dim\OO_{\rm min}=\,\dim \big({\rm Im}\,\ad e\big)\,\ge\, \dim\big({\rm Im}\,\varphi\big(D_H(x_1^{p-1})\big)\,\ge\,
2(p-1)s.
\end{equation}
If $G$ is of type ${\rm E}_6$ then Proposition~\ref{balanced}(i) shows that $p=5$ and $s=3$. But then $2(p-1)s=24>22=\dim\OO_{\rm min}$.
As this contradicts (\ref{image}) this case cannot occur. If $\g$ is of type ${\rm E}_7$ then Proposition~\ref{balanced}(ii) yields that $p=5$ and $s=5$.
Since $\dim \OO_{\rm min}=34<40=2(p-1)s$ this case cannot occur either.
\subsection{}\label{3.21}
Suppose $G$ is of type ${\rm E}_8$. Then
Proposition~\ref{balanced} says that $p=7$ or $p=11$. Moreover, if $p=7$ then
$\g_{\tilde{h}_0}$ has dimension $80$ or $38$. If $\dim\g_{\tilde{h}_0}=38$ then $s=5$ and $\dim\OO_{\rm min}=58<60=2(p-1)s$ which violates (\ref{image}).

Suppose $p=7$ and $\dim\g_{\tilde{h}_0}=80$. Then $s=4$ and part~(f) of
the proof of Proposition~\ref{balanced}
shows that no generality will be lost by assuming that
 $\widetilde{h}_0=2t_7+2t_8$.
 In this case, there exists an element $e\in\OO({\rm D}_4)$ which admits an optimal
cocharacter $\lambda\colon\,\Bbbk^\times\rightarrow
G$ with the property that $e\in\g(\lambda_e,2)$ and $({\rm d}\lambda)(1)=\widetilde{h}_0$. Moreover, it follows from (\ref{e8-7-2})
and the preceding discussion that the weight space $\g(\lambda,2)$ coincides with the $2$-eigenspace of $\ad \tilde{h}_0$. In view of \cite[Theorem~2.3]{P03} this implies that the subset
$\g(\tilde{h}_0,2)\cap \OO({\rm D}_4)$ is Zariski dense in
$\g(\tilde{h}_0,2)$. As $\OO({\rm D}_4)\subset \N_p(\g)$ by \cite[Table 9]{Law} and \cite[Theorem~4.1]{PSt} we now deduce that $x^{[p]}=0$ for all $x\in\g(\tilde{h}_0,2)$.
We may assume that the image of $\widetilde{h}_0$ in $\G_0$ equals $2D_H(x_1x_2)$.

As $[2D_H(x_1x_2),D_H(x_2)]=2D_H(x_2)$ we can find an element
$u\in \g(\tilde{h}_0,2)\cap M_{(0)}$ which maps onto $D_H(x_2)=-\del_1$ under the epimorphism
$M_{(0)}\twoheadrightarrow \G_0$.
One checks directly that ${\rm Im}\big(\ad_S D_H(x_2)^{p-1}\big)$ has dimension $p-1$.
Since $D_H(x_2)\not\in S_{(0)}$, it is straightforward to see for any irreducible restricted representation $\rho\colon\,S\to \gl(E)$ such
that $E= u(S)\otimes_{u(S_{(0)})}\,E_0$ as $S$-modules,
$\rho\big(D_H(x_2)\big)$ has $p\dim E_0$ Jordan blocks of size $p$.
Since $u^{[p]}=0$ by
our earlier remark, the definition of $s$ in (\ref{3.20}) shows that $\ad  u$ has at least
$4(p-1)=24$
Jordan blocks of size $p=7$.

As $\dim\, (\Ad G)\, u\le \dim\, (\Ad G)\, e=168$,
combining \cite[Table~9]{Law} with \cite[Theorem~4.1]{PSt} now yields $u\in\OO({\rm D}_4)$.
Note that $D_H(x_1^k x_2^{p-1})\in S_{k-1}$ for $1\le k\le p-2$. Since  $(\ad u)^p=0$ and $$\big(\ad D_H(x_2)\big)^{p-1}\big(D_H(x_1^kx_2^{p-1})\big)\ne 0,\qquad
 \big[2D_H(x_1x_2),D_H(x_1^kx_2^{p-1})\big]=
-2(k+1)D_{H}(x_1^kx_2^{p-1}),$$
it follows that $\g_u\cap \g(\tilde{h}_0,2j)\ne 0$
for five values of $j\in\F_p^\times$. This, however, contradicts \cite[p.~131]{LT11}. Therefore, the case where $p=7$ cannot occur.

If $p=11$ then $[M_{(5)}, M_{(5)}]\subseteq M_{(10)}=M_{(p-1)}=0$ and $$\dim M_{(5)}\,=\,\textstyle{\sum} _{i=5}^9\, \G_i\ge \dim\textstyle{\sum} _{i=5}^9\, S_i\,=\,4p+(p-1)=54.$$
Hence $\g$ contains an abelian Lie subalgebra of dimension $54$. The set of all $d$-dimensional abelian Lie subalgebras of $\g$ is a closed subset of the Grassmannian ${\rm Gr}(d,\g)$ invariant under the action of a maximal torus $T$ of $G$. So it follows from the Borel fixed-point theorem that the root system $\Phi=\Phi(G,T)$ contains an {\it abelian} subset $\mathcal A$ of size $54$
(it has the property that $\beta+\gamma\not\in\Phi$ for all $\beta,\gamma\in \mathcal{A}$). However, an old result of Malcev
says that in type ${\rm E}_8$ the size of
any abelian subset of $\Phi$ cannot be bigger than $36$; see \cite{Mal}. It is not hard to see that this result of Malcev is still valid in our situation; see \cite{PeS} for detail. Thus $G$ is not of type ${\rm E}_8$.

If $G$ is of type ${\rm F}_4$ then $\g$ (and hence $M$) acts faithfully on a $26$-dimensional irreducible 
$G$-module $V$ of highest weight $\varpi_4$. It is well known that the restricted Lie algebra $\g$ embeds into  a restricted Lie algebra $\widetilde{\g}$ of type ${\rm E}_6$ in such a way that $\widetilde{\g}=\g\oplus V$ as 
$\g$-modules and $\mathcal{O}_{\rm min}(\g)\subset\mathcal{O}_{\rm min}(\widetilde{\g})$. Let $e\in\OO_{\rm min}(\g)$ and denote by $\rho$ the representation of $\g$ in $\mathfrak{gl}(V)$. As
$\dim\, [e,\g]= \dim\OO_{\rm min}(\g)=16$ 
and $\dim\, [e,\widetilde{\g}]=\dim \OO_{\rm min}(\widetilde{\g})=22$ the above shows that $$\dim {\rm Im}\,\rho(e)=22-16=6.$$ Since $V$ is a restricted 
$M$-module, the 
subspace $V_{(0)}:=\{v\in V\,|\,\,\rho(M_{(1)})\cdot v=0\}$ is nonzero.
For $i>0$ we put 
$V_{(i)}:=\big(\rho(M_{(-1)}\big)^i\cdot V_{(0)}.$ As $V$ is an irreducible $\g$-module and $M_{(-1)}$ generates the Lie algebra $\g$ we obtain a finite filtration
$V=V_{(-r)}\supset\cdots\supset V_{(-1)}\supset V_{(0)}$ of $V$ with the property that $\rho(M_{(i)})\cdot V_{(j)}\subseteq V_{(i+j)}$ for all $i,j\in\Z$.
 Let ${\rm gr}(V)=\bigoplus_{i\ge -r}\, {\rm gr}_i (V)$ be the corresponding graded $\mathcal{G}$-module where, as usual,
${\rm gr}_i(V)=V_{(i)}/V_{(i+1)}$ and $V_{(1)}=0$ by convention. Since the restricted Lie algebra $M=M_{(0)}$ contains a nonzero toral element $\tilde{h}_0$ of $\g$, the $\mathcal{G}$-module ${\rm gr}(V)$ has at least one  non-trivial composition factor, say $W$. 

By construction, all elements of $\mathcal{G}_-$ acts nilpotently on ${\rm gr}(V)$. Therefore, the ideal $N(\mathcal{G})$ of $\mathcal{G}$ acts trivially on
$W$. Hence $W$ is a non-trivial irreducible
$\bar{\mathcal G}$-module. Since we can choose $\tilde{h}_0$ in such a way that $h_0:={\rm gr}_0(\tilde{h}_0)$ is a nonzero toral element of 
$S_0\subseteq \bar{\mathcal G}_0$, the $S$-module $W$ has a non-trivial composition factor, too; we call it $\overline{W}$. Let $\bar{\rho}$ denote the representation of $S$ in $\mathfrak{gl}(\overline{W})$. Recall that we can find $e\in \OO_{\rm min}\cap M_{(p-2)}$ whose image 
$\bar{e}$ in $\bar{G}_{p-2}$ coincides with $D_H(x_1^{p-1})$.
Our remarks earlier in this part now imply that
$\dim\,{\rm Im}\,(\bar{\rho}(\bar{e})\ge 2(p-1)$. 
But then $$6=\dim\, {\rm Im}\,\rho(e)\ge \dim\,{\rm Im}\,\bar{\rho}(\bar{e})\ge 2(p-1).$$ As $p\ge 5$ 
we reach a contradiction which shows that the case where $(a_1,a_2)=(1,0)$ is impossible.
\subsection{}\label{3.22} Finally, suppose  $(a_1,a_2)=(1,1)$, i.e. the grading of the Cartan type Lie algebra $S = H(2;\underline{1})^{(2)}$ is standard. Then $S_0\,=\,\Bbbk D_H(x_1^2)\oplus\Bbbk D_H(x_1x_2)\oplus \Bbbk D_H(x_2^2)$ is isomorphic to $\sl_2$ and $\bar{G}_{-1}=S_{-1}=\Bbbk D_H(x_1)\oplus\Bbbk D_H(x_2)$ is a $2$-dimensional irreducible $S_0$-module. Furthermore, $\bar{G}_{k}=0$ for $k\le -2$ and  $\G_1=S_1=\Bbbk D_H(x_1^3)\oplus\Bbbk D_H(x_1^2x_2)\oplus \Bbbk D_H(x_1x_2^2)\oplus\Bbbk D_H(x_2^3)$ is an irreducible $4$-dimensional $S_0$-module.

As before, we choose a toral element $\tilde{h}_0\in M_{(0)}$ which maps onto $D_H(x_1x_2)$ under the epimorphism $M_{(0)}\twoheadrightarrow \G_0$. Since the endomorphism $\ad \tilde{h}_0$ is semisimple, there exist $v_{-1,\pm 1}\in M_{(-1)}\cap \g(\tilde{h}_0,\pm 1)$  and $v_{0,\pm 2}\in M_{(0)}\cap \g(\tilde{h},\pm 2)$ such that $${\rm gr}_{-1}(v_{-1,-1})=D_H(x_1),\ \, {\rm gr}_{-1}(v_{-1,1})=D_H(x_2),\ \,{\rm gr}_{0}(v_{0,-2})=D_H(x_1^2),\ \,{\rm gr}_{0}(v_{0,2})=D_H(x_2^2).$$ Besides, there exist $v_{1,\pm 3}\in M_{(1)}\cap \g(\tilde{h}, \pm 3)$ and  $v_{1,\pm 1}\in M_{(1)}\cap \g(\tilde{h}, \pm 1)$ such that
$${\rm gr}_{1}(v_{1,-3})=D_H(x_1^3),\ \, {\rm gr}_{1}(v_{1,-1})=D_H(x_1x_2^2),\ \,{\rm gr}_{1}(v_{1,1})=D_H(x_1^2x_2),\ \,{\rm gr}_{1}(v_{1,3})=D_H(x_2^3).$$ Since $p>3$ we have that
$\G_2=S_2=[S_1,S_1]$. Our earlier remarks show that
\begin{equation}\label{H}
[v_{0,-2},v_{1,3}]\equiv b\, v_{1,1}\mod M_{(2)}\quad \mbox{for some }\, b\in\Bbbk^\times.
\end{equation}
To unify notation we put $\tilde{h}_0=:v_{0,0}$.
If $N(\G)=0$ then $\g=M_{(-1)}$ contains
$\g_{v_{0,0}}$. Since
the latter contains a maximal torus of $\g$ the subalgebra $M$ is regular in $\g$. So from now on we may assume that  $N_2(\G)=[\G_{-1},\G_{-1}]\ne 0$. In this case $[\G_{-1},\G_{-1}]\cong\wedge^2\,\bar{\G}_{-1}$ is
a $1$-dimensional $S_0$-module and
$M_{(-2)}=
[M_{(-1)},M_{(-1)}]+M_{(-1)}=\Bbbk w\oplus M_{(-1)}$ where $w=[v_{-1,-1}, v_{-1,1}]\ne 0$.

We claim that $w\in \n_\g(M_{(1)})$. Since $w\in M_{(-2)}$ and $M_{(2)}=[M_{(1)},M_{(1)}]+M_{(3)}$ we
just need to show that $[w,v_{i,1}]\in M_{(1)}$ for   $i\in\{\pm 1,\pm3\}$.
Since $w\in M_{(-2)}\cap \g_{v_{0,0}}$ and $$M_{(-2)}\subseteq \g_{v_{0,0}}+\g(v_{0,0},1)+\g(v_{0,0},-1)+M_{(0)}$$ we have
$[w,v_{0,\pm 2}]\in M_{(0)}$. As $[w,v_{1,3}]\in
\g(v_{0,0}, 3)\cap M_{(-1)}$  and $M_{(-1)}\subseteq \g(v_{0,0},1)+\g(v_{0,0},-1)+M_{(0)}$ it must be that $[w,v_{1,3}]\in M_{(0)}$. Hence
$[v_{0,-2},[w,v_{1,3}]]\in M_{(0)}$. But then (\ref{H}) yields
\begin{eqnarray*}
M_{(0)}&\ni& [v_{0,-2},[w,v_{1,3}]]=[[v_{0,-2},w], v_{1,3}]+[w,[v_{0,-2},v_{1,3}]]\\
&\in& [M_{(0)},M_{(1)}]+b[w,v_{1,1}]+[w,M_{(2)}]
\subseteq\, b[w,v_{1,1}]+M_{(0)}.
\end{eqnarray*}
As $b\ne 0$ and $M_{(0)}\cap \g(v_{0,0},1)\subseteq M_{(1)}$ this yields $[w,v_{1,1}]\in M_{(1)}$.
So $[w,[v_{0,\pm 2},v_{1,1}]]\in
[M_{(0)},v_{1,1}]+[v_{0,\pm 2},M_{(1)}]\subseteq M_{(1)}$ forcing $[w,v_{1,3}],\,[w,v_{1,1}]\in M_{(1)}$. Finally, $[w,[v_{0,-2}, v_{1,-1}]]\in
[M_{(0)},v_{1,-1}]+[v_{0,-2},M_{(1)}]\subseteq M_{(1)}$. Hence $[w,v_{1,-3}]\in M_{(1)}$  proving the claim.
As a result, $w\in \n_\g(M_{(1)})$ contrary to the maximality of $M$. We now conclude that the case where $S=H(2;\underline{1})^{(2)}$ is impossible.
\subsection{}\label{3.23} Suppose that $S\cong H(2;\underline{1};\Phi)^{(1)}$ where $\Phi$  is one of $\Phi(\tau)$, $\Phi(1)$, $\Phi(2)$.
Then ${\rm Der}(S)=S_p$; see \cite[Theorem~7.2.2]{Str04}. The grading of $S$ gives rise to that of ${\rm  Der}(S)$. Identifying $S$ with $\ad S$ we have that
$S_i^{[p]}\subseteq  {\rm Der}_{pi}(S)$ for all $i\in \Z$. Then it follows from Jacobson's formula for $p$-th powers that
\begin{equation}\label{p-th}
{\rm Der}_0(S)\,=\,\big(\textstyle{\sum}_{i\in\Z,\,j\ge 0}\,S_i^{[p]^j}\big)_0\,=\,
\textstyle{\sum}_{i\ge 0}\,S_0^{[p]^i}.
\end{equation}
As $\G_0\cong\bar{\G}_0\subseteq {\rm Der}_0(S)$
this implies that $S_{-1}=\G_{-1}$ is an irreducible $S_0$-module.

If $\Phi=\Phi(\tau)$ then $\omega=\Phi(\omega_S)=(1+x_1^{p-1}x_2^{p-1}){\rm d}x_1\wedge{\rm d}x_2$ and $S$ is contained in the $\Bbbk$-span of all $D_{H,\omega}(f)$ with $f\in O(2;\underline{1})$; see \cite[Theorem~6.5.7(2)]{Str04}. Thanks to
(the proof of)
Lemma~\ref{non-grad} we may assume  that the grading of $S$ is induced by the action of the torus
$T_\Phi=\{g(t,t^{-1})\,|\,\,t\in \Bbbk^\times\}$. Then $S_0$ is contained in the $\Bbbk$-span of all $D_{H,\omega}(x_1^kx_2^k)$ with
$k\in\{1,\ldots, p-1\}$. Using \cite[(6.5.5)]{Str04} it is then straightforward to check that $S_0$ is abelian. The irreducibility of the $S_0$-module $S_{-1}$ now gives $\dim \G_{-1}=\dim S_{-1}=1$.
But then $\G_{-1}$ is not a faithful $\G_0$-module.
This contradiction shows that $S\not\cong H(2;\underline{1};\Phi(\tau))^{(1)}$.

Suppose $\Phi=\Phi(l)$. Since in the present case $\underline{n}=\underline{1}$ we may assume that $l=1$.
Then $\omega=\Phi(\omega_S)=\exp(x_1^{(p)}){\rm d}x_1\wedge{\rm d}x_2$. Using (\ref{p-th}) we observe that $\G_0\cong \bar{\G}_0$ identifies with the $p$-closure of $S_0$ in ${\rm Der}_0(S)$. By the proof of Lemma~\ref{non-grad}, we may assume that the grading of $S$ is induced by the action of the torus $T_\Phi=\{g(1,t)\,|\,\,t\in\Bbbk^\times\}$. In this case, $S_0$ contains $D_{H,\omega}(x_2)=-\partial_1-x_1^{p-1}x_2\partial_2$. Since $\bar{\G}_0$ is a restricted subalgebra of ${\rm Der}(S)$ we then have $x_2\del_2=D_{H,\omega}(x_2)^{[p]}\in \bar{\G}_0$; see \cite[p.~45]{Str09}. This means that $\G_0\cong \bar{\G}_0$ contains the degree derivation of $\G$.
Since $M_{0}$ is a restricted subalgebra of $\g$ we can find a toral element $t_0\in M_{(0)}$ which maps onto
the degree derivation $x_2\partial_2$ under the canonical homomorphism $M_{(0)}\twoheadrightarrow
\bar{\G}_0$.

From \cite[10.4]{Str09} we know that $S_p=S\oplus \Bbbk(x_2\del_2)$ and $\bar{\G}_i=S_i$ for all $i\ne 0$.
Since $S_{-1}\ne 0$ there exists $\varepsilon\in\{\pm 1\}$ such that $S_i=\{x\in S\,|\,\,(g(1,t))
\cdot x=t^{\varepsilon i}x\}$ for all $i\in\Z$.
As $S$ is spanned by all $D_{H,\omega}(f)$ with $f\in O(1;\underline{1})$ by \cite[Theorem~6.5.7(2)]{Str04}, this shows that $\bar{\G}_i=S_i=0$ for all $i\le -p$.
If $\varepsilon=-1$ then $S_i=0$ for $i>1$. So we can
repeat verbatim the argument used in the last two paragraphs of
(\ref{3.8}) to obtain that $N(\G)=0$.
If $\varepsilon=1$ then $S_i=0$ for $i>1$ and we reach
the same conclusion by arguing as in (\ref{3.13}). In any event,
$\g_{t_0}\subset M$ which implies that $M$ is a regular subalgebra of $\g$. Since this contradicts our assumption on $M$ we deduce that the case $S\cong H(2;\underline{1};\Phi)^{(1)}$ is impossible.

\subsection{}\label{3.24} Suppose $S\cong H(2;(2,1);\Phi)^{(2)}$. Then $\dim S\ge p^3-2$ which implies that $G$ is of type ${\rm E}_7$ and $p=5$.
Therefore, $$\dim N(\G)=\dim \g-\dim \bar{\G}\le \dim\g-\dim S\le 133-123=10.$$
From this it is immediate that $S=S^{(\infty)}$ acts trivially on all $\bar{\G}$-modules $N(\G)^k/N(\G)^{k+1}$ with $k\ge 1$.
Let $M'$ be the inverse image of $S_0$ under the canonical homomorphism $M=M_{(0)}\twoheadrightarrow \bar{\G}_0$.

If $\dim M'>(\dim\g)/2$ then $\kappa_{\vert M'}\ne 0$. Applying
(\ref{kappa}) then yields that the restriction of  $\kappa_{\G}$ to $S_0$ is nonzero.
Since $S_0$ acts trivially on each $N(\G)^k/N(\G)^{k+1}$, it follows that the restriction of $\kappa_{\bar{\G}}$ to $S_0$ is nonzero as well.
Since $S$ is an ideal of $\G$, the restriction of $\kappa_{\bar{\G}}$ to $S$ coincides with the Killing form of $S$. In view of the above this means that $\kappa_S\ne 0$. But then $\kappa_S$  is non-degenerate by the simplicity of $S$. Since $S$ is strongly degenerate, i.e. contains a nonzero element $c$ with $(\ad c)^2=0$, this is impossible; see \cite[Lemma~4.4]{PS3}, for example.
It follows that
\begin{equation}\label{1/2}\dim S_{\ge 0}\,\le\, \dim M'\le (\dim \g)/2.\end{equation}
If $S=H(2;(2,1))^{(2)}$ we may assume without loss of generality that the grading of $S$ has type $(a_1,a_2)$ with respect to the standard generators $x_1,x_2$ of $O(2;(2,1))$; see \cite[Theorem~4.7]{PS3}.
Since $S_{-1}$ is an irreducible
${\rm Der}_0(S)$-module, the proof of \cite[Lemma~4.18]{PS3} shows that $(a_1,a_2)$ is one of
$(-1,-1)$, $(-1,0)$, $(0,-1)$, $(0,1)$, $(1,0)$, $(1,1)$. The description in {\it loc.\,cit.} shows that in the last three cases $S_{\ge 0}$ has codimension $< p^2=25$ in $S$, implying that $\dim S_{\ge 0}\ge 123-24=94>(\dim \g)/2$. In view of (\ref{1/2}) it follows that $(a_1,a_2)\in \big\{(-1,-1), (-1,0), (0,-1)\big\}$. Repeating almost verbatim the arguments used (\ref{3.18}) and (\ref{3.19}) we find, in each of the remaining cases, an element $x\in M_{(1)}\cap \OO_{\rm min}$ such that $(\ad {\rm gr}_1(x))^4\ne 0$. Since $(\ad x)^3=0$ for every $x\in\OO_{\rm min}$, we reach a contradiction. This shows that the case where
$S\cong H(2;(2,1))^{(2)}$ cannot occur.

 If $S\cong H(2;(2,1);\Phi(\tau))^{(1)}$ we argue as in (\ref{3.23}). Indeed, in this case
 ${\rm Der}(S)=S_p$ by \cite[Theorem~7.2.2(3)]{Str04} and
 $\omega=(1+x_1^{(p^2-1)}x_2^{(p-1)}){\rm d}x_1\wedge{\rm d}x_2$.
 It follows from (\ref{p-th}) that $S_{-1}$ is an irreducible $S_0$-module.
 By the proof of
Lemma~\ref{non-grad}, we may assume  that the grading of $S$ is induced by the action of the torus
$T_\Phi=\{g(t,t^{-p-1})\,|\,\,t\in \Bbbk^\times\}$ whilst \cite[Theorem~6.5.7(2)]{Str04} yields that
  $S$ is contained in the $\Bbbk$-span of all $D_{H,\omega}(f)$ with $f\in O(2;(2,1))$. Then  $S_0$ is contained in the $\Bbbk$-span of all $D_{H,\omega}\big(x_1^{(kp+k-p)}x_2^{(k)}\big)$ with
$k\in\{1,\ldots, p-1\}$. As a consequence, the nilradical of $S_0$ has codimension $1$ in $S_0$. As $S_{-1}$ is an irreducible $S_0$-module $S_{-1}$ this implies that $\dim \G_{-1}=\dim S_{-1}=1$.
But then $\G_{-1}$ is not faithful over $\G_0$, a contradiction.
Therefore, $S\not\cong H(2;(2,1);\Phi(\tau))^{(1)}$.

If $S\cong H(2;(2,1);\Phi(1))$ then $\omega=\exp\big(x_1^{(p^2)}\big){\rm d}x_1\wedge{\rm d}x_2$ and we may assume that the grading of $S$ is induced by the action of $T_\Phi=\{g(1,t)\,|\,\,t\in\Bbbk^\times\}$ (see the proof of Lemma~\ref{non-grad}).
Since $S_{-1}\ne 0$ there exists $\varepsilon\in\{\pm 1\}$ such that $S_i=\{x\in S\,|\,\,(g(1,t))
\cdot x=t^{\varepsilon i}x\}$ for all $i\in\Z$.
From \cite[Theorem~7.2.(4)]{Str04} we know that
${\rm Der}(S)=S_p$ whilst \cite[Theorem~7.1.3(2)]{Str04} implies that ${\rm Der}(S)$ is spanned by
$S\cong \ad S$ and the elements $x_2\partial_2$ and $\big(\partial_1-x_1^{(p^2-1)}x_2\partial_2\big)^p$ both of which have degree $0$. By \cite[Theorem~6.5.8]{Str04}, $S$ is spanned by the homogeneous elements
$$D_{H,1}(f)=D_{\omega}\big(\exp\big(x_1^{(p^2)}\big) \cdot f\big),\qquad\  f\in O(2;(2,1)).$$
This implies that if $\varepsilon=1$ then the graded subalgebra $S_{\ge 0}$
of ${\rm Der}(S)$ has codimension $p^2=25$ in $S$. As $\dim S_{\ge 0}=100>(\dim\g)/2$,
this violates (\ref{1/2}). So it must be that $\varepsilon=-1$.
But then the above description shows that $S_0\cong W(1;2)$ as Lie algebras, $\G_i=S_i=0$ for
$i>1$, and
$\G_{1}=S_1={\rm span}\,\big\{D_{H,1}(x_1^{(i)})\,|\,\,
0\le i\le p^2-1\big\}$ is an irreducible (non-restrictable) $S_0$-module of dimension $p^2$.
Using \cite[(6.5.8)]{Str04}
is it straightforward to check that
$D_{H,1}(x_1^{(i)})=x_1^{(i-1)}\partial_2$ for $1\le i\le p^2-1$ and $D_{H,1}(1)=x_1^{(p^2-1)}\del_2$.
Moreover, the element $u:=\partial_1-x_1^{(p^2-1)}x_2\partial_2\in S_0$ operates on $S_1$
as a single Jordan block of size $p^2$ with eigenvalue
$1$.
Put $v:=\sum_{i=0}^{p^2-1}\,x_1^{(i)}\partial_2$. Then
$v\in S_1$ and direct computations show that
$[u,v]=v$ and $(\ad v )^4\ne 0$. Since $\G_i=0$ for $i>1$, the ideal ${\rm nil}(M)=M_{(1)}$ is abelian. Since $\G_0$ acts faithfully on $S_1$, it coincides with $\c_\g ({\rm nil}(M))$. Let
$\pi\colon\,M_{(0)}\to \G_0$ be the canonical homomorphism, and pick any
$\tilde{u}\in \pi^{-1}(u)$. Let $R$ be the Lie subalgebra of $M$ generated by
$\tilde{u}$ and $v$ (here we identify $M_{(1)}$ with $S_1$). Since $[\tilde{u},v]=v$ and $v\in {\rm nil}(M)$, the above discussion in conjunction with Corollary~\ref{max-root} show that $v\in\OO_{\rm min}$.
But then $(\ad v)^3=0$. As this contradicts our choice of $v$ we conclude that $S\not\cong H(2;(2,1);\Phi(1))$.

The case $S\cong H(2;(2,1);\Phi(2))$ is quite similar, but shorter.
Here $\omega=\exp\big(x_2^{(p)}\big){\rm d}x_1\wedge{\rm d}x_2$ and the grading of $S$ is induced by the action of $T_\Phi=\{g(t,1)\,|\,\,t\in\Bbbk^\times\}$.
Arguing as before we observe that $\varepsilon=\pm 1$.
 If $\varepsilon=1$ then $S=\bigoplus_{i\ge -1}\,S_i$ and
$S_{\ge 0}$ has codimension $p=5$ in $S$. But then $\dim S_{\ge 0}=120>(\dim \g)/2$ violating (\ref{1/2}).
If $\varepsilon=-1$ then $S_i=0$ for $i>1$ and
$$S_{-1}={\rm span}\,\big\{D_\omega\big(\exp\big(x_2^{(p)}\big)\cdot x_1^{(2)}x_2^{(i)}\big)\,|\,\,0\le i\le p-1\big\}.$$
Using \cite[(6.5.8)]{Str04} is is straightforward to check that $S_{-1}\subset W(2;\underline{1})$. Since
the Lie algebra $S_{<0}$ is generated by $S_{-1}$,
it follows that $S_{1-p}\subset W(2;\underline{1})$
(here we regard both $S$ and $W(2;\underline{1})$ as Lie subalgebras of $W(2;(2,1))$).
However, $D_\omega\big(\exp\big(x_2^{(p)}\big)\cdot x_1^{(p)}\big)\in S_{1-p}\setminus W(2;\underline{1})$.
This contradiction finally shows that $S$ is not a Lie algebra of type $H$.
\subsection{}\label{3.25} If $S= K(3;\underline{1})$ then
$S\subset W(3;\underline{1})$  has basis $\big\{D_K(x_1^{a_1}x_2^{a_2}x_3^{a_3})\,|\,\,0\le a_i\le p-1\big\}$. In particular, $\dim S=p^3$. Hence this case may occur only if $p=5$ and $G$ is a group of type ${\rm E}_7$.
Recall that if $f\in O(3;\underline{1})$ then
$D_K(f)=f_1\partial_1+f_2\partial_2+f_3\partial_3$ where
$f_1=x_1\partial_3(f)-\partial_2(f)$, $f_2=x_2\partial_3(f)+\partial_1(f)$ and $f_3=2f-x_1\partial_1(f)-x_2\partial_2(f)$; see \cite[2.11]{BGP}, for example.

Repeating verbatim the arguments used at the beginning of (\ref{3.24}) we observe that
$\dim S_{\ge 0}\le (\dim \g)/2$.
As all derivations of $S$ are inner by
\cite[7.1.2(4)]{Str04}, the $S_0$-module $S_{-1}$ is faithful and irreducible. The Lie algebra $S$ has basis
$\big\{D_K(f)\,|\,\,f\in O(3;\underline{1})\big\}$ and any $\Z$-grading of $S$ is induced by an admissible grading of $O(3;\underline{1})$; see \cite[p.~276]{PS3}. Hence it may be assumed without loss of generality that there exists a triple $(a_1,a_2,a_3)\in\Z^3$ with $a_3=a_1+a_2$ such that
\begin{equation}\label{typeK}
\deg\big(D_K(x_1^{r_1}x_2^{r_2}x_3^{r_3})\big)=
(r_1+r_3-1)a_1+(r_2+r_3-1)a_2.\end{equation}
We first suppose that $a_2=0$. Since $S_{-1}\ne 0$
it follows from (\ref{typeK}) that $a_1=\pm 1$ and $S_0$ is spanned by $D_K(x_1x_2^i)$ and $D_K(x_2^ix_3)$ with $0\le i\le p-1$. Also, $$\textstyle{\sum}_{i<0}\, S_{ia_1}\,=\,S_{-a_1}\,=\,
{\rm span}\,\big\{D_K(x_2^i)\,|\,\,0\le i\le p-1\}.$$
If $a_1=1$ then $\dim S_{\ge 0}=120>(\dim \g)/2$. So this case cannot occur. Hence $a_1=-1$ implying $S_{\ge 0}=S_0\oplus S_1$. Then the proof of \cite[Lemma~4.12]{PS3} shows that
$S_0\,\cong\, W(1;\underline{1})\ltimes O(1;\underline{1})$
acts irreducibly and faithfully on $S_1\cong O(1;\underline{1})$. Set $u:=D_K(x_1)$ and $v:=D_K(x_2)$. It is straightforward to see that
$u\in S_0$ and $v$ spans $S_{-1}\cap \ker (\ad u)$.
Identify $S_1$ with $M_{(1)}={\rm nil}(M)$ and let $\tilde{u}$ be a preimage of $u$ in $M=M_{(0)}$.
The preceding remarks show that $\c_\g({\rm nil}(M))=M_{(1)}$. Applying Corollary~\ref{max-root}
with the abelian Lie algebra
$R=\Bbbk \tilde{u}\oplus\Bbbk v$ we now conclude that $v\in \OO_{\rm min}$. But $(\ad D_K(x_2))^4\ne 0$ and $(\ad e)^3=0$ for all $e\in\OO_{\rm min}$. This contradiction shows that $a_1 \ne 0$. Arguing similarly (with the roles of $x_1$ and $x_2$ interchanged) one obtains that $a_2\ne 0$.

If $a_1\ne |a_2|$ then none of $D_K(1)$, $D_K(x_1)$, $D_K(x_2)$, $D_K(x_1^2)$, $D_K(x_2^2)$ have degree $0$. 
As $D_K(x_3)$ and $D_K(x_1x_2)$ belong to the standard maximal subalgebra of $S$ it follows that $[S_0, S_0]$ acts nilpotently on $S$. The irreducibility of the $\G_0$-module $\G_{-1}=S_{-1}$ then forces $\dim \G_{-1}=1$ yielding $\G_k=0$\ for $k\ge 2$. But then $\dim \g=\dim S=125$, a contradiction. Therefore, $a_1=|a_2|$.

If $a_1=-a_2$ then arguing as in part~(d) of the
proof of \cite[Lemma~4.12]{PS3} one observes that
$S_0$ contains a nonzero ideal acting nilpotently on $S$. As this contradicts the faithfulness of the $S_0$-module $S_{-1}$, we conclude that $a_1=a_2$.
Then $a_1=\pm 1$ because $S_{-1}\ne 0$. If $a_1=1$ then
$\dim S_{\ge 0}=122>(\dim \g)/2$. As this is not the case we have that $a_1=a_2=-1$. Then $\G_2=S_2=\Bbbk D_K(1)$ and $\G_k=S_k=0$ for $k\ge 2$. It follows that $M$ normalises a line $\Bbbk e$ with $e\in\N(\g)$.
As $\n_\g(\Bbbk e)$ is contained in the optimal parabolic subalgebra of $e$ by \cite[Theorem~A]{P03}, we now deduce that the case where $S\cong K(3;\underline{1})$ is impossible.
\subsection{}\label{3.26} Finally, suppose that $S$ is isomorphic to the Melikian algebra $\mathcal{M}(1,1)$.
Then $p=5$ and $\dim S=125$ implying that the group $G$ has type ${\rm E}_7$. By \cite[7.1.4]{Str04}, all derivations of $S$ are inner. As $\dim \G=125<133=\dim\g$ it must be that $N(\G)\ne 0$.
Arguing as at the beginning of (\ref{3.24}) we observe that $\G=S=S^{(\infty)}$ acts trivially on the nonzero $\G$-module $N(\G)/N(\G)^2$.
As a consequence, $N(\G)$ contains  a graded ideal $I$ of $\G$ such that $N(\G)^2\subset I$ and $\dim(N(\G)/I)=1$.
Let $L=\G/I$. This Lie algebra is a central extension of $\mathcal{M}(1,1)$ and its centre $\z(L)=N(\G)/I$ is contained in $\G_{-k}/I_{-k}$ for some $k\ge 2$.
Since $\G_{-k}=\G_{-1}^k$ for all $k\ge 1$ it must be that $\z(L)\subset [L,L]$, i.e. the extension $$0\to\z(L)\to L\to \mathcal{M}(1,1)\to 0$$ is non-split. Since this contradicts \cite[Proposition~6.2]{PS6}
we now conclude that the present case case cannot occur.
This completes the proof of Theorem~\ref{thm:main}.
\section{Further remarks and observations}
\subsection{Non-existence of Hamiltonian subalgebras of $\g$}
We retain our assumption that $G$ is an exceptional algebraic $\Bbbk$-group and $p$ is a good prime for $G$.
It is proved in \cite{HS2} that $\g$ does not contain Lie subalgebras
$M$ isomorphic to $H(m;\underline{n};\Phi)^{(2)}$. Of course, the majority of pairs $(m;\underline{n})$ are ruled out by simple dimension arguments, but the most difficult case where $M\cong H(2;\underline{1};\Phi)^{(2)}$ is addressed in {\it loc.\,cit.} by using a description of the Witt subalgebras of $\g$.
Since this result is important for classifying all maximal subalgebras of $\g$, an alternative proof
is given below.
\begin{prop}[Herpel--Stewart]
The Lie algebra $\g$ does not contain Lie subalgebras isomorphic to $H(2;\underline{1};\Phi)^{(2)}$.
\end{prop}
\begin{proof} (a)
Suppose $S\cong H(2;\underline{1};\Phi)^{(2)}$ is a Lie subalgebra of $\g$  and no exceptional group $H$ of dimension $<\dim\g$ has the property that $H(2;\underline{1};\Phi)^{(2)}\hookrightarrow \Lie(H)$.
Let $\mathcal{S}$ be the $p$-envelope of $S$ in $\g$. If $\z(\mathcal{S})$ contains a semisimple element, $t$ say, then $S$ is contained in the Levi subalgebra $\g_t=\Lie(L)$ where $L=G_t$.
If $\z(\mathcal{S})$ contains a nonzero nilpotent element, $e$ say, then $S\subset \g_e$. Since $\g_e$ is contained in a proper parabolic subalgebra of $\g$, the Lie algebra $S$ again injects into the Lie algebra of a proper Levi subgroup $L$ of $G$.
In any event, the Lie algebra of one of the simple components of $L$ must contain an isomorphic copy of $S$. As this component must be classical by our assumption on $\g$,
it is straightforward to see that
$S$ affords a faithful representation of
dimension $<23\le p^2-2$. However, our discussion in (\ref{3.20}) shows that this is impossible.
Therefore, $\z(\mathcal{S})=0$ that is $\mathcal{S}\cong S_p$, the minimal $p$-envelope of $S$; see \cite[Corollary~1.1.8(2)]{Str04}. In particular, this implies that all composition factors of
the $S$-module $\g$ are restrictable.

(b) Since all composition factors of the $\mathcal{S}$-module $\g$ are restricted,
the arguments used in (\ref{3.17}) show that $\mathcal{S}$ contains a nonzero $p$-balanced toral element, say $h$. Then the centraliser $\g_h$ is listed in one of the seven cases of
Proposition~\ref{balanced}.
Let $S=S_{(-1)}\supset S_{(0)}\supset S_{(1)}\supset\cdots\supset S_{(q)}$ be the standard filtration of the Cartan type Lie algebra $S$ and denote by $O(2;\underline{1})[d]$ the subspace of all homogeneous
truncated polynomials of degree $d$ in $O(2;\underline{1})$. It is well known that $q\in \{2p-5,2p-4\}$ and $\dim S_{(i)}/S_{(i+1)}\ge \dim O(2;\underline{1})[i+2]$ for all $i\ge -1$. From this it is immediate that $S_{(p-1)}$ is an abelian subalgebra of $\g$ and
$$\dim S_{(p-1)}\,\ge\, \textstyle{\sum}_{i\ge p+1}\,O(2;\underline{1})[i]\,=\,(p-2)+(p-3)+\cdots+2\,=\,\textstyle{\frac{1}{2}}(p-1)(p-2)-1.$$
If $p=11$ then  $\dim S_{(p-1)}\ge 44$ and if $p=7$ then $\dim S_{(p-1)}\ge 14$. Applying \cite{Mal} and \cite{PeS} 
and arguing as in (\ref{3.21}) we now observe that cases (v) and (vi) of Proposition~\ref{balanced} are impossible. 

(c) Suppose $S\cong H(2;\underline{1})^{(2)}$. Some dimension estimates used in (\ref{3.20}) and (\ref{3.21}) are still applicable since they only rely on the structure of $S$ and properties of
its irreducible restricted representations. Furthermore, keeping in mind that Lie algebras of type ${\rm F}_4$ 
embed into Lie algebras of type ${\rm E}_6$ and arguing as at the end of (\ref{3.21}) we see that case (vii) of Proposition~\ref{balanced} cannot occur in the present situation. 

We adopt the notation introduced in (\ref{3.20}). In particular, we let $E_1,\ldots, E_r$ stand for the {\it induced} composition factors of the $\ad_\g S$-module $\g$ and write $l$ for the multiplicity of the adjoint module $S$ in $\g$. As in (\ref{3.20}) we put
 $$s:=l+\textstyle{\sum}_{i=1}^r\,(\dim E_i)/p^2.$$
 Recall that $S$ is closely related with
$O(2;\underline{1})$ regarded as a Lie algebra through its standard Poisson bracket $\{\,\cdot\,,\,\cdot\,\}$. More precisely, $\Bbbk 1\subset \{O(2;\underline{1}),O(2;\underline{1})\}$ is the centre of the Poisson algebra $O(2;\underline{1})$ and $S\cong\{O(2;\underline{1}),O(2;\underline{1})\}/\Bbbk 1$ as Lie algebras. The derived subalgebra
$\{O(2;\underline{1}),O(2;\underline{1})\}$ is spanned by all monomials $x_1^{m_1}x_2^{m_2}$ with $0\le m_1,m_2\le p-1$ and $m_1+m_2<2(p-1)$.
In other words, we may identify $S$ with the the $\Bbbk$-span of all $x_1^{m_1}x_2^{m_2}\in O(2;\underline{1})$, with $0<m_1+m_2<2(p-1)$ in such a way that
\begin{equation}\label{pois}
[x_1^{m_1}x_2^{m_2},x_1^{n_1}x_2^{n_2}]\,=\,
(m_1n_2-m_2n_1)x_{1}^{m_1+n_1-1}x_2^{m_2+n_2-1}\qquad\, (0\le m_i,n_i\le p-1,\,i=1,2).
\end{equation} Let $h=-2(1+x_1)x_2$, a toral element of $S$ not contained in $S_{(0)}$. Using (\ref{pois}) it is easy to see that all nonzero eigenvalues of $\ad h$ have multiplicity $p$. If $\rho\colon\,S\to \gl(E)$ is a restricted representation of $S$ induced from $S_{(0)}$, so that $E= u(S)\otimes_{u(S_{(0)})}\,E_0$ for some irreducible restricted $S_{(0)}$-module $E_0$, then  {\it all} eigenvalues of $\rho(h)$
have multiplicity $p\dim E_0$ because $h\not\in S_{(0)}$ and $h^{[p]}=h$. From this it is immediate that 
$h$ is a $p$-balanced element of $\g$ and $\dim\g(h,i)=ps$ for all $i\in\F_p^\times$.

(d) Let $V$ be the $2$-dimensional subspace of $S$ spanned by $x_1$ and $(1+x_1)^2x_2$. By construction, $V\cap S_{(0)}=0$ and $[h,v]=2v$ for all $v\in V$.
Analysing Kac coordinates of $p$-balanced $G$-orbits obtained in the course of proving Proposition~\ref{balanced}
one finds out that in cases (i), (ii) and (iv) there always exist a nilpotent element $e\in \OO({\rm A}_{p-1})$ and an optimal cocharacter $\tau\colon\,\Bbbk^\times\to G$ for $e$ such that $e\in \g(\tau,2)$ and $h=({\rm d}\tau)(1)$. In order to see this one just needs to compare
(\ref{e6}), (\ref{e7}) and (\ref{e8-7-2}) with \cite[pp.~402--407]{Car}. Moreover, in all three case
$\g(h,2)=\g\big(({\rm d}\tau)(1),2\big)=\g(\tau,2)\oplus \Bbbk n$ for some $n\in\OO_{\rm min}$. As $\dim V=2$ it must be that
$V\cap \g(\tau,2)\ne 0$. In view of \cite[Theorem~2.3]{P03} this implies that a nonzero element $v\in V$ is contained in the Zariski closure of $\OO({\rm A}_{p-1})$ in $\g$. Note that $v=\lambda x_1+\mu (1+x_1)^2x_2=v=u+v_0$ where $v_0 = 2\mu x_1x_2+x_1x_2^2\in S_{(0)}$ and $u=\lambda x_1+\mu x_2$ is a {\it nonzero} element in the linear span of
$x_1=\partial_2$ and $x_2=-\partial_1$.

As $\OO({\rm A}_{p-1})\subset\N_p(\g)$ we have that $v^{[p]}=0$. Since $v\not\in S_{(0)}$ it is straightforward to see that for any irreducible restricted representation $\rho\colon\,S\to \gl(E)$ such
that $E= u(S)\otimes_{u(S_{(0)})}\,E_0$ as $S$-modules,
the endomorphism $\rho(v)$ has $p\dim E_0$ Jordan blocks of size $p$.
On the other hand, using (\ref{pois}) one checks directly that
the ${\rm Im}(\ad_S u)^{p-1}$ has dimension $p-1$. A standard filtration argument then shows that $\dim\,{\rm Im}\,(\ad_S v)^{p-1}\ge \dim\,{\rm Im}\,(\ad_S u)^{p-1}$. Since $(\ad v)^p=0$ and $\dim S=p^2-2$ it follows that $\ad_S v$ has $p-1$ Jordan blocks of size $p$.   Consequently,
$\ad v$ has at least $s(p-1)$ Jordan blocks of size $p$.

Regarding $v=\lambda x_1+\mu (1+x_1)^2x_2$ as an element of the Poisson algebra $O(2;\underline{1})$ we observe that $v^k \in \{O(2;\underline{1}),O(2;\underline{1})\}$ for all $1\le k\le p-1$. We identify the $v, v^2,\ldots, v^{p-1}$ with their images in $S\,\cong\,\{O(2;\underline{1}),O(2;\underline{1})\}/\Bbbk 1$. Since $\{v,v^k\}=0$ and
$[h,v]=2v$ we have that
$[v,v^k]=0$ and $[h,v^k]=2k v^k$ for all $1\le k\le p-1$.
This implies that
$\g(h,i)\cap \g_v\ne 0$ for all $i\in\F_p^\times$.

Note that $v^{p-1}, u^{p-1}\in S_{(p-3)}$ and $v^{p-1}-u^{p-1}\in S_{(p-2)}$. Since the automorphism group ${\rm Aut}(S)$
preserves all components $S_{(i)}$ of the standard filtration of $S$, it is straightforward to see that $u^{p-1}$ lies in the Zariski closure of $\Bbbk^\times\big({\rm Aut}(S)\cdot v^{p-1}\big)$.
In view of Chevalley's Semicontinuity Theorem, this means that $\dim {\rm Im}\,(\ad_S v^{p-1})\ge \dim{\rm Im}\,(\ad_S u^{p-1})$. Since for every irreducible restricted representation $\rho\colon\,S\to \gl(E)$ induced from $S_{(0)}$ the vector space $E$ carries an ${\rm Aut}(S)$-module structure compatible with that of $S$, we also have, by the same token, that $\dim {\rm Im}\,\rho(v^{p-1})\ge \dim{\rm Im}\,\rho(u^{p-1})$. Since $u^{p-1}$ and $x_1^{p-1}$ are conjugate under the action of ${\rm Aut}(S)$ the images of $\ad_S(u^{p-1})$ and  $\ad_S(x_1^{p-1})$ have equal dimensions. The same applies to the images of $\rho(u^{p-1})$ and  $\rho(x_1^{p-1})$ thanks to the compatible action of ${\rm Aut}(S)$ on $E$.

(e) If $h$ is as in case~(i) of Proposition~\ref{balanced} then $p=5$, the group $G$ has type ${\rm E}_6$ and $s=3$. So $\ad v$ has
at least $12$ Jordan blocks of size $p$.
If $h$ is as in case~(ii) of Proposition~\ref{balanced} then $p=5$, the group $G$ has type ${\rm E}_7$ and $s=5$. Therefore, $\ad v$ has at least $20$ Jordan blocks of size $p$. If $h$ is as in case~(iii) of Proposition~\ref{balanced} then $p=7$, the group $G$ has type ${\rm E}_8$ and $s=5$. Therefore, $\ad v$ has at least $30$ Jordan blocks of size $p$.
Since  $(\Ad\, G)\,v\subseteq \overline{\OO({\rm A}_{p-1})}$ (and hence $\dim\, [\g,v]\le \dim\OO({\rm A}_{p-1})$) we combine \cite[Tables~6, 8 and 9]{Law}
with \cite[Theorem~4.1]{PSt} to conclude that either $v\in\OO({\rm A}_{p-1})$ or $G$ is of type ${\rm E}_8$ and $v\in\OO({\rm E}_8({\rm a}_7))$.

Suppose $v\in\OO({\rm E}_8({\rm a}_7))$.
Then $p=7$ and $v$ is distinguished in $\g$. So it follows from \cite[11.8]{Borel} that all maximal toral subalgebras of the normaliser $\n_\g(\Bbbk v)=\Lie\big(N_G(\Bbbk v)\big)$ are $1$-dimensional and conjugate under the adjoint action of $G_v$. By \cite[p.~157]{LT11}, there is a nonzero toral element $t\in \n_\g(\Bbbk v)$ such that $\g(t,2p-2)\cap\g_v=0$.
As $[h,v]=2v$ the above implies that $\g(h,i)\cap\g_v=0$ for some $i\in\F_p^\times$, a contradiction.

Now suppose  $v\in\OO({\rm A}_{p-1})$. Then it is immediate from \cite[pp.~87, 104, 158]{LT11} that the subspace $$\g_v\cap\g(h,2p-2)=\g_v\cap \g\big(({\rm d}\tau)(1),2p-2\big)=\g_v\cap\g(\tau,2p-2)$$ is contained in the Zariski closure of $\OO_{\rm min}$. Our earlier remarks then show that $v^{p-1}\in\OO_{\rm min}$. In particular, $\dim {\rm Im}\,
(\ad v^{p-1})<2(p-1)s$ in all three cases.
On the other hand, arguing as in (\ref{3.20}) and taking into account the last paragraph of part~(d) one observes that
$\dim {\rm Im}\,\ad(v^{p-1})\ge  2(p-1)s$. This contradiction shows that cases~(i), (ii) and (iv) of Proposition~\ref{balanced} cannot occur.

If $h$ is as in case~(iii) of Proposition~\ref{balanced}, then it follows from (\ref{e8-7-1}) that $\g(h,i)=0$ for some $i\in\F_p^\times$. However, $0\ne v^k\in \g(h,2k)$ for all $k\in\{1,\ldots, p-1\}$. This proves that $S\not\cong H(2;\underline{1})^{(2)}$.

(f) Suppose $S\cong H(2;\underline{1}; \Phi(\tau))^{(1)}$. The rule
$\{x_1,x_2\}_\omega\,=\,(1+x_1)(1+x_2)$
extends uniquely to a Poisson bracket on  $O(2;\underline{1})$ and $S$ can be identified with
$\big(O(2;\underline{1})/\Bbbk 1,\{\,\cdot\,,\,\cdot\,\}_\omega\big)$ as Lie algebras.
Moreover, $S_p={\rm Der}(S)=T\oplus S$ where
$T=\Bbbk (1+x_1)\del_1\oplus \Bbbk(1+x_2)\del_2$; see
\cite[p.~42]{Str09} (as before, we identify $S$ with $\ad S$). We have shown in part~(a)  that $S_p$ is isomorphic to the $p$-envelope of $S$ in $\g$.
It is well known that decomposing $S$ into weight spaces with respect to $T$ enables one to identify $S$ with the Block algebra ${\rm Bl}(\F_p\oplus\F_p)$ which is the $\Bbbk$-span of the symbols $v_\alpha$ for $\alpha\in (\F_p\oplus\F_p)\setminus\{0\}$ with Lie bracket given by
\begin{equation}\label{block}[v_\alpha,v_\beta]=(\alpha|\beta)v_{\alpha+\beta}\end{equation}
where $(\,\cdot\,|\,\cdot\,)$ is a non-degenerate skew-symmetric $\F_p$-valued bilinear form on $\F_p\oplus \F_p$; see \cite[Theorem~10.3.2]{Str09}. Here $\F_p\oplus\F_p$ is nothing but the dual space of $T^{\rm tor}=\F_p(1+x_1)\partial_1\oplus \F_p(1+x_2)\partial_2$ and $(\F_p\oplus\F_p)\setminus\{0\}$ identifies with the set of all weights of $T$ on $S$. By \cite[Lemma~4.6.4]{BW} or by the proof of \cite[Theorem~10.7.3]{Str09}, any non-trivial irreducible restricted $S_p$-module $E$ has $p^2-1$ nonzero $T$-weights, all of the same multiplicity (depending on $E$). From this it is immediate that
any nonzero toral element of $T$ is $p$-balanced in $\g$.

Pick linearly independent $\alpha,\beta \in\F_p\oplus \F_p$ and set
$V=\Bbbk v_\alpha\oplus \Bbbk v_\beta$, a $2$-dimensional subspace of $\g$ There is a toral element $h\in T$ with $\alpha(h)=\beta(h)=2$. This element is $p$-balanced in $\g$ by the preceding remark. By \cite[Theorem~10.3.2(5)]{Str09}, there exist linearly independent toral elements $t_\alpha,t_\beta\in T$ such that $v_\alpha^{[p]}=at_\alpha$ and $v_\beta^{[p]}=bt_\beta$
for some $a,b\in\Bbbk^\times$. Combining (\ref{block}) with Jacobson's formula for $p$-th powers one observes that
$$(\lambda v_\alpha+\mu v_\beta)^{[p]}-\lambda^pa t_\alpha-\mu^pbt_\beta\in\, \textstyle{\sum}_{\gamma\in (\F_p\alpha+\F_p\beta)\setminus\{0\}}\,\Bbbk v_\gamma$$
for all $\lambda,\mu\in \Bbbk$. This implies that
$V\cap\N_p(\g)=0$. On the other hand, it follows from (\ref{e6}), (\ref{e7}), (\ref{e8-7-1}), 
(\ref{e8-7-2}) and (\ref{f4-5-1})
that $\g(h,2)\cap \N_p(\g)$ contains a subspace of codimension $\le 1$ in $\g(h,2)$. Since $V\subset \g(h,2)$ is $2$-dimensional we reach a contradiction thereby proving that the case $S\cong H(2;\underline{1};\Phi(\tau))^{(1)}$ is impossible.

(g) Finally, suppose $S\cong H(2;\underline{1};\Phi(1))$.
It is well known that $S$ is isomorphic to an Albert--Zassenhaus algebra $L(\Gamma,\Theta)$ where $\Gamma$ is an additive subgroup of $\Bbbk$  and $\Theta\colon\, \Gamma\to \Bbbk$ is a group homomorphism. Recall that $L(\Gamma,\Theta)$ is spanned by the symbols $u_\alpha$ with $\alpha\in \Gamma$ and
 \begin{equation}\label{AZ}
[u_\alpha,u_\beta]=\big(\beta-\alpha+\alpha\Theta(\beta)-
 \beta\Theta(\alpha)\big)u_{\alpha+\beta}.
 \end{equation}
All isomorphism types of Albert--Zassenhaus algebras of a given dimension are determined in \cite{BIO}.
Since in the present case $\dim S=|\Gamma|=p^2$, we may assume further that $\Gamma=\F_{p^2}$, the set of all roots of $X^{p^2}-X=0$ in $\Bbbk$, and $\Theta={\rm Fr}$, the Frobenius automorphism of $\F_{p^2}$; see \cite[Corollary~5.3]{BIO}. The $[p]$-powers $u_\alpha^{[p]}=(\ad u_\alpha)^p$ with $\alpha\in \F_{p^2}$ span a self-centralising $2$-dimensional torus $T$
of $S_p\cong \Der(S)$ such that $T\cap S=\Bbbk u_0$. It has $p^2$ weights on $S$
each of multiplicity $1$; see \cite[Theorem~10.4.6]{Str09}. The corresponding weight spaces are nothing but $\Bbbk u_\beta$ with $\beta\in\F_{p^2}$.
Since $[u_0,u_\beta]=\beta u_\beta$ for all $\beta\in\F_{p^2}$ by (\ref{AZ}), it is straightforward to see that $u_0^{[p]}\not\in \Bbbk u_0$. Hence
$T=\Bbbk  u_0\oplus \Bbbk u_0^{[p]}$ and $T\cap S$ is not a restricted subalgebra of $S_p$. Applying
\cite[Lemma~4.8.1]{BW} we now deduce that any non-trivial irreducible restricted $S_p$-module $E$ has $p^2-1$ nonzero $T$-weights, all of the same multiplicity (depending on $E$). As before, this implies that
any nonzero toral element of $T$ is $p$-balanced in $\g$.
It follows from (\ref{AZ}) that
\begin{eqnarray}\label{bracket}
[u_\alpha,u_{\beta+k\alpha}]&=&\big(\beta+(k-1)\alpha+
\alpha(\beta^p+k\alpha^p)-(\beta+k\alpha)\alpha^p\big)
u_{\beta+(k+1)\alpha}\\
\nonumber&=&\big((\beta(1-\alpha^p)+\alpha(\beta^p+k-1)\big)
u_{\beta+(k+1)\alpha}.
\end{eqnarray}
for all $\alpha,\beta\in\F_{p^2}$ and all $k\in\F_p$.
Taking $\alpha=1$, the identity element of $\F_{p^2}$, we get
$$[u_1,u_{\beta+k}]=\big(\beta^p+k-1\big)u_{\beta+k+1}
\qquad\, (\forall\, \beta\in\F_{p^2}).$$
Therefore, $[u_1^{[p]},u_\beta]=\prod_{i\in\F_p}(\beta^p+i)\cdot u_\beta\,=\,(\beta^{p^2}-\beta^p)u_\beta=(\beta-\beta^p)u_\beta$ for all $\beta\in \F_{p^2}$ which yields $u_1^{[p]}=u_0-u_0^{[p]}$. This means that $u_1^{[p]}\ne 0$ and $(\ad u_1)^{p-1}(u_\beta)\ne 0$ whenever $\beta\not\in\F_p$.
Applying (\ref{bracket}) with $(\alpha,\beta)=(\beta,1)$ we get
$$[u_\beta,u_{1+k\beta}]=\big(1+k\beta-\beta^p\big)
u_{1+(k+1)\beta}\qquad\,(\forall\, \beta\in\F_{p^2}).$$
Note that $\Theta^2={\rm Fr}^2={\rm Id}$. If $\beta\in\F_{p^2}$ is such that $\beta^p=-\beta$
then $[u_\beta, u_{1+k\beta}]=\big(1+(k+1)\beta\big)u_{1+(k+1)\beta}$, implying
\begin{eqnarray*}
(\ad u_\beta)^{p-1}(u_1)
&=&\big(\textstyle{\prod}_{i\in\F_p^\times}(1+i\beta)\big)u_{1-\beta}=\beta^{p-1}
\big(\textstyle{\prod}_{i\in\F_p^\times}(\beta^{-1}+i)\big)u_{1-\beta}\\
&=&\beta^{p-1}(\beta^{1-p}-1)u_{1-\beta}
=(1-\beta^{p-1})u_{1-\beta}=\beta^{-1}(\beta-\beta^p)u_{i-\beta}=2u_{1-\beta}.
\end{eqnarray*}
It follows that $[u_\beta^{[p]},u_1]=2[u_\beta,u_{1-\beta}]=2u_1$. In particular, $u_\beta^{[p]}\ne 0$.

Let  $V$ be the $\Bbbk$-span of $u_1$ and $u_\beta$ where $\beta^p=-\beta$. As $[u_0,u_1]=[u_0^{[p]},u_1]=u_1$ and $[u_0,u_\beta]=-[u_0^{[p]},u_\beta]=\beta u_\beta$, the $T$-weights of $u_1$ and $u_\beta$ are
$\F_p$-independent. Therefore, they form a basis of $T^*$. As a consequence, there exists a nonzero toral element $h\in T$ such that $[h,v]=2v$ for all $v\in V$.
This element must be $p$-balanced in $\g$ by our earlier remarks.
Jacobson's formula for $p$-th powers shows that
$$(\lambda u_1+\mu u_\beta)^{[p]}-\lambda^{p-1}\mu (\ad u_1)^{p-1}(u_\beta)-\lambda\mu^{p-1}(\ad u_\beta)^{p-1}(u_1)\,\in T\oplus\textstyle{\sum}_{\alpha\not\in \{0,\pm(\beta-1)\}}\,\Bbbk u_\alpha
$$ for all $\lambda,\mu\in\Bbbk$. Since both
$(\ad u_1)^{p-1}(u_\beta)\in S_{\beta-1}$ and $(\ad u_\beta)^{p-1}(u_1)\in S_{1-\beta}$ are nonzero by our choice of $\beta$ we now deduce that $V\cap \N_p(\g)=0$
(one should also keep in mind here that $u_1^{[p]}\ne 0$ and $u_\beta^{[p]}\ne 0$).
At this point we can argue as at the end of part~(f) to conclude that $S\not\cong H(2;\underline{1};\Phi(1))$. This completes the proof.
\end{proof}
\subsection{Counterexamples to Morozov's theorem in bad characteristic: type ${\rm E}_8$}\label{4.2} Suppose $G$ is a group of type ${\rm E}_8$ and $p=5$. Although $5$ is a bad prime for $G$ there is a bijection between the nilpotent orbits in $\g=\Lie(G)$ and $\g_{\mathbb C}=\Lie(G_{\mathbb C})$ which preserves the orbit dimensions. In the notation of \cite[2.6]{P03}, each $G$-orbit $\OO\subset \N(\g)$ has the form $\OO=\OO(I,J)$ for a suitable pair $(I,J)\in\mathcal{P}(\Pi)$
 and by\cite[Theorem~1.4]{CP} each cocharacter $\lambda_{I,J}\in X_*(G)$ constructed in \cite[2.6]{P03} is still optimal in the sense of the Kempf--Rousseau theory for a nice representative
$e_{I,J}\in\OO(I,J)$. In particular, the stabiliser $G_{e_{I,J}}$ is contained in the parabolic subgroup $P(\lambda_{I,J})$.

Due to the above-mentioned bijection between the nilpotent orbits of $\g$ and $\g_{\mathbb C}$
each Hesselink stratum of $\N(\g)$ (i.e. each Lusztig's nilpotent piece of $\g$) is a single $G$-orbit; see \cite{CP} for details. This means that for any $(I,J)\in\mathcal{P}(\Pi)$ the orbit
$\big(\Ad\,P(\lambda_{I,J})\big)\cdot e_{I,J}$ is Zariski dense in $\bigoplus_{i\ge 2}\g(\lambda_{I,J},i)$ and $\Lie(G_{e_{I,J}})\subseteq \Lie(P(\lambda_{I,J})$.
Comparing the tables in \cite[pp.~90--93]{UGA05} and \cite[pp.~405--407]{Car} one observes that the elements $e\in \N(\g)$ for which $\Lie(G_e)\subsetneq \g_e$  lie in two orbits which have Dynkin labels
${\rm E}_8$ and ${\rm A}_4+{\rm A}_3$\footnote{The tables in \cite{UGA05} are known to contain a number of errors in characteristics $2$ and $3$. These errors are fixed in \cite{Ste} which provides a correct list of representatives of the nilpotent $G$-orbits in bad characteristics together with their Jordan block structure.}.
That $\OO({\rm E}_8)$ has this property was first pointed out in \cite[Theorem~5.9]{Spr}.
Note that $\OO({\rm A}_4+{\rm A}_3)\subset \N_p(\g)$ has dimension $200$ whilst \cite[p.~92]{UGA05} implies that $\dim\g_e=50$ for any
 $e\in \OO({\rm A}_4+{\rm A}_3)$.
It turns out that in characteristic $5$ the centraliser $\g_e$ has very unusual properties.

By \cite[p.~86]{UGA05}, the standard representative
$e=\sum_{\alpha\in\Pi\setminus \{\alpha_5\}}\,e_\alpha$   still lies in
$\OO({\rm A}_4+{\rm A}_3)$ which in conjunction with \cite[Theorem~5.2]{CP} shows that the cocharacter $\tau\in X_*(G)$ from
\cite[p.~148]{LT11} is still optimal for $e$.
Using \cite[p.~148]{LT11}
it is not hard to check that
$\g_e(\tau,0)$ is spanned by an $\sl_2$-triple
$\{e_0, h_0, f_0\}$ and $\g_e(\tau,1)$ is an irreducible module for $\g_e(\tau,0)\cong \sl_2$ of highest weight $3$.
In \cite[\S\,4]{LMT} (which also deals with $\OO({\rm A}_4+{\rm A}_3)$ in characteristic $5$) one finds two nonzero elements
$X,Y\in\g_e(\tau,-1)$ and checks directly that $[h_0,X]=X$ and $[h_0,Y]=-Y$.
On the other hand, the preceding remarks entail that $ \Lie(G_e)\subseteq \bigoplus_{i\ge 0}\,\g_e(\tau,i)$ has dimension $48$. As $\dim \g_e=50$ we now deduce that $$\g_e=\textstyle{\bigoplus}_{i\ge -1}\,\g_e(\tau,i),\quad \Lie(G_e)\,=\,\textstyle{\bigoplus}_{i\ge 0}\,\g_e(\tau, i),\quad
\dim \g_e(\tau,-1)=2.$$

Set $h:=({\rm d}\tau)(1)$ and $\n_e:=\n_\g(\Bbbk e)$. As $[h,e]=2e$ we have that $\n_e=\Bbbk h\oplus\g_e$.
As the torus $\tau(\Bbbk^\times)$ acts on $\g_e$ by Lie algebra automorphisms
the radical of $\g_e$ is a graded subspace of $\g_e=\bigoplus_{i\ge -1}\,\g_e(\tau,i)$. In particular, it is $(\ad h)$-stable.
At the author's request Thomas Purslow has checked the following by using some standard GAP routines:
\begin{itemize}
\item[(i)\,] the radical $A$ of $\g_e$ is abelian and has dimension $24$;

\smallskip

\item[(ii)\,] the Lie subalgebra $\g_e'$ of $\g_e$ generated by $\g(\tau, 1)$ and $\g(\tau, -1)$ has dimension $47$;

\smallskip

\item[(iii)\,] the normaliser $\mathfrak{w}:=\n_\g(A)$ has dimension $74$;

\smallskip

\item[(iv)\,] the Lie algebra $\mathfrak{w}/A$ is simple and restricted.
\end{itemize}
\begin{theorem}\label{frakw} The following are true for any $e\in \OO({\rm A}_4+{\rm A}_3)$:
\begin{itemize}
\item[(1)\,] $A\subset \g_e'$ and $\g_e'/A\,\cong\, H(2;\underline{1})^{(2)}$ as Lie algebras.

\smallskip

\item[(2)\,] $A={\rm rad}(\g_e)$ and $\g_e/A\,\cong\, H(2;\underline{1})$ as Lie algebras.

\smallskip

\item[(3)\,] $A={\rm rad}(\n_e)$ and $\n_e/A\,\cong\, {\rm Der}\big(H(2;\underline{1})^{(2)}\big)$ as Lie algebras.

\smallskip

\item[(4)\,] $A={\rm rad}(\mathfrak{w})$ and $\mathfrak{w}/A\,\cong\, W(2;\underline{1})$ as Lie algebras.

\smallskip

\item[(5)\,] $A\,\cong\,\big(O(2;\underline{1})/\Bbbk 1\big)^*$ as $W(2;\underline{1})$-modules.
\smallskip

\item[(6)\,] $A\subset \N(\g)$ and $\mathfrak{w}$ is a maximal Lie subalgebra of $\g$.
\end{itemize}
\end{theorem}
\begin{proof}
(a) Using \cite[p.~148]{LT11} and our earlier remark that $\g_e(\tau,0)\cong\mathfrak{sl}_2$
it is straightforward to see that $\g_e(\tau,1)$ is an
irreducible $4$-dimensional $\g_e(\tau,0)$-module.
If $[\g_e(\tau,-1),\g_e(\tau,1)]=0$ then the
Engel--Jacobson theorem implies that
$\sum_{i\ne 0}\,\g_e(\tau,i)$ is a nilpotent
ideal of codimension $3$ in $\g_e$. As this contradicts
(i) it must be that $[\g_e(\tau,-1),\g_e(\tau,1)]=\g_e(\tau,0)$ and $A\subseteq \bigoplus_{i\ge 2}\,\g_e(\tau,i)$.
Let $L=\g_e/A$. As $A$ is a graded subspace of $\g_e$ we have that $L=\bigoplus_{i\ge -1}\,L_i$ where $L_i=\g_e(\tau,i)/A\cap\g_e(\tau,i)$. Also, $L_0\cong \sl_2$, and $L_{-1}$ is an irreducible $L_0$-module. Since $L$ is semisimple it satisfies the conditions of the Weak Recognition Theorem; see \cite[Theorem~2.66]{BGP}. As $L$ is a restricted Lie algebra, $L_0\cong \sl_2$ and $L_1\not\cong L_{-1}^*$ applying that theorem shows that
$L$ is sandwiched between $H(2;\underline{1})^{(2)}$
and $H(2;\underline{1})$. As a consequence,
the Lie
subalgebra $L'$ generated by $L_{\pm 1}$ is isomorphic to $H(2;\underline{1})^{(2)}$, so that $\dim L'=p^2-2=23$.
In conjunction with (i) and (ii) this gives $A\subset \g_e'$ proving (1).

(b) As $L\subseteq H(2;\underline{1})$ and
$\dim H(2;\underline{1})=p^2+1=26=\dim\g_e-\dim A=\dim L$ by \cite[2.10]{BGP} and (i), the equality must hold, i.e. $L=H(2;\underline{1})$. This proves (2).
Since $\n_e=\Bbbk h\oplus\g_e$ and $A$ is $(\ad h)$-stable, we have that $A={\rm rad}(\n_e)$.
In this situation \cite[Theorem~2.66]{BGP} applies to the graded Lie algebra $\n_e/A$ forcing $$H(2;\underline{1})\cong L\subsetneq \n_e/A\subseteq CH(2;\underline{1}).$$ As $CH(2;\underline{1})\cong \Der\big(H(2;\underline{1})^{(2)}\big)$ by \cite[Theorem~7.1.2(2)]{Str09}, statement~(3) follows.

(c)As $L_{-i}=0$ for $i>1$
it also follows from \cite[Theorem~2.66]{BGP} that the grading of the Cartan type Lie algebra $L=H(2;\underline{1})$ is standard. In particular,
 $L_{k}=0$ for $k\ge 2p-4=6$. In view of \cite[p.~148]{LT11} 
 this means that $A$ contains $\g_e(\tau, 9)$, a $2$-dimensional irreducible $\g_e(\tau, 0)$-module
contained in the centre of the Lie algebra $\bigoplus_{i>0}\,\g_e(\tau,i)$. Since $\Bbbk e$ is a trivial submodule of the $L$-module $A$ and $A/\Bbbk e$ has dimension $23=p^2-2$, we can apply \cite[Corollary~510]{RH} to conclude that $A/\Bbbk e\cong L'$ as $L$-modules.

Since no Lie algebras of the form $\psl_{kp}$ or
$\Lie(\mathcal{H})$, where $\mathcal H$ is a simple algebraic $\Bbbk$-group, have  dimension $50=\dim\mathfrak{w}/A$, the Classification Theorem
from \cite{PS6} implies that $\mathfrak{w}/A$ is isomorphic to a restricted Lie algebra of Cartan type. As $\dim W(m;\underline{1})=m5^{m}$ for $m\ge 1$,
$\dim S(m;\underline{1})^{(1)}=(m-1)(5^m-1)$ for $m\ge 3$, $\dim H(2m;\underline{1})^{(2)}=5^{2m}-2$ for $m\ge 1$, and $\dim K(2m+1,\underline{1})\ge 5^{2m+1}-1$ for $m\ge 1$, we have only one option, namely, $\mathfrak{w}/A\cong W(2;\underline{1})$.
This proves (4).

(d) Since $A$ is a graded subspace of $\g_e$ and $A\cap\g_e(\tau,i)=0$ for $i=\{-1,0,1\}$ by our remarks in part~(a) we have that $A\subseteq \bigoplus_{i\ge 2}\,\g_e(\tau, i)$. In particular, $A\subset N(\g)$.
Since $\mathfrak{w}$ is $(\Ad\,\tau(\Bbbk^\times))$-stable and
$A(\tau,i)=0$ for $i\le 1$, it must be that
$[\mathfrak{w}(\tau,k),e]=0$ for $k\le -1$. This forces
$\mathfrak{w}=
\g_e(\tau,-1)\oplus \big(\bigoplus_{i\ge 0}\mathfrak{w}(\tau,i)\big).$ As $W(2;\underline{1})$ has a unique subalgebra of codimension $2$ by Kreknin's theorem, we  now obtain that the image of $\bigoplus_{i\ge 0}\mathfrak{w}(\tau,i)$ in $\mathfrak{w}/A$ coincides with the standard maximal subalgebra of $\mathfrak{w}/A\cong W(2;\underline{1})$. Consequently,  $\mathfrak{w}(\tau, 0)=\n_e(\tau, 0)\cong \gl_2$.

Since $A$ is a non-trivial restricted $W(2;\underline{1})$-module of dimension $p^2-1$, it follows from
\cite[Theorem~4.2]{RH1} that either $A\cong O(2;\underline{1})/\Bbbk 1$ or $A\cong \big(O(2;\underline{1})/\Bbbk 1\big)^*$ as $W(2;\underline{1})$-modules.
In the former case the annihilator $A^{\g_e(\tau,-1)}$ of $\g_e(\tau,-1)\cong
\Bbbk \partial_1\oplus\Bbbk\partial_2$ in $A$ is a
$2$-dimensional irreducible module over $\g_e(\tau,0)\cong\sl_2$ which generates an irreducible $L$-submodule isomorphic to $L'$. Since $\Bbbk e$ is
a trivial $L$-submodule of $A$, this would imply
that $\dim A^{\g_e(\tau,-1)}\ge 3$ contrary to the fact that $\big(O(2;\underline{1})/\Bbbk 1\big)^{\partial_1}\cap
\big(O(2;\underline{1})/\Bbbk 1\big)^{\partial_2}$ is $2$-dimensional. So $A\cong \big(O(2;\underline{1})/\Bbbk\big)^*$ proving (5).

(e) It remains to show that $\mathfrak{w}$ is a maximal subalgebra of $\g$. Suppose this is not the case and let $\mathcal{L}$ be a proper restricted Lie subalgebra of $\g$ with $\mathfrak{w}\subsetneq \mathcal{L}$. If $I$ is a nonzero ideal of $\mathcal{L}$ then $I^A\ne 0$ and hence $A\subseteq I$ by the irreducibility of the $\mathfrak{w}$-module $A$. If ${\rm nil}(\mathcal{L})\ne 0$ then we can take $\mathfrak{z}({\rm nil}(\mathcal{L}))\ne 0$ for $I$ to conclude that
${\rm nil}(\mathcal{L})$ is a nilpotent ideal of $\g_e$. But then ${\rm nil}(\mathcal{L})=A$, by part~(a),
forcing $\mathcal{L}\subseteq \mathfrak{w}$. As this contradicts our choice of $\mathcal{L}$ we deduce that ${\rm nil}(\mathcal{L})=0$.

Suppose ${\rm rad}(\mathcal{L})\ne 0$. Then $\mathcal{L}$ contains a nonzero abelian restricted ideal. Since ${\rm nil}(\mathcal{L})=0$ this ideal must contain a nonzero $[p]$-semisimple element of $\g$, say $z$. As $(\ad_{\mathcal{L}}z)^2=0$ and ${\rm Ker}\,\ad_{\mathcal{L}} z={\rm Ker}\,(\ad_{\mathcal{L}} z)^2$ by the semisimplicity of $\ad_{\mathcal{L}} z$, it must be that $\z(\mathcal{L})\ne 0$.
Taking $I=\z(\mathcal{L})$ we then obtain
that $A\subseteq \z(\mathcal{L})$ which, in turn,
yields $\mathcal{L}\subseteq\g_e$, a contradiction.

Thus from now on we may assume that $\mathcal{L}$ is semisimple. If $\mathcal{L}$ contains two distict minimal ideals $I_1$ and $I_2$ then the above shows that
$A\subseteq I_1\cap I_2=0$, a contradiction. Therefore,
$\mathcal{L}$ contains a unique minimal ideal, $J$ say. In this situation  Block's theorem says that there exist a simple Lie algebra $S$ and a nonnegative integer $m$ such that
$J\cong S\otimes O(m;\underline{1})$ as Lie algebras
and $$S\otimes O(m;\underline{1})\subset \mathcal{L}\subseteq \big(\Der(S)\otimes O(m;\underline{1})\big)\rtimes \big({\rm Id}_S\otimes W(m;\underline{1})\big).$$ Furthermore, since $\mathcal{L}$ is restricted the image of $\mathcal{L}$ under the canonical projection $$\pi\colon\,\big(\Der(S)\otimes O(m;\underline{1})\big)\rtimes \big({\rm Id}_S\otimes W(m;\underline{1})\big)\rightarrow\, W(m;\underline{1})$$
is a transitive subalgebra of $W(m;\underline{1})$, i.e. does not preserve the maximal ideal of the commutative algebra $O(m;\underline{1})$; see \cite[Corollary~3.3.5]{Str04}. Since $\dim \g=248$,
$\dim S\ge 3$ and $\dim O(n;\underline{m})=5^m$
we have that $m\in\{0,1,2\}$.

(f) Suppose $m>0$ and let $\m$ be the maximal ideal of the local algebra $O(m;\underline{1})$. If $\pi(\mathfrak{w})=0$ then $\mathfrak{w}\subseteq \Der(S)\otimes O(m;\underline{1})$. In this case the abelian ideal $J_0:=S\otimes \m^{m(p-1)}$ of $J$ is an $(\ad \mathfrak{w})$-stable and $J_0^A$ is a nonzero $\mathfrak{w}$-submodule of $J_0$.
But then $J_0^A\subset \g_e\subset \mathfrak{w}$ is an abelian ideal of $\mathfrak{w}$ implying $J_0^A=A$. As $J_0^A\subseteq
\z\big(S\otimes \m\big)$ and $S\otimes \m={\rm nil}(J)$ this entails that $S\otimes \m$ is a nilpotent ideal of  $\g_e$ containing $A$. But then $S\otimes\m=A=J_0^A\subseteq S\otimes \m^{m(p-1)}$, a contradiction. We thus deduce that $\pi(\mathfrak{w})$ is a nonzero Lie subalgebra of $W(m;\underline{1})$.

Since $A\subset J$ by part~(e) and $W(2;\underline{1})\cong\mathfrak{w}/A$ is a simple Lie algebra,  the above discussion shows that $\pi$ identifies $W(2;\underline{1})$ with a nonzero Lie subalgebra of $W(m;\underline{1})$. The concluding remark in part~(e) now shows that $m=2$ and the restriction of $\pi$ to $\mathfrak{w}$ is surjective. Since in this situation $\dim S<10$ it is immediate from \cite[Theorem~1.1]{PS6} that $\Der(S)=\ad S\cong S$, so that $$\mathcal{L}=J+\mathfrak{w}\,\cong\big(S\otimes O(2;\underline{1})\big)\rtimes\big({\rm Id}_S\otimes \big(W(2;\underline{1})\big)\ \mbox{ and }\  J\cap\mathfrak{w}=A. $$
The algebra $O(2;\underline{1})$ is spanned by the monomials $x^{\bf a}:=x_1^{a_1}x_2^{a_2}$ with $0\le a_i\le p-1$.
As $e\in S\otimes O(2;\underline{1})$ we can
write $e=s_0\otimes 1+\sum_{|{\bf a}|\ge 1}s_{\bf a}\otimes x^{\bf a}$ for some $s_0,s_{\bf a}\in S$.
As $A$ is an irreducible $\mathfrak{w}$-module and $W(2;\underline{1})$ has no Lie subalgebras of codimension $1$, it follows from Kreknin's theorem that $\mathfrak{w}$ contains a unique Lie subalgebra of codimension $2$, namely, the inverse image of the standard maximal subalgebra $W(2;\underline{1})_{(0)}$
under the canonical homomorphism $\mathfrak{w}\twoheadrightarrow W(2;\underline{1})$; we call it $\mathfrak{w}_{(0)}$.
It is immediate from our discussion in part~(d) that $\mathfrak{w}=\g_e(\tau,-1)\oplus\mathfrak{w}_{(0)}$.
This yields that $\pi(\g_e(\tau,-1))$ contains elements of the form $\partial_1+u_1$ and $\partial_2+u_2$ for some $u_1,u_2\in W(2;\underline{1})_{(0)}$. As $e$ commutes with $\g_e(\tau,-1)$  this implies that $s_0\ne 0$.

Let $\mathbf{p-1}=(p-1,p-1)$. Since $s_0\otimes x^\mathbf{p-1}$ commutes with $e$ we have that $s_0\otimes x^\mathbf{p-1}\in \g_e\cap J\subseteq \mathfrak{w}\cap J=A$. Applying to $s_0\otimes x^\mathbf{\bf{p-1}}\ne 0$ the endomorphisms $\ad y$ with $y\in\g_e(\tau,-1)$ we observe that for any ${\bf a}=(a_1,a_2)$ with $0\le a_1,a_2\le p-1$ the $\mathfrak{w}$-module $A$ contains an element of the form
$s_0\otimes x^{\bf a}+\sum_{|{\bf b}|>|{\bf a}|}\,
t_{\bf b}\otimes x^{\bf b}$ for some $t_{\bf b}\in S$.
But then $p^2-1=\dim A\ge \dim O(2;\underline{1})=p^2$. This contradiction shows that $m=0$.

(g) From now on we may assume that any proper restricted Lie subalgebra $\mathcal{L}$ of $\g$ containing $\mathfrak{w}$ is sandwiched between $S\cong \ad S$ and $\Der(S)$ for some
simple Lie algebra $S$. It follows from \cite[Chapter~7]{Str04} and the Classification Theorem proved in \cite{PS6} that the Lie algebra of outer derivations of $S$ is solvable. Since the Lie algebra
$\mathfrak{w}$ is perfect we have that $\mathfrak{w}\subset S$. Therefore, no generality will be lost by assuming that $\mathfrak{w}$ is a maximal subalgebra of $S$.

Suppose $S\cong\psl_{5k}$ for some $k\ge 1$. Since $\dim S> 74$
by (iii) and $5k-2\,=\,{\rm TR}(S)\,\le\, {\rm TR}(\mathcal{L})\,\le\, {\rm TR}(\g)\,=\,8$ by
\cite{P87}, \cite{P90} and
\cite[Theorem~1.2.7]{Str04} it must be that $k=2$.
Since $\Der(S)\cong\mathfrak{pgl}_{10}$ by \cite[Lemma~2.7]{BGP}
and ${\rm TR}(\mathfrak{pgl}_{10})=9$ by \cite{P87} and \cite{P90}
the restricted Lie algebras $\mathcal{L}$ and $\psl_{10}$
are isomorphic (one should keep in mind here that $\mathcal{L}$ contains $\ad S$ which has
codimension $1$ in $\Der(S)$). As a consequence,
$\psl_{10}$ contains a maximal toral subalgebra of $\g$, say $\t$.
In view of \cite[Theorem~13.3]{Hum} or \cite[11.8]{Borel}
there is a maximal torus $T$ in $G$ with $\t=\Lie(T)$.
Since $p>3$, it follows from \cite[Ch.~II, \S\,3]{Sel}
that all root spaces of $\g$ with respect  to $\t$ are $1$-dimensional and hence  $(\Ad\,T)$-stable.
In conjunction with \cite[Ch.~II, \S\,4]{Sel} this implies the set $\Phi_0:=\{\gamma\in\Phi(G,T)\,|\, \mathcal{L}_\gamma\ne 0\}$ is a root subsystem of the root system $\Phi=\Phi(G,T)$.
As $|\Phi_0|=90$, it is immediate from the Borel--de Siebenthal theorem that
$\Phi_0$ is contained in a maximal root subsystem  of type
${\rm E}_7+{\rm A}_1$ or ${\rm D}_8$ in $\Phi$. The simplicity of the restricted Lie algebra $\mathcal{L}$ then entails that it embeds into a restricted Lie algebra $\g_0$ of type ${\rm E}_7$ or ${\rm D}_8$. The first possibility cannot occur as Lie algebras of type ${\rm E}_7$ do not contain $8$-dimensional toral subalgebras. Hence $\g_0\cong\so(V)$ where $V$ is a $16$-dimensional vector space over $\Bbbk$. Let $\rho\colon\,\mathcal{L}\to \gl(V)$ be the representation of $\mathcal{L}$ induced by inclusion $\mathcal{L}\subset \so(V)$. As $\dim \mathcal{L}>74>60=(\dim \g_0)/2$, the trace form
$(X,Y)\mapsto {\rm trace}\big((\rho(X)\circ \rho(Y)\big)$ on $\mathcal{L}$
is nonzero and hence non-degenerate by the simplicity of $\mathcal{L}\cong\psl_{10}$. Since this contradicts \cite{Bl} we now conclude that $S\not\cong\psl_{5k}$.

Suppose $S\cong \Lie(\mathcal{H})$ for some simple algebraic $\Bbbk$-group $\mathcal{H}$. Since $\dim S<248$ and $S$ is simple, $5$ is a very good prime for $\mathcal{H}$.
Theorem~\ref{thm:main} then says that $\mathfrak{w}=\Lie(\mathcal{P})$ for some maximal parabolic subgroup $\mathcal{P}$ of $\mathcal{H}$.
In view of our discussion at the beginning of part~(e) this implies that $A=\Lie\big(R_u(\mathcal{P})\big)$. Then
$\Lie(\mathcal{P})/\Lie\big(R_u(\mathcal{P})\big)\cong W(2;\underline{1})$ as Lie algebras. Since $\Lie(\mathcal{P})/\Lie\big(R_u(\mathcal{P})\big)$ is isomorphic to the Lie algebra of a reductive $\Bbbk$-group this is false.

Therefore, $S$ is a Lie algebra of Cartan type (not necessarily restricted).
Since $\g\cong\g^*$ as $(\Ad\,G)$-modules, the Lie algebra $\g$ admits a non-degenerate $\g$-invariant bilinear form, $b$ say. We stress that $b\ne \kappa$
as in the present case the Killing form of $\g$ is identically zero. The explicit formulae in \cite[p.~661]{CP} show that the form $b$ is symmetric.
If $\dim S>124$ then the restriction of $b$ to $S$ is nonzero and hence non-degenerate by the simplicity of $S$. Hence $\g=S\oplus E$ where $E=\{x\in\g\,|\,\,b(x,S)=0\}$. Since $b$ is $S$-invariant, $[S,E]\subseteq E$.
Since $\mathfrak{w}\subset S$ by our earlier remarks and $E^A\ne 0$ is $(\ad \mathfrak{w})$-stable, the irreducibility of the $\mathfrak{w}$-module $A$ yields
$A\subseteq E$. But then $A\subseteq S\cap E=0$, a contradiction.
So $74<\dim S\le 124$.
The results proved in \cite[\S\S\,6.3--6.7]{Str04} now imply that $S\cong H(2;(2,1);\Phi)^{(2)}$ where $\Phi\in \{{\rm Id},\Phi(\tau)\}$. In any event, the standard maximal subalgebra $S_{(0)}$ of $S$ is a restricted Lie subalgebra of $\Der(S)$ and $S_{(0)}/{\rm nil}(S_{(0)})\cong \sl_2$; see \cite[Theorems~7.2.2 and 6.3.10]{Str04}.

Since $\mathfrak{w}\subset S$ and $S_{(0)}$ has codimension $2$ in $S$, the Lie algebra $\mathfrak{w}\cap S_{(0)}$ has codimension $\le 2$ in $\mathfrak{w}$. Since $S_{(0)}$ has no homomorphic images isomorphic to $W(2;\underline{1})$ and $W(2;\underline{1})$ has no subalgebras of codimension $1$, our discussion in part~(f) shows that $\mathfrak{w}\cap S_{(0)}=\mathfrak{w}_{(0)}$. Note that ${\rm nil}(\mathfrak{w}_{(0)})$ coincides with the preimage of $W(2;\underline{1})_{(1)}$ under the canonical homomorphism $\mathfrak{w}_{(0)}\twoheadrightarrow W(2;\underline{1})_{(0)}$ and $\mathfrak{w}_{(0)}/{\rm nil}(\mathfrak{w}_{(0)})\cong \gl_2$.
By our remark at the beginning of this part, the
$p$-closure $\mathcal{S}$ of $S$ in $\g$ is semisimple. Therefore, $\mathcal{S}\cong S_p$ as restricted Lie algebras,  where $S_p$ denotes the $p$-envelope of $S$ in $\Der(S)$. It follows that $\mathfrak{w}_{(0)}$ is a {\it restricted} subalgebra of $S_{(0)}$.
But then the  restricted subalgebra
$$
\big(\mathfrak{w}_{(0)}+{\rm nil}(S_{(0)}\big)/{\rm nil}(S_{(0)}\,\cong\,\mathfrak{w}_{(0)}/\big(\mathfrak{w}_{(0)}\cap{\rm nil}(S_{(0)}\big)$$ of $S_{(0)}/{\rm nil}(S_{(0)})\cong \sl_2$ has a copy of $\gl_2$ as a homomorphic image. As this is obviously false, we conclude that $S$ does not exist. Thus $\mathfrak{w}$ is a maximal subalgebra of $\g$ and our proof is complete.
\end{proof}
\begin{conjecture}
Suppose $G$ is a group of type ${\rm E}_8$ and $p=5$. We conjecture that any maximal subalgebra $M$ of $\g$ with ${\rm rad}(M)\ne 0$ is either conjugate to
$\mathfrak{w}$ under the adjoint action of $G$ or has the form $M=\Lie(P)$ for some maximal parabolic subgroup $P$ of $G$ or coincides with the centraliser of a toral element $t$ of $\g$. In the latter case, $M=\Lie(G_t)$ and $G_t$ is a semisimple group of type ${\rm A}_4{\rm A_4}$.
\end{conjecture}
\begin{remark}
It is quite possible that the maximal subalgebra $\mathfrak{w}$ gives rise to a Weisfeiler filtration $\mathcal{F}$ of $\g$ such that the corresponding graded Lie algebra ${\rm gr}_\mathcal{F}(\g)$ is isomorphic to the special Cartan type Lie algebra $S(3;\underline{1})^{(1)}$. This guess is supported by numerology: $\dim S(3;\underline{1})^{(1)}=2(p^3-1)=248$, and the fact that
the Lie algebra $\mathcal{G}=S(3;\underline{1})^{(1)}$ does admit a $\Z$-grading
$\mathcal{G}=\bigoplus_{i\in\Z}\,\mathcal{G}_i$ with $\mathcal{G}_i=0$ for $i\ge 2$,
$\mathcal{G}_0\cong W(2;\underline{1})$ and $\mathcal{G}_1\cong\big(O(2;\underline{1})/\Bbbk 1\big)^*$ as $W(2;\underline{1})$-modules;  see \cite[p.~283]{PS3}. This grading did not feature in our proof of Theorem~\ref{thm:main} because the possibility that $S\cong S(3;\underline{1})^{(1)}$ was  ruled out in \ref{3.15} by dimension arguments.
\end{remark}
\subsection{Non-existence of Melikian subalgebras of $\g$}
We continue assuming that $G$ is a group of type ${\rm E}_8$ and $p=5$.
In this subsection we are going to show that no Melikian algebras $M=\mathcal{M}(m,n)$ can occur as subalgebras of $\g$. Since $\dim \mathcal{M}(m,n)=5^{m+n+1}$ and $\dim \g=248$, the Lie algebras $\mathcal{M}(m,n)$ with $m+n\ge 3$ are immediately ruled out by dimension reasons.

Suppose $M=\mathcal{M}(1,1)$ is a Lie subalgebra of $\g$
and let $\mathcal{M}$ be the $p$-closure of $M$ in $\g$.
The centre $\z(\mathcal{M})$ is an abelian restricted subalgebra of $\g$ and hence decomposes as  $\z(\mathcal{M})\,=\,\z(\mathcal{M})_s\oplus\z(\mathcal{M})_n$ where
$\z(\mathcal{M})_s$ and $\z(\mathcal{M})_n$ are the maximal torus and the nilradical of $\z(\mathcal{M})$, respectively. If $\z(\mathcal{M})_s\ne 0$ then
$\mathcal{M}$ is contained in the centraliser of a nonzero toral element of $\g$, say $t$. Since $t\in \Lie(T)$ for some maximal torus $T$ of $G$, by \cite[11.8]{Borel}, and $M$ is a simple Lie algebra of dimension $125$, it is immediate from the Borel--de Siebenthal theorem that $M$ embeds into a Lie algebra of type ${\rm E}_7$. If $\z(\mathcal{M})_n\ne 0$ then
$\mathcal{M}$ is contained in the centraliser of a nonzero element $e\in\N_p(\g)$. Since $\dim \g_e\ge 125$ we have that $\OO(e)\ne \OO({\rm A}_4+{\rm A}_3)$. Then our discussion at the beginning of (\ref{4.2})
shows that $\g_e$ is contained in a proper parabolic subalgebra $\p=\Lie(P)$ of $\g$. Since $P=L\cdot R_u(P)$ for some Levi subgroup $L$ and $\Lie(R_u(P))$ is nilpotent, the simple Lie algebra $M$ injects into $\l:=\Lie(L)$. Since $\dim L\ge 125$ the Lie algebra
$[\l,\l]$ must have type ${\rm E}_7$.

A Lie algebra $\g_0$ of type ${\rm E}_7$ admits a non-degenerate invariant symmetric bilinear form
$\eta\colon\,\g_0\times \g_0\to \Bbbk$. If $M\subset \g_0$ the the restriction of $\eta$ to $M$ is nonzero because $\dim M>(\dim \g_0)/2$. The simplicity of $M$ then shows that $\eta_{\vert\,M}$ is a non-degenerate. But then $\g_0=M\oplus M^\perp$ where $M^\perp$ is the orthogonal complement of $M$ in $\g_0$. As $\dim M^\perp=\dim\g_0-125=133-125=8$ we have that $\dim M>
64= \dim \gl(M^\perp)$. As a consequence, the simple Lie algebra $M$ must act trivially on $M^\perp$, that is $[M,M^\perp]=0$.
But then $[\g_0,M]=[M+M^\perp,M]\subseteq M$ which implies that $M$ is an ideal of $\g_0$. This contradiction shows that $\z(\mathcal{M})=0$. Since $M$ is a restricted Lie algebra, applying \cite[Corollary~1.1.8]{Str04} yields $\mathcal{M}\cong M_p\cong M$. Therefore, $\mathcal{M}=M$, i.e. $M$ is a restricted subalgebra of $\g$.

Let $b$ be the invariant symmetric bilinear form on $\g$ introduced at the end of the proof of Theorem~\ref{frakw}. Since $M$ is simple and $\dim M>(\dim \g)/2$ the restriction of $b$ to $M$ is non-degenerate. Therefore, $\g=M\oplus M^\perp$ where $M^\perp=\{x\in \g\,|\,\,b(x,M)=0\}$. As $[M,M^\perp]\subseteq M^\perp$, the preceding discussion shows that $M^\perp$ is a restricted $M$-module.
If all composition factors of the $M$-module $M^\perp$
are trivial then $[M,M^\perp]=[M^{(\infty)},M^\perp]=0$
implying that $M$ is an ideal of $\g$. As this contradicts the simplicity of $\g$ we may assume  that at least one composition factor of the $M$-module $M^\perp$ is non-trivial. In particular, $M^\perp$ is a faithful $M$-module.

It was first observed by Kuznetsov in \cite{Ku} that the restricted Melikian algebra $M$ admits a $(\Z/3\Z)$-grading $M=M_{\bar{0}}\oplus M_{\bar{1}}\oplus M_{\bar{2}}$ such that $M_{\bar{0}}$ is a restricted Lie subalgebra of $M$ isomorphic to $W(2;\underline{1})$.
Given a $2$-dimensional torus $\t_0$  of $M_{\bar{0}}$ we can decompose the restricted $M_{\bar{0}}$-module
$M^\perp$  into weight spaces with respect to $\t$.
We denote by $\Gamma(M^\perp,\t_0)$ the set of all {\it nonzero} $\t$-weights of $M^\perp$.
Since the Lie algebra $M_{\bar{0}}$ is simple and $M^\perp$ is a faithful $M$-module, at least one composition factor of the $M_{\bar{0}}$-module $M^\perp$ is non-trivial.
Applying \cite[10.7.3]{Str09} we now deduce that $\Gamma(\t_0,M^\perp)=p^2-1=24$.
By \cite[Lemma~4.1]{P94}, there exists a $2$-dimensional torus $\t_0$ in $M_{\bar{0}}$ whose centraliser $H$ in $M$  is a $5$-dimensional Cartan subalgebra of $M$ with the property that $[H,[H,H]]=\t_0$. Each weight space
$M^\perp_\gamma$ with $\gamma\in\Gamma(\t_0,M^\perp)$ is invariant under the adjoint action of $H$. Since
$\gamma\in\t_0^*\setminus\{0\}$ and $\t_0\subseteq [H,H]$, the $H$-module  $M^\perp_\gamma$ has no $1$-dimensional composition factors. Representation theory of nilpotent Lie algebras then yields that each $\t_0$-weight space $M_\gamma^\perp$ with $\gamma\in\Gamma(\t_0,M^\perp)$ has dimension divisible by $p=5$. As a consequence, $\sum_{\gamma\in \Gamma(\t_0,M^\perp)}\,\dim M_\gamma^\perp\ge p(p^2-1)=120.$ Since $\dim M^\perp=240-125=123$ and $\dim \c_M(\t_0)=\dim H=5$ this gives $$\dim\c_\g(\t_0)=\dim\c_M(\t_0)+\dim \c_{M^\perp}(\t_0)\le 5+3=8.$$
Since $\t_0$ is a toral subalgebra of $\g$ and the field $\Bbbk$ is infinite we can find an element $t\in \t_0$ such that
$\c_\g(\t_0)=\g_t$. By \cite[11.8]{Borel}, there exists a maximal torus $T$ of $G$ such that $t\in \Lie(T)$. Since $\Lie(T)$ is abelian and has dimension $8$ it must be that $\c_\g(\t_0)=\g_t=\Lie(T)$. Since $H\subset \c_\g(\t_0)$ is non-abelian, this is impossible. We have reached a contradiction thereby proving that $\mathcal{M}(1,1)$ is not isomorphic to a Lie subalgebra of $\g$.

\subsection{Counterexamples to Morozov's theorem in bad characteristic: type ${\rm G}_2$} In this subsection we assume that $G$
is a group of type ${\rm G}_2$ and $p=2$. It was first observed by Robert Steinberg in \cite[2.6]{St}
that $\g\cong\psl_4$. This curious fact is explained in \cite[Sect.~4]{CE} as follows: in characteristic $2$, the irreducible $G$-module $E=L(\varpi_1)$ of highest weight $\varpi_1$
has dimension $6$ and carries a nonzero $G$-invariant symplectic form.  This gives us an inclusion $\g\subseteq \sp(E)$. Since $\g$ is a simple Lie algebra it is contained in the second derived subalgebra of $\sp(E)$ which is isomorphic to $\psl_4$. Hence $\g\cong \psl_4$ by dimension reasons.

Set $V:=O(2;\underline{1})$, a $4$-dimensional vector space over $\Bbbk$, and regard $V$ as the left regular $O(2;\underline{1})$-module. The Lie algebra $W(2;\underline{1})$ acts on $V$ by derivations. Let $\hat{\mathcal{W}}$ denote the semidirect product $W(2;\underline{1})\ltimes O(2;\underline{1})$ where $O(2;\underline{1})$ is regarded as an abelian ideal of $\hat{\mathcal{W}}$ acted upon by  $W(2;\underline{1})$ in the natural fashion. The above-mentioned actions of $O(2;\underline{1})$ and $W(2;\underline{1})$ are compatible in the sense that they give rise to a faithful representation $\rho\colon\,\hat{\mathcal{W}}\to \gl(V)$. It is straightforward to see that the representation $\rho$ is irreducible.

In characteristic $2$, the standard grading of $L:=W(2;\underline{1})$ is surprisingly short, namely, $L=L_{-1}\oplus L_0\oplus L_1$, and we have that $L_0\cong \gl_2$ and $[L_{-1},L_{1}]=L_0$. Furthermore, $L_1\cong L_{-1}^*$ is an irreducible $2$-dimensional
$L_0$-module. From this it is immediate that $L\cong \sl_3$ as Lie algebras (this can also be deduced from the fact that $L$ acts faithfully on the $3$-dimensional vector space $O(2;\underline{1})/\Bbbk 1$). In particular, $L$ is simple.
Let $\m$ be the maximal ideal of $O(2;\underline{1})$, so that $O(2;\underline{1})=\Bbbk 1\oplus \m$. Obviously, $\rho(1)={\rm Id}_V$ and $\rho(\m)^2=0$. Since the trace of ${\rm Id}_V$ is zero and $W(2;\underline{1})=W(2;\underline{1})^{(1)}$ by the preceding remark, $\rho(\hat{\mathcal{W}})\subset \sl(V)$.
Note that $\Bbbk 1$ coincides with the centre of the Lie algebra $\hat{\mathcal{W}}$. We now set $\mathcal{W}:=\hat{\mathcal{W}}/\Bbbk 1$, a semidirect product of $W(2;\underline{1})$ and  ${\rm nil}(\mathcal{W})=O(2;\underline{1})/\Bbbk 1$. By the above,
$\mathcal{W}$ is an $11$-dimensional Lie subalgebra of $\psl(V)\cong \g$.

We claim that $\mathcal{W}$ is a maximal subalgebra of
$\psl(V)$. Indeed, suppose the contrary. Then $\psl(V)$ has a proper subalgebra of codimension $\le 2$, say $\mathfrak{r}$. Taking the inverse image of $\mathfrak{r}$ in $\sl(V)$ we observe that $\sl(V)$ has a proper Lie subalgebra of codimension $\le 2$. Since the set of all $d$-dimensional subalgebras of $\sl(V)$ is Zariski closed in the Grassmannian ${\rm Gr}\big(d, \sl(V)\big)$ and is acted upon by a Borel subgroup $B$ of $\SL(V)$, it follows from the Borel fixed-point theorem that $\sl(V)$ contains a proper Lie subalgebra of codimension $\le 2$ normalised by $\Ad\,B$; we call it $\p$. Since the $\Bbbk$-span of ${\rm Id}_V$ is the only proper nonzero ideal of $\sl(V)$, the Lie subalgebra $\p+\Lie(B)$ of $\sl(V)$ is proper and $(\Ad\, B)$-stable. From this it is immediate that $\SL(V)$ contains a proper parabolic subgroup of codimension $\le 2$. Since $\dim V=4$ we reach a contradiction thereby proving the claim.

If $\mathcal{W}$ is a parabolic subalgebra of $\g\cong\psl(V)$ then its preimage $\rho(\hat{\mathcal{W}})$ in $\sl(V)$ is a proper parabolic subalgebra of $\sl(V)$.
But then $V$ is a reducible $\rho(\hat{\mathcal{W}})$-module. As this contradicts the irreducibility of $\rho\colon\,\hat{\mathcal{W}}\to \gl(V)$ we conclude
that the subalgebra $\mathcal{W}$ of $\g$ is a counterexample to Morozov's theorem in type ${\rm G}_2$.

{\footnotesize
\bibliographystyle{amsalpha}

\end{document}